\documentclass[pdflatex,sn-mathphys-num]{sn-jnl}%
\usepackage[utf8]{inputenc}
\usepackage{mathtools}
\usepackage{quiver}
\usepackage{graphicx}
\usepackage[caption=false]{subfig} %
\usepackage{mathrsfs}
\usepackage{amsfonts}
\usepackage{amsthm}
\usepackage{amssymb}
\usepackage{amscd}
\usepackage{amstext}
\usepackage{blkarray}
\usepackage{amsmath}
\usepackage[all]{xy}
\xyoption{2cell}
\UseTwocells
\usepackage{stackengine}
\usepackage{calrsfs}
\usepackage{framed}
\usepackage{url}
\usepackage{float}
\usepackage{tikz}
\usetikzlibrary{positioning,intersections,cd,matrix, arrows, backgrounds}
\usetikzlibrary{shapes,shapes.geometric,shapes.misc}
\usetikzlibrary{decorations}
\usetikzlibrary{decorations.pathreplacing}
\newcommand\Ar[3]{\ar[from={#1}, to={#2}, #3]}
\newcommand\Nname[1]{|[alias=#1]|}
\usepackage[OT2,T1]{fontenc}
\usepackage{microtype}
\usepackage{xurl}
\usepackage{etoolbox}
\makeatletter
\@ifundefined{@makespecialcolbox}{%
  \@ifundefined{@make@specialcolbox}{}{%
    \let\@makespecialcolbox\@make@specialcolbox
  }%
}{}
\apptocmd{\ps@headings}{\let\@oddhead\@empty\let\@evenhead\@empty}{}
  {\PackageError{journal-Springer}{Could not patch the headings page style}{Check the sn-jnl class.}}
\apptocmd{\ps@titlepage}{\let\@oddhead\@empty\let\@evenhead\@empty}{}
  {\PackageError{journal-Springer}{Could not patch the title page style}{Check the sn-jnl class.}}
\makeatother
\usepackage{enumitem}
\usepackage{hyperref}
\newcommand{\refitem}[1]{\hyperref[#1]{#1}}
\usepackage{anyfontsize} %
\IfFormatAtLeastTF{2026-06-01}{%
  \newcommand{\newsharedtheorem}[2]{\newtheorem{#1}[thm]{#2}}%
}{%
  \usepackage{aliascnt}
  \newcommand{\newsharedtheorem}[2]{%
    \newaliascnt{#1}{thm}%
    \newtheorem{#1}[#1]{#2}%
    \aliascntresetthe{#1}}%
}
\usepackage[noabbrev,capitalise,nosort]{cleveref}
\usepackage{stmaryrd}%
\SetSymbolFont{stmry}{bold}{U}{stmry}{m}{n} %

\theoremstyle{plain}
\newtheorem{thm}{Theorem}[section]
\crefname{thm}{Theorem}{Theorems}

\newsharedtheorem{prp}{Proposition}
\crefname{prp}{Proposition}{Propositions}

\newsharedtheorem{lem}{Lemma}
\crefname{lem}{Lemma}{Lemmas}

\newsharedtheorem{cor}{Corollary}
\crefname{cor}{Corollary}{Corollaries}

\theoremstyle{definition}
\newsharedtheorem{dfn}{Definition}
\crefname{dfn}{Definition}{Definitions}

\newsharedtheorem{exm}{Example}
\crefname{exm}{Example}{Examples}

\newsharedtheorem{rmk}{Remark}
\crefname{rmk}{Remark}{Remarks}

\newsharedtheorem{ntn}{Notation}
\crefname{ntn}{Notation}{Notations}

\newsharedtheorem{cvn}{Convention}
\crefname{cvn}{Convention}{Conventions}
\crefname{subsection}{Subsection}{Subsections}

\newcommand{\thistheoremname}{}
\theoremstyle{plain}  %
\newtheorem*{genericthm*}{\thistheoremname}
\newenvironment{namedthm*}[1]
  {\renewcommand{\thistheoremname}{#1}
   \begin{genericthm*}}
  {\end{genericthm*}}

\newcommand\al{\alpha}
\newcommand\be{\beta}
\newcommand\ga{\gamma}
\newcommand\de{\delta}
\newcommand\ep{\varepsilon}
\newcommand\ze{\zeta}
\newcommand\et{\eta}
\renewcommand\th{\theta}%
\newcommand\io{\iota}
\newcommand\ka{\kappa}
\newcommand\la{\lambda}
\newcommand\ro{\rho}
\newcommand\si{\sigma}

\newcommand\ta{\tau}

\newcommand\ph{\phi}
\newcommand\ps{\psi}

\newcommand\om{\omega}

\newcommand\Th{\Theta}

\renewcommand\top{\operatorname{top}}%
\newcommand\soc{\operatorname{soc}}
\newcommand\Ker{\operatorname{Ker}}
\newcommand\Cok{\operatorname{Coker}}
\newcommand\cok{\operatorname{coker}}
\renewcommand\Im{\operatorname{Im}}%

\newcommand\Hom{\operatorname{Hom}}

\newcommand\rad{\operatorname{rad}}

\newcommand\End{\operatorname{End}}

\newcommand\rep{\operatorname{rep}}
\renewcommand\mod{\operatorname{mod}}%

\newcommand{\udim}{\operatorname{\underline{dim}}\nolimits}

\newcommand\prj{\operatorname{prj}}

\newcommand\rank{\operatorname{rank}}

\newcommand\calC{{\mathcal C}}

\newcommand\calF{{\mathcal F}}

\newcommand\calL{{\mathcal L}}

\newcommand\bbA{\mathbb{A}}
\newcommand\bbD{\mathbb{D}}
\newcommand\bbE{\mathbb{E}}
\newcommand\bbI{\mathbb{I}}

\newcommand\bbZ{\mathbb{Z}}

\newcommand\bbQ{\mathbb{Q}}
\newcommand\bbR{\mathbb{R}}

\newcommand{\bfg}{\boldsymbol{g}}
\newcommand{\bfzero}{\mathbf{0}}

\newcommand\op{^{\mathrm{op}}} %

\newcommand\inv{^{-1}}

\newcommand\iso{\cong}
\newcommand\ds{\oplus}

\newcommand\udl{\underline}
\newcommand\ovl{\overline}
\newcommand\Ds{\bigoplus}

\def\dsm#1,#2..#3{\bigoplus_{{#1}={#2}}^{#3}}
\def\sm#1,#2..#3{\sum_{{#1}={#2}}^{#3}}

\newcommand\id{1\kern-.25em{\text{{\rm l}}}}

\newcommand\isoto{\ \raise.8ex\hbox{$^{\sim}$}\kern-.7em\hbox{$\to$}\ }

\newcommand\ya[1]{\xrightarrow{#1}}
\newcommand\blank{\operatorname{-}}

\def\repr[#1;#2;#3;#4;#5]{
\left(
\begin{matrix}#1\\#2\end{matrix}
#3
\begin{matrix}#4\\#5\end{matrix}
\right)}

\newcommand\bmat[1]{\begin{bmatrix} #1 \end{bmatrix}}
\newcommand\smat[1]{\begin{smallmatrix} #1 \end{smallmatrix}}
\newcommand\sbmat[1]{\left[\begin{smallmatrix} #1 \end{smallmatrix}\right]}

\renewcommand{\theequation}{\thesection.\arabic{equation}}
\renewcommand{\thefigure}{\thesection.\arabic{figure}}

\newcommand\bfL{\mathbf{L}}
\newcommand\bfP{\mathbf{P}}

\newcommand\bfZ{\mathbf{Z}}

\newcommand\bfr{\mathbf{r}}

\newcommand\sfJ{\mathsf{J}}
\newcommand\sfP{\mathsf{P}}

\renewcommand\k{\Bbbk}

\newcommand\dom{\operatorname{dom}}
\newcommand\cod{\operatorname{cod}}
\newcommand\com{\mathrm{com}}

\newcommand{\thf}{\boldsymbol{\Th}}

\newcommand{\pE}{\mathcal{E}}

\newcommand{\Cov}{\operatorname{Cov}}
\newcommand{\Seg}{\operatorname{Seg}}
\newcommand{\Qvr}{\operatorname{Qvr}}

\newcommand{\src}{\operatorname{sc}}
\newcommand{\snk}{\operatorname{sk}}
\newcommand{\conv}{\operatorname{conv}}

\newcommand{\plex}{\preceq_{\mathrm{lex}}}

\newcommand{\bfM}{\mathbf{M}}

\newcommand{\hatbfM}{\hat{\mathbf{M}}}
\newcommand{\hatM}{\hat{M}}
\newcommand{\tdbfM}{\tilde{\mathbf{M}}}
\newcommand{\tdM}{\tilde{M}}
\newcommand{\checkbfM}{\check{\mathbf{M}}}

\newcommand{\br}[1]{\left\llbracket #1\right\rrbracket} %

\newcommand{\dset}{\mathord{\downarrow}} %
\newcommand{\uset}{\mathord{\uparrow}} %

\newcommand{\Tr}{\operatorname{Tr}}

\newcommand{\tot}{\mathrm{tot}}
\renewcommand{\ss}{\mathrm{ss}}
\newcommand{\zz}{\mathrm{zz}}
\newcommand{\cc}{\mathrm{cc}}

\newcommand{\Ac}{\operatorname{Ac}}
\newcommand{\Pzero}{\bmat{\sfP_{b_1,a_1} & \mathbf{0}\\\mathbf{0}&\mathbf{0}}}
\newcommand{\pr}{\operatorname{pr}}
\newcommand{\DDs}{\displaystyle\Ds}

\renewcommand{\vec}[1]{\smash{\ensurestackMath{\stackengine{1pt}{#1}{\scriptscriptstyle\sim}{U}{c}{F}{F}{S}}}
  \vphantom{#1}
}

\newcommand{\sub}{\mathrm{C}}
\newcommand{\bfa}{\mathbf{a}}
\newcommand{\bfb}{\mathbf{b}}

\newcommand{\free}{\operatorname{free}}

\newcommand{\ACQ}{\mathrm{ACQ}}
\newcommand{\FP}{\mathrm{FP}}

  \newcommand{\rotldots}{\mathrel{\rotatebox[origin=c]{135}{$\ldots$}}}

\DeclarePairedDelimiterX\setc[2]{\{}{\}}{\,#1 \;\delimsize\vert\; #2\,}

\begin{document}

\title[Interval Replacements of Persistence Modules]{Interval Replacements of Persistence Modules}

\author[1,2,3]{\fnm{Hideto} \sur{Asashiba}}\email{asashiba.hideto@shizuoka.ac.jp}

\author[4]{\fnm{Etienne} \sur{Gauthier}}\email{etienne.gauthier@inria.fr}

\author*[2]{\fnm{Enhao} \sur{Liu}}\email{liu.enhao.b93@kyoto-u.jp}

\affil[1]{\orgdiv{Department of Mathematics, Faculty of Science}, \orgname{Shizuoka University}, \orgaddress{\street{836 Ohya, Suruga-ku}, \city{Shizuoka}, \postcode{4228529}, \country{Japan}}}

\affil*[2]{\orgdiv{Institute for Advanced Study}, \orgname{Kyoto University}, \orgaddress{\street{Yoshida Ushinomiya-cho, Sakyo-ku}, \city{Kyoto}, \postcode{6068501}, \country{Japan}}}

\affil[3]{\orgdiv{Osaka Central Advanced Mathematical Institute}, \orgaddress{\street{3-3-138 Sugimoto, Sumiyoshi-ku}, \city{Osaka}, \postcode{5588585}, \country{Japan}}}

\affil[4]{\orgdiv{{\'E}cole Polytechnique}, \orgname{Institut Polytechnique de Paris}, \orgaddress{\street{Route de Saclay}, \city{Paris}, \postcode{91128}, \country{France}}}

\abstract{{Persistence modules indexed by finite posets arise in topological data analysis.
Unlike in the one-parameter case, they need not decompose into interval modules.
Interval replacements approximate them by formal differences of interval decomposable modules.
Extending earlier work on 2D-grids, we study these replacements for arbitrary compression systems.
Different choices can give different interval rank invariants and replacements, including those related to interval resolutions on commutative ladders.
Our main results are an explicit formula for interval ranks and an essential-cover theorem.
Using Auslander--Reiten theory, we express interval ranks in terms of ranks of matrices formed from the structure maps of the module.
The essential-cover theorem recovers each interval rank from a suitable restriction of the module.
The replacement construction preserves the corresponding interval rank invariant, which recovers the usual rank invariant on segments.
Essential covers also give a sufficient condition for two compression systems to yield the same invariant.
For total compression, the invariant coincides with the generalized rank invariant of Kim and M{\'e}moli.
On 2D-grids, essential covers recover the zigzag formula of Dey, Kim, and M{\'e}moli for this invariant.
}}



\maketitle

\section{Introduction}\label{sec1}
\enlargethispage{12pt}%

Persistent homology is one of the main tools used in topological data analysis (TDA), playing an important role in examining {topological properties of the data} \cite{MR1949898}.
{A {\em one-parameter} filtration\footnote{{Throughout this paper, a \emph{filtration} is defined to be a functor $\mathcal{F}$ from a poset, regarded as a category, to the category of topological spaces. {We say} that $\mathcal{F}$ is indexed by the poset.}} arising from the data yields a representation of a totally ordered set after applying homology to the filtration; this representation is commonly {referred to as a} {\em $1$-dimensional} {\em persistence module} in the literature~\cite{MR2121296,MR3333456,MR3868218}.}

{Many geometric models in persistent homology, such as the multicover modeling~\cite{sheehyMulticoverNerveGeometric2012} and the chromatic alpha complexes~\cite{dimontesanoChromaticAlphaComplexes2026}, involve complicated underlying posets of filtrations beyond totally ordered sets. {These constructions naturally extend} one-parameter persistence to a multi-parameter framework, leading to the concept of the multi-dimensional persistence module~\cite{carlssonTheoryMultidimensionalPersistence2009}. To be precise, the filtration is indexed by a $d$D-grid poset, defined as a product of $d$ totally ordered sets. {More generally}, persistence modules are understood as modules over the incidence category {of a poset}, or equivalently, functors from the poset (regarded as a category) to the category $\mod \k$ of finite-dimensional vector spaces over a field $\k$.}

However, {except in a few cases}, the category of $d$-dimensional persistence modules has infinitely many indecomposables up to isomorphism if $d > 1$ \cite{leszczynskiRepresentationTypeTensor1994,MR4091895}.
{Consequently, there is no finite collection of indecomposable modules whose multiplicities determine the isomorphism type of every $d$-dimensional persistence module. Hence, an approach that computes invariants by first enumerating all indecomposable modules and then determining their multiplicities is not feasible in general.}
Accordingly, defining meaningful and computationally feasible invariants of multi-dimensional persistence modules remains an active area of research.
To address the challenges mentioned above,
{we choose a finite subset of indecomposables and approximate a given persistence module using the chosen indecomposables.}
{Following the approach used in~\cite{MR4402576, ASASHIBA2023100007}, we choose as this subset the set of all interval modules, for the following two reasons. First, over a finite poset there are only finitely many interval modules up to isomorphism, and each of them is determined by its support, which is a connected and convex full subposet; they can be described combinatorially and enumerated. Second, for $1$-dimensional persistence modules the interval modules are exactly the indecomposables, and are displayed as the bars of a barcode; an approximation by interval modules therefore preserves the standard interpretation of a bar as the range of parameters in which a topological feature persists.}

In what follows, let $\bfP$ be a finite poset, and $I$ an interval subposet (namely, a connected and convex full subposet).
The set of all interval subposets is denoted by $\bbI$.
As mentioned above, we sometimes regard $\bfP$ as a category.

{As is well known}, there exists an isomorphism from the category of functors
$\bfP \to \mod \k$ to the category $\mod \k[\bfP]$ of $\k$-linear functors $\k[\bfP] \to \mod\k$ that is given by a $\k$-linearization,
where $\k[\bfP]$ is the \emph{incidence category} (\cref{dfn:inc-cat} (1) and (2)) of $\bfP$.
In this paper, we deal with the latter, and {call each of its objects a} \emph{persistence module} over $\bfP$ (or indexed by $\bfP$).
We denote by $V_I$ the interval module defined by $I$ (\cref{dfn:intv}).

{In \cite{ASASHIBA2023100007}, the \emph{interval replacement} $\de^*(M)$ of a persistence module $M$ over a 2D-grid was introduced as an element of the split Grothendieck group, represented by a pair of interval decomposable modules.}
The important points are that $M$ and $\de^*(M)$ share the same rank invariants (and hence also dimension vectors) for all $* = \tot, \ss, \cc$, {the total, source-sink, and corner-complete compressions, respectively}, and that the interval replacement
gives a way to examine the persistence module $M$ by using interval modules.

\enlargethispage{5pt}%
{Different choices of compression can give different interval rank invariants and interval replacements.
For commutative ladders, \cite[Section~5.1]{ASASHIBA2023107397} introduces a W-shaped zigzag compression whose interval rank invariant can differ from those under $\ss$ and $\cc$ \cite[Example~5.2]{ASASHIBA2023107397}.
The corresponding signed interval multiplicities can be computed from interval resolutions and give a homological invariant relative to the interval modules \cite[Corollaries~5.7 and~5.9]{ASASHIBA2023107397}.
This motivates studying how to compute and compare invariants under general compression systems.
}

\subsection{Purposes}
{In this paper, we study interval replacements over an arbitrary finite poset $\bfP$ and for any compression system $\xi$ (\cref{dfn:comp-sys-simplified-ver}). Our main results are an explicit formula for $I$-ranks (\textbf{Main result A}) and the essential-cover theorem (\textbf{Main result B}).}
{A compression system $\xi$ assigns to each interval $I$ an order-preserving map $\xi_I\colon I^\xi\to\bfP$ whose image is contained in $I$ and contains the set $\src(I)$ of minimal elements and the set $\snk(I)$ of maximal elements}\footnote{By $\min I$ (resp.\ $\max I$) we denote the minimum
(resp.\ the maximum) of $I$. Therefore, to distinguish minimal/maximal from minimum/maximum, we use the notation $\src/\snk$ for the former, which are short forms of source/sink.} of $I$. Then $\xi_I$ gives the restriction functor $R_I\coloneqq R^\xi_I \colon \mod \k[\bfP] \to \mod \k[I^{\xi}]$.
For example, the family $\tot$ of the inclusions $\tot_I \colon I \hookrightarrow \bfP$ for all intervals $I$ turns out to be a compression system, called the total compression system.

{The extension of the usual rank invariant~\cite{carlssonTheoryMultidimensionalPersistence2009} from segments to intervals is already given by the generalized rank invariant of Kim and M{\'e}moli~\cite{kim2021generalized}.
Botnan, Oppermann, and Oudot~\cite{botnan2022signed} study signed decompositions of this invariant, while \cite{ASASHIBA2023100007,ASASHIBA2023107397} develop interval invariants and replacements using compression on 2D-grids.
We now describe interval ranks under an arbitrary compression system over a finite poset.
Let $M\in\mod\k[\bfP]$.
We define the $I$-\emph{rank} $\rank_I^\xi M$ of $M$ under $\xi$ as the multiplicity of $R_I(V_I)$ in the indecomposable decomposition of $R_I(M)$ (\cref{dfn:int-rk}).
The family $\rank_\bbI^\xi M\coloneqq (\rank_I^\xi M)_{I\in\bbI}$ is called the \emph{interval rank invariant} of $M$ under $\xi$.
\cref{prp:seg-rank} shows that its restriction to segments is the usual rank invariant for every compression system.
In particular, $\rank_I^\tot M$ is called the \emph{total} $I$-rank of $M$.

Regarded as a function $\bbI\to\bbZ$, the interval rank invariant has a M{\"o}bius inversion $\de_M^\xi$, called the \emph{signed interval multiplicity} of $M$ under $\xi$ (\cref{dfn:delta}).
This defines the interval replacement $\de^\xi(M)$ (\cref{dfn:int-repl}).
Interval ranks extend additively to the split Grothendieck group (\cref{dfn:rank-de}), so we can compare the invariant of $M$ with that of its interval replacement.
By \cref{prp:IJ} and M{\"o}bius inversion, this construction satisfies $\rank_I^\xi\de^\xi(M)=\rank_I^\xi M$ for every $I\in\bbI$ (\cref{thm:rank}).
In particular, it preserves the usual rank invariant and dimension vector by \cref{prp:seg-rank}.}

{For a fixed compression system $\xi$, the following theorem expresses the $I$-rank of $M$ in terms of ranks of matrices formed from its structure linear maps (\cref{dfn:structure-linear-map}). This gives an explicit way to compute the invariant preserved by interval replacement.}

\begin{namedthm*}{{Main result A} (\cref{thm:general w/o conditions})}
\label{thm:intro2}
Let $\xi = (\xi_I \colon I^{\xi} \to \bfP)$ be a compression system for $\k[\bfP]$,
$M \in \mod \k[\bfP]$, and $I$ an interval of $\bfP$
with $\src(I^{\xi}){=} \{a_1,\dots, a_n\}$, $\snk(I^{\xi}){=} \{b_1,\dots, b_m\}$ $($elements are pairwise distinct$)$ for some $m,n \ge 1$. Obviously, for each $a \in \src(I^{\xi})$, there exists some $b \in \snk(I^{\xi})$ such that $a \le b$. Hence we may assume that $a_1 \le b_1$ without loss of generality.
Then  we have
\begin{equation}
    {\rank_I^\xi M} =
\rank \bmat{\tdbfM & \mathbf{0}\\
\bmat{M_{\xi_I(b_1),\xi_I(a_1)}&\mathbf{0}\\\mathbf{0}&\mathbf{0}} & \hatbfM\\
}
- \rank \tdbfM
-\rank \hatbfM,
\end{equation}
where $\tdbfM, \hatbfM$ are the matrices defined in 
\cref{thm:n-1-case w/o conditions,thm:1-n-case w/o conditions},
whose nonzero entries are given by structure linear maps 
$M_{b,a}\colon M(a) \to M(b)$ of $M$ corresponding to the unique morphism
from $a$ to $b$ in $\bfP$ {for all $a\le b$ in $\bfP$}.
If $m = 1$ $($resp.\ $n = 1)$, then
$\hatbfM$ $($resp.\ $\tdbfM)$ is an empty matrix, and hence the formula
has one of the special forms given in {\cref{prp:mult-for-1-1-case}} and {\cref{thm:n-1-case w/o conditions,thm:1-n-case w/o conditions}}.
\end{namedthm*}

{As the above result shows, once a persistence module $M$ is given, its interval ranks can be computed from its structure linear maps. In applications, $M$ is often obtained by applying homology to a filtration indexed by $\bfP$. An explicit construction of $M$ over the whole of $\bfP$ involves determining its vector spaces and structure maps and may be costly when $\bfP$ is large. The essential-cover technique introduced below shows that, for a fixed interval $I$, the $I$-rank can instead be recovered from a suitable restriction $R_{\om}(M)$ along an order-preserving map $\om\colon\bfZ\to\bfP$. When $M$ arises from a filtration $\calF$, this restriction is obtained by applying homology to $\calF\circ\om$, so the full module $M$ need not be constructed first.}
Roughly speaking, the essential cover ${\om}\colon \bfZ\to \bfP$ is an order-preserving map, and we say that ${\om}$ {\em essentially covers} an interval $I$ relative to a compression system $\xi$ if ${\om}(\bfZ)$ contains necessary morphisms in $\bfP$ for computing {the $I$-rank} of any $M\in \mod \k[\bfP]$ under $\xi$. We refer the reader to~\cref{sec:Essential cover relative to compression systems} and \cref{sec:form-add-hull} for a fuller treatment.

\begin{namedthm*}{{Main result B} (\cref{thm:ess-cov-int-rk-inv})}
\label{thm:intro3}
Let $\xi = \left(\xi_I \colon I^{\xi} \to \bfP\right)_{I\in \bbI}$ be a compression system. Fix an interval $I$ of $\bfP$ and let ${\om}\colon \bfZ\to \bfP$ be an order-preserving map that essentially covers $I$ relative to $\xi$. Then for every $M\in \mod \k[\bfP]$ we have
\begin{align*}
{\rank_I^\xi M} = \bar{d}_{R_{{\om}}(M)}(R_{{\om}}(V_I)),
\end{align*}
where $R_{{\om}}$ denotes the restriction functor induced by ${\om}$, and $\bar{d}_N(L)$ denotes the maximal number of copies of $L$ that can be taken as a direct summand of $N$. If $L$ is indecomposable, then $\bar{d}_N(L)$ is just the usual multiplicity of $L$ in $N$.
\end{namedthm*}

{For another compression system $\ka$, taking $\om=\ka_I\colon I^\ka\to\bfP$ in \textbf{Main result B} gives $\rank_I^\xi M=\rank_I^\ka M$ whenever $\ka_I$ essentially covers $I$ relative to $\xi$, since $R_I^\ka(V_I)$ is indecomposable.
The same equality holds if $\xi_I$ essentially covers $I$ relative to $\ka$, by interchanging $\xi$ and $\ka$.
Thus, if for every $I\in\bbI$ either condition holds, then $\rank_\bbI^\xi M=\rank_\bbI^\ka M$ for every $M\in\mod\k[\bfP]$ (\cref{cor:sufficient condition when two compression systems induce the same interval rank invariants}).
In particular, if $\xi_I$ essentially covers $I$ relative to $\tot$ for every $I\in\bbI$, then $\rank_\bbI^\xi M=\rank_\bbI^\tot M$ for every $M\in\mod\k[\bfP]$.}

We show some examples ({\cref{exm:int-rk_D_4_cases,exm:int-rk_D_4_cases_2nd}}) to demonstrate how the essential-cover technique is used for computing {interval ranks} under compression systems.

\subsection{Related works}
{In \cite{kim2021generalized}, Kim and M{\'e}moli introduced the generalized rank invariant for persistence modules over posets, using limits and colimits.
It coincides with our interval rank invariant under the total compression system (see \cref{exm:total}, \cref{rmk:concide-gri-int-rk}, and \cref{lem:gen-rk-inv}).
Our construction allows any compression system $\xi$, including the source-sink compression system (\cref{exm:ss}).
Although all compression systems recover the usual rank on segments, their interval rank invariants need not agree on arbitrary intervals.
We give a sufficient condition for two compression systems to induce the same invariant and use it to construct a zigzag compression system with the same invariant as the total compression system (\cref{cor:sufficient condition when two compression systems induce the same interval rank invariants,cor:tot-zz-case}).}

In \cite{botnan2022signed}, Botnan, Oppermann, and Oudot introduced {a general framework for signed barcode decompositions at the level of (generalized) rank invariants}. In detail, given a {locally finite} collection $\mathcal{I}$ of intervals of a poset and a map $r\colon \mathcal{I}\to \bbZ$ {with locally finite support}, {there exist unique disjoint multisets} $R$ and $S$ of elements of $\mathcal{I}$ such that $r$ {is the difference of the generalized rank invariants of the {interval decomposable} modules} $\Ds_{I\in R} V_{I}$ {and} $\Ds_{I\in S} V_{I}$ (see \cite[Corollary 2.5]{botnan2022signed}). From this result, one can obtain {a special case} of {the preservation property in \cref{thm:rank}}, that is, the persistence module and its {signed barcode decomposition} share the same generalized rank invariant once we let $\mathcal{I}=\bbI$ and take $r$ to be our interval rank invariant $\rank^\tot_{\bbI} M\colon \bbI \to \bbZ$. However, {\cref{thm:rank}} shows that the interval replacement preserves the interval rank invariant, {under any compression system, including the total compression system}. {Botnan, Oppermann, and Oudot also treat the general case without these local-finiteness assumptions: for an arbitrary collection $\mathcal{I}$ of intervals of an arbitrary poset, the minimal rank decomposition of a map $r\colon \mathcal{I}\to \bbZ$ is unique whenever it exists (see \cite[Corollary 2.10]{botnan2022signed}).} Compared with their results, {we study compression systems} and propose {the interval rank invariant of persistence modules under a compression system. The restriction of this invariant to segments coincides with the classical rank invariant}. {In our framework, the interval ranks we propose for a persistence module $M$ are defined as the multiplicities of interval modules appearing in the decomposition of its ``restriction''.} {This allows us to derive explicit formulas using Auslander--Reiten theory.}

{Dey, Kim, and M{\'e}moli \cite[Theorem 3.12]{deyComputingGeneralizedRank2024} reduce the computation of the generalized rank invariant to the zigzag path that concatenates the lower and upper zigzags of each interval (called the boundary cap in their terminology).} {In the same spirit, a closely related development is~\cite{deyComputingGeneralizedRanks2024}, which reduces the computation of generalized ranks to zigzag persistence and extends the underlying index poset from 2D-grids to finite posets via an unfolding technique.
{This reduction allows existing algorithms for zigzag persistence to be used.}}

{In comparison, our work has two main contributions.
First, we give explicit formulas for computing the interval rank invariant under any compression system from the structure linear maps of a persistence module (\textbf{Main result A}).
Second, for a fixed interval $I$ and compression system $\xi$, the essential-cover technique allows the $I$-rank $\rank_I^\xi M$ to be computed from a restricted module over another poset (\textbf{Main result B}).}
In some cases, the new poset can be chosen to be algorithmically tractable---e.g., a zigzag poset---so that fast algorithms are applicable. {For total compression on a 2D-grid, we construct a zigzag essential cover for each interval $I$; hence the total $I$-rank of any persistence module can be computed from its corresponding zigzag restriction. These covers form the compression system $\zz\coloneqq (\zz_I)_{I\in\bbI}$ (\cref{exm:zz}).} {By \cref{cor:tot-zz-case}, the interval rank invariants under $\tot$ and $\zz$ coincide. Together with the identification of the total interval rank with the generalized rank in \cref{rmk:concide-gri-int-rk}, this gives the zigzag formula in \cite[Theorem~3.12]{deyComputingGeneralizedRank2024}. The identification with the generalized rank} {follows from} \cite[Lemma 3.1]{chambersPersistentHomologyDirected2018}, but the description of the proof was imprecise; in formalizing it we found a minor gap, which we close by providing a complete proof (see the proof of \cref{lem:gen-rk-inv}).

In \cite{hiraokaRefinementIntervalApproximations2025}, Hiraoka, Nakashima, Obayashi, and Xu also established {a general theory for approximating persistence modules} over a finite fully commutative acyclic quiver by interval decomposable modules, {which is similar in spirit to ours}. They defined the so-called interval approximation (which, essentially, coincides with our interval replacement $\delta^{\xi}(M)$). For the sake of fast computation, {they define the partial interval approximation using a subcollection of the set $\bbI$ of all intervals} (which shares a similar idea of considering those intervals having ``good'' shapes in \cite{botnan2022signed}). For instance, they define the partial interval approximation restricted to the collection of $k$-essential intervals and estimate the computational complexity of (partial) interval approximation. {We also point out a difference in terminology. In \cite[Definition 3.37 and Example 3.38]{hiraokaRefinementIntervalApproximations2025}, a (partial) interval approximation is called ``rank invariant'' if the rank of $M$ along every path can be expressed as a weighted sum of the ranks of interval modules along that path, with weights given by the approximation. In our work, the term ``interval rank invariant'' denotes the family $(\rank_I^\xi M)_{I\in\bbI}$ under a compression system $\xi$.} {Moreover, {\cref{thm:rank}} extends~\cite[Theorem 3.30]{hiraokaRefinementIntervalApproximations2025}}, in the sense that forming the interval replacements preserves $I$-ranks not only for all segments $I$ but also for all intervals $I$. One of their main contributions is providing an efficient method to compute the indecomposable decomposition of persistent homology {indexed by a 2D-grid with 2 rows and 4 columns (called a \emph{commutative ladder} and denoted by ${\rm CL}(4)$).} By finding 76 linearly independent rank functions using zigzags of the grid and then solving {the resulting system of linear equations}, they achieve the desired decomposition without obtaining the representation (persistence module) of ${\rm CL}(4)$. See the list \cite{HNOXb} of zigzags they selected. Another main contribution in~\cite{hiraokaRefinementIntervalApproximations2025} is {the introduction of the connected persistence diagram, a new visualization of interval approximation in the commutative-ladder context (see~\cite[Definition 4.7]{hiraokaRefinementIntervalApproximations2025}).}

\subsection{Our contributions}
\begin{enumerate}
\item {We extend the compression and interval replacement constructions in \cite{ASASHIBA2023100007,ASASHIBA2023107397} to finite posets and arbitrary compression systems and show that interval replacement preserves the corresponding interval rank invariant (\cref{thm:rank}).} We follow the convention in \cite{ASASHIBA2023100007} to view the interval replacement of the persistence module as an element in the split Grothendieck group.

\item {We obtain the explicit interval-rank formulas in \textbf{Main result A} by applying \cite[Theorem~3]{Asashiba2017} together with \cref{lem:dim-Hom-coker}, which computes Hom dimensions from projective presentations.
The former gives multiplicities in terms of Hom dimensions; our formulas specify the block index sets and matrix entries directly from the poset and the structure maps of the persistence module.
For this, we construct projective presentations of the constant module $V_S$ on a finite poset $S$ (defined in \cref{subsec:presentations-finite-poset}) and of the other terms needed in the multiplicity formula: the quotient by the socle in the injective case, and the Auslander--Reiten translation and middle term of the almost split sequence in the noninjective case.
These constructions are stated with the appropriate hypotheses on $S$, independently of a compression system.
Taking $S=I^\xi$ gives the required Hom dimensions and hence the interval-rank formula.}

\item

{As a consequence of \textbf{Main result B}, we give a sufficient condition for two compression systems to induce the same interval rank invariant (\cref{cor:sufficient condition when two compression systems induce the same interval rank invariants}).}
{In the 2D-grid case, \cref{cor:tot-zz-case}, together with \cref{rmk:concide-gri-int-rk}, yields the zigzag formula in \cite[Theorem~3.12]{deyComputingGeneralizedRank2024}.}

\item {We provide MPIntInv, a program that computes the interval rank invariant and the interval replacement} under the total and source-sink compression systems of persistence modules over any $d$D-grid ($d\ge 2$). See \cref{rmk:code} for details.

\end{enumerate}

\subsection{Organization}

The paper is organized as follows.
\cref{sec2} is devoted to {collecting the necessary terminology}
and fundamental properties {for later use}, in particular,
incidence categories and incidence algebras defined by a finite poset, and the M{\"o}bius inversion.

{In \cref{sec3}, we introduce compression systems, justify and define interval ranks, and compare them with the generalized rank invariant.
In \cref{sec4}, we define signed interval multiplicities and interval replacement.
We then prove that interval replacement preserves the interval rank invariant (\cref{thm:rank}).}

{In \cref{sec5}, we give an explicit formula for the interval rank invariant under any compression system (\textbf{Main result A}).
For this, we construct projective presentations over a general finite poset $S$, including those for the Auslander--Reiten translation and the middle term of the almost split sequence when $V_S$ is not injective. We then apply these constructions with $S=I^\xi$.}

{In \cref{sec:Essential cover relative to compression systems}, we prove the essential-cover theorem (\textbf{Main result B}), which recovers interval ranks from restricted modules. As a consequence, we give a sufficient condition for two compression systems to induce the same invariants, including the case of total compression (\cref{cor:sufficient condition when two compression systems induce the same interval rank invariants}).}

Finally, in \cref{sec6}, we give some examples to show the incompleteness of the interval rank invariant. At the end, we demonstrate the use of interval replacement to distinguish different filtrations.

\section{Preliminaries}\label{sec2}

{We recall the representation-theoretic and combinatorial notions used in the paper. Incidence categories provide the setting for persistence modules and their decompositions, while M{\"o}bius inversion will be used to define signed interval multiplicities in \cref{sec4}.}

Throughout this paper, $\k$ {is a field;} {$\bbR$ (resp.\ $\bbQ$) denotes the real (resp.\ rational) field. $\bbZ$ denotes the ring of integers. The category of finite-dimensional $\k$-vector spaces is denoted by $\mod \k$. We let $\bfP = (\bfP, \le)$ denote a finite poset.}

For each positive integer $n$, we denote by $[n]$ the set $\{1, 2, \dots, n\}$ endowed with the usual linear order $i < i+1$ for $i = 1, 2, \dots, n-1$. Then $[n]$ becomes a totally ordered set. Posets of this type play an important role in one-parameter persistent homology. As another example of a poset, given two posets $\bfP_1$ and $\bfP_2$, we define their direct product $\bfP_1 \times \bfP_2$ to be the poset whose partial order is given by
$(x, y) \le (x', y')$ if and only if
$x \le x'$ and $y \le y'$ for all $(x, y), (x', y') \in \bfP_1 \times \bfP_2$. In particular, we set $G_{m,n}\coloneqq [m] \times [n]$, and call it a {\em 2D-grid} (with $n$ rows and $m$ columns). When $n$ is equal to 2, we further call $G_{m, 2}$ {a {\em commutative ladder}} and denote it by ${\rm CL}(m)$. More generally, for any $d\ge 2$ and positive integers $n_1,\dots,n_d$, we define the \emph{$d$D-grid}
\[
G_{n_1,\dots,n_d}\coloneqq [n_1]\times\cdots\times [n_d],
\]
endowed with the product order. Namely, $(x_1,\dots,x_d)\le (x'_1,\dots,x'_d)$ if and only if $(x_i\le x'_i)$ for all $i=1,\dots,d$.

\subsection{Incidence categories}

{We describe persistence modules as modules over incidence categories and introduce interval modules. The Krull--Schmidt theorem then gives the summand multiplicities from which interval ranks will be defined.}

\begin{dfn}
\label{dfn:left-right-mod}
A $\k$-linear category $\calC$ is said to be {\em finite}
if it has only finitely many
objects and for each pair $(x,y)$ of objects,
the Hom-space $\calC(x,y)$ is finite-dimensional.

$\k$-linear covariant functors $\calC \to \mod \k$
are called {\em left $\calC$-modules}.
{Together with natural transformations as morphisms, these functors form} a $\k$-linear category, which is denoted by $\mod \calC$.

Similarly, $\k$-linear contravariant functors $\calC \to \mod \k$ are called
\emph{right $\calC$-modules},
which are usually identified with covariant functors
$\calC\op \to \mod \k$.
The category of right $\calC$-modules is denoted by $\mod \calC\op$.

We denote by $D$ the usual $\k$-duality $\Hom_\k(\blank, \k)$,
which induces the duality functors
$\mod \calC \to \mod \calC\op$ and $\mod \calC\op \to \mod \calC$.
\end{dfn}

\begin{dfn}
\label{dfn:inc-cat}
The poset $\bfP$ is regarded as a category as follows.
The set $\bfP_0$ of objects is defined by $\bfP_0\coloneqq \bfP$.
For each pair $(x,y) \in \bfP \times \bfP$,
the set $\bfP(x,y)$ of morphisms from $x$ to $y$
is defined by $\bfP(x,y)\coloneqq \{p_{y,x}\}$ if $x \le y$,
and $\bfP(x,y)\coloneqq \varnothing$ otherwise, where
{$p_{y,x}$ denotes the unique morphism from $x$ to $y$.}
The composition is defined by $p_{z,y}p_{y,x} = p_{z,x}$
for all $x, y, z \in \bfP$ with $x \le y \le z$.
The identity $\id_x$ at an object $x \in \bfP$ is
given by $\id_x = p_{x,x}$.
\begin{enumerate}
\item 
The \emph{incidence category} $\k[\bfP]$ of $\bfP$ is 
defined as the $\k$-linearization of the category $\bfP$.
Namely, it is
a $\k$-linear category defined as follows.
The set of objects $\k[\bfP]_0$ {is equal to $\bfP$;}
for each pair $(x,y) \in \bfP \times \bfP$, the set of morphisms $\k[\bfP](x,y)$ is the vector space with basis $\bfP(x,y)$; thus it is a one-dimensional vector space $\k p_{y,x}$ if $x \le y$, or zero otherwise. The composition is defined as the $\k$-bilinear extension of that of $\bfP$. Note that $\k[\bfP]$ is a finite $\k$-linear category.

\item
Covariant ($\k$-linear) functors $\k[\bfP] \to \mod\k$ are called
\emph{persistence modules} (over $\bfP$ or indexed by $\bfP$). {{In particular}, when $\bfP$ is a totally ordered set {(resp.\ a $d$D-grid)}, persistence modules are referred to as \emph{$1$-dimensional} (resp.\ \emph{$d$-dimensional}) in the literature.}

\item

Let $M$ denote a persistence module over $\bfP$. Then the \emph{dimension vector of $M$}, denoted by $\udim M$, is a function
	\[
	\udim M\colon \bfP\to \bbZ,\, x\mapsto \dim M(x),
	\]
where $\dim M(x)$ denotes {the dimension of the vector space} $M(x)$. Note that $\bfP$ is finite. We call $\dim M\coloneqq \sum_{x\in \bfP}\dim M(x)$ the \emph{dimension of $M$}. If $\dim M$ is finite, then we say that $M$ is \emph{finite-dimensional}. Since each $M(x)$ is a finite-dimensional $\k$-vector space, persistence modules over finite posets defined in this paper {are always finite-dimensional}.

\end{enumerate}
In the sequel, 
we set $[\le]_\bfP\coloneqq \{(x,y) \in \bfP \times \bfP \mid x \le y\}$,
and $A\coloneqq \k[\bfP]$ (therefore, $A_0 = \bfP$),
and so the category of persistence modules over $\bfP$ is denoted by $\mod A$.
{The pair $(x,y)\in[\le]_\bfP$ indexes the morphism $p_{y,x}\colon x\to y$.
We write the target first in $p_{y,x}$, consistently with the composition rule above and the notation $M_{y,x}=M(p_{y,x})$ in \cref{dfn:structure-linear-map}.}
\end{dfn}

\begin{rmk}
The $\k$-linearization of posets is functorial. More concretely, if $f\colon \bfP_1\to \bfP_2$ is an order-preserving map, then $\k[f]\colon \k[\bfP_1]\to \k[\bfP_2]$ is a linear functor between incidence categories given by
\begin{equation}
\label{eq:k_linear_ext_ob}
	\k[f](x)\coloneqq f(x),
\end{equation}
for all $x\in \k[\bfP_1]_0$, and
\begin{equation}
\label{eq:k_linear_ext_mor}
	\k[f](0)\coloneqq0,\ \text{and}\ \k[f](kp_{y, x})\coloneqq kp_{f(y), f(x)}\ (k\in \k),
\end{equation}
where $p_{y,x}$ is {the canonical basis vector} of $\k[\bfP_1](x, y)$ whenever $(x, y)\in [\leq]_{\bfP_1}$, and $p_{f(y),f(x)}$ is {the canonical basis vector} of $\k[\bfP_2](f(x), f(y))$.
\end{rmk}

\begin{dfn}
\label{dfn:structure-linear-map}
Let $M\in \mod A$. Then {we call the linear maps} $M(p_{y,x})\colon M(x)\to M(y)$ the \emph{structure linear maps of $M$} for all $(x,y)\in [\leq]_{\bfP}$. To shorten the notation, $M(p_{y,x})$ is also written as $M_{y,x}$. Clearly, for all $x\in \bfP$, $M_{x,x}$ is {the identity map of} $M(x)$.
\end{dfn}

\begin{dfn}
\label{dfn:ss-intervals}
Let $I$ be a nonempty full subposet of $\bfP$.
\begin{enumerate}
\item
For any $(x, y) \in [\le]_\bfP$, we set $[x,y]\coloneqq \{z \in \bfP \mid x \le z \le y\}$, and call it the \emph{segment} from $x$ to $y$ in $\bfP$.
The set of all segments in $\bfP$ is denoted by $\Seg(\bfP)$.
\item

The \emph{Hasse quiver}\footnote{{In this paper, we use the term Hasse quiver instead of directed Hasse diagram because {directed Hasse diagrams are quivers}, and the term quiver is commonly used in representation theory of algebras. We collect some basic definitions and conventions about quivers in \cref{Original definition of a compression system}.}} $H(\bfP)$ of $\bfP$ is a directed graph defined as follows.
The set of vertices, denoted by $H(\bfP)_0$, is equal to $\bfP$. The set of arrows, denoted by $H(\bfP)_1$, is given by the set $\{a_{y,x} \mid x < y, [x,y] = \{x,y\}\}$, and the source and the target of $a_{y,x}$ are $x$ and $y$,
respectively, where we set $a_{y,x}\coloneqq p_{y,x}$. In the sequel, we frequently visualize and express finite posets by their Hasse quivers. {For example, the 2D-grid} $G_{5,2}$ has the following Hasse quiver $H(G_{5,2})$:
\begin{equation}
\label{eq:G52}
\begin{tikzcd}
(1,2) & (2,2) & (3,2) & (4,2) & (5,2)\\
(1,1) & (2,1) & (3,1) & (4,1) & (5,1)
\Ar{1-1}{1-2}{}
\Ar{1-2}{1-3}{}
\Ar{1-3}{1-4}{}
\Ar{1-4}{1-5}{}
\Ar{2-1}{2-2}{}
\Ar{2-2}{2-3}{}
\Ar{2-3}{2-4}{}
\Ar{2-4}{2-5}{}
\Ar{2-1}{1-1}{}
\Ar{2-2}{1-2}{}
\Ar{2-3}{1-3}{}
\Ar{2-4}{1-4}{}
\Ar{2-5}{1-5}{}
\end{tikzcd}.
\end{equation}
\enlargethispage{2pt}%
We say that a poset is of \emph{Dynkin type} if its Hasse quiver is of Dynkin type\footnote{{By this we mean that the underlying graph of the quiver is a Dynkin graph. We refer the reader to \cite[Section~VII.2]{assemElementsRepresentationTheory2006} for the complete list of Dynkin graphs ($\bbA$, $\bbD$, and $\bbE$)}.}. {A \emph{zigzag poset} is a finite poset whose Hasse quiver has underlying graph a path. Thus zigzag posets are precisely the posets of Dynkin type $\bbA$; this convention includes totally ordered sets.}

\item
A \emph{source} (resp.\ \emph{sink}) of $I$ is nothing but a minimal (resp.\ maximal) element in $I$,
which is characterized as an element $x \in I$ such that in the Hasse quiver $H(I)$, there is no arrow with target (resp.\ source) $x$.
The set of all sources ({resp.\ }sinks) in $I$ is denoted by $\src(I)$
(resp.\ $\snk(I)$).

\item
{$I$ is said to be \emph{connected} if for all $x, y\in I$, there is a sequence of elements $x = z_0, z_1, \ldots, z_{n-1}, z_n = y$ in $I$ {such that any two} consecutive elements $z_i$ and $z_{i+1}$ are comparable. Namely, either $z_i \leq z_{i+1}$ or $z_{i+1}\leq z_i$ holds for $i = 0, \ldots, n-1$.}
\item
$I$ is said to be \emph{convex} if
for any $x, y \in I$ with $x \le y$, we have $[x,y] \subseteq I$.
\item
The \emph{convex hull} $\conv(I)$ of $I$ is defined as the smallest (with respect to the inclusion) convex subset of $\bfP$ that contains $I$. Equivalently, $\conv(I)$ is the union of all segments between elements of $I$.
\item
$I$ is called an \emph{interval} if $I$ is connected and convex.
\item
The set of all intervals of $\bfP$ is denoted by $\bbI(\bfP)$, or simply by $\bbI$. We regard $\bbI$ as a poset $\bbI = (\bbI, \le)$ with the inclusion relation: $I \le J \Leftrightarrow I \subseteq J$
for all $I, J \in \bbI$. 
Since $\bfP$ is finite, $\bbI$ is also finite. 
\item 
{Let $I\in\bbI$. We denote by $\Cov(I)$ the set of elements of $\bbI$ that cover $I$, namely,}
\[
\Cov(I) \coloneqq \{L \in \bbI \mid I < L \text{ and } [I, L] = \{I,L\}\}.
\]
\item
Let $U$ be a subset of $\bbI$. The least upper bound of $U$ is called the \emph{join} of $U$, and is denoted by $\bigvee U$.
{The join is unique if it exists.}
\end{enumerate}
\end{dfn}

\begin{rmk}
Any segment $[x,y]$ in $\Seg(\bfP)$ is an interval with source $x$ and sink $y$.
Hence $\Seg(\bfP) \subseteq \bbI(\bfP)$
(see the statements just after \cref{lem:intv-as-K-L} {for a more precise relation}).
\end{rmk}

\begin{ntn}
\label{ntn:principal-upsets}
For any poset $S = (S, \le)$ and $x \in S$, we set
$$
\begin{aligned}
\uset_S\, x\coloneqq \{y \in S \mid x \le y\}, \text{ and}\ \dset_S\, x\coloneqq\{y \in S \mid y \le x\}.
\end{aligned}
$$
\end{ntn}

\begin{dfn}
\label{dfn:Alexandrov}
{Let $S$ be a finite poset, and $U\subseteq S$.
We call $U$ an \emph{upset} (resp.\ \emph{downset}) of $S$ if $x\in U$ and $x\le y$ (resp.\ $y\le x$) imply $y\in U$.}
\end{dfn}

{An interval is determined by its sources and sinks. We use the following notation from \cite{BBH2024approximations} to make this precise.}

\begin{dfn}
\label{dfn:antichain}
A subset $K$ of $\bfP$ is called an \emph{antichain} in $\bfP$ {if any two distinct elements} of $K$ are incomparable. We denote by $\Ac(\bfP)$ the set of all antichains in $\bfP$.
For any $K, L \in \Ac(\bfP)$, we define $K \le L$ if for all $x \in K$, there exists $z_x \in L$ such that $x \le z_x$, and for all $z \in L$, there exists $x_z \in K$ such that $x_z \le z$. In this case, we define $[K, L]\coloneqq \{y \in \bfP \mid x \le y \le z \text{ for some }
x \in K \text{ and for some }z \in L\}$.
\end{dfn}

\begin{lem}
\label{lem:intv-as-K-L}
$\{[K, L] \mid K, L \in \Ac(\bfP), \ K \le L\}$
forms the set of all convex subsets in $\bfP$.
\end{lem}

\begin{proof}
    Let $K, L \in \Ac(\bfP)$ such that $K \le L$. First, let us show that $[K, L] = \conv(K \cup L)$. Let $y \in \conv(K \cup L)$. By definition, there exist $x_0, z_0 \in K \cup L$ such that $x_0 \le y \le z_0$. Now assume both $x_0 \in K$ and $z_0 \in K$. In this case, since $K \in \Ac(\bfP)$ and $x_0$ and $z_0$ are comparable, {necessarily} $x_0 = z_0$ and so $y = x_0 = z_0 \in K \subseteq [K,L]$. Similarly, if both $x_0 \in L$ and $z_0 \in L$, we have $y \in L \subseteq [K,L]$. So {assume} either $x_0 \in K, z_0 \in L$ or $x_0 \in L, z_0 \in K$. If $x_0 \in K, z_0 \in L$, then by definition we have $y \in [K,L]$. Now assume we have $x_0 \in L, z_0 \in K$. Since $K \le L$, there exists $k \in K$ such that $k \le x_0$. So we have $k \le x_0 \le y \le z_0$ with both $k, z_0 \in K$. Therefore $k = z_0$ because $K \in \Ac(\bfP)$, and so $y = k = z_0 \in K \subseteq [K,L]$. This proves that $[K, L] \supseteq \conv(K \cup L)$, and so $[K, L] = \conv(K \cup L)$. In particular, $[K, L]$ is a convex set. Now let $S$ be a convex subset in $\bfP$. Because $\bfP$ is finite, we can define $K\coloneqq \src(S),\, L\coloneqq \snk(S)$, and we have $K \le L$. Then it is clear that $S = \conv(K \cup L)$, and so $S = [K,L].$
\end{proof}

Hence we have $\bbI = \{[K, L] \mid K, L \in \Ac(\bfP), \ K \le L, \ [K, L] \text{ is connected}\}$.
In particular, we have $I = [\src(I), \snk(I)]$ for all $I \in \bbI$
and $\Seg(\bfP) = \{I \in \bbI(\bfP) \mid |\src(I)| = 1 = |\snk(I)|\}$.
{The following corollary is immediate from \cref{lem:intv-as-K-L}.}

\begin{cor}
\label{cor:EV}
Let $I \in \bbI$.
Then $I = \conv(\src(I) \cup \snk(I))$.
In particular, if $\src(I) \cup \snk(I) \subseteq J \in \bbI$,
then $I \le J$. \qed
\end{cor}

\begin{dfn}
\label{dfn:rank-inv}
{Let $M\in\mod A$ and $[x,y]\in\Seg(\bfP)$.
We call $\rank_{[x,y]}M\coloneqq \rank M_{y,x}$ the $[x,y]$-\emph{rank} of $M$.
The family $\rank_{\Seg(\bfP)}M\coloneqq (\rank_{[x,y]}M)_{[x,y]\in\Seg(\bfP)}$
is called the \emph{rank invariant} of $M$.}
\end{dfn}

\begin{dfn}[Interval modules]
\label{dfn:intv}
Let $I$ be an interval of $\bfP$.
\begin{enumerate}
\item

A persistence module $V_I$ over $\bfP$ is defined as follows. For $x\in \bfP$,
\[
V_{I}(x)\coloneqq\begin{cases}
    \k, & \textnormal{if } x\in I,\\
    0, & \textnormal{otherwise},
\end{cases}
\]
and for $p\in \k[\bfP](x, y)$,
\[
V_{I}(p)\coloneqq\begin{cases}
    k \id_\k, & \textnormal{if } (x, y)\in [\leq]_{I}\ \textnormal{and } p \coloneqq kp_{y,x}\ \textnormal{for some } k\in \k,\\
    0, & \textnormal{otherwise}.
\end{cases}
\]
It is easy to check that $V_I$ is indecomposable.

\item
A persistence module isomorphic to $V_I$ for some $I \in \bbI$
is called an \emph{interval module}.
\item
A persistence module is said to be \emph{interval decomposable} if
it is isomorphic to a finite direct sum of interval modules.
Thus 0 is trivially interval decomposable.
\end{enumerate}
\end{dfn}

{For a module $M$ and an integer $r\ge0$, we denote by $M^r$ the direct sum of $r$ copies of $M$, with $M^0=0$.}
We will use the notation $d_M(L)$ to denote
the multiplicity of an indecomposable direct summand $L$
of a module $M$ in its indecomposable decomposition
as explained in the following well-known theorem.

\begin{thm}[Krull--Schmidt]
\label{thm:KS}
Let $\calC$ be a finite $\k$-linear category, and
fix a complete set $\calL=\calL_\calC$
of representatives of isoclasses of indecomposable
objects in $\mod \calC$.
Then every finite-dimensional left $\calC$-module $M$ is
isomorphic to the direct sum $\Ds_{L \in \calL} L^{d_M(L)}$
for some unique function $d_M \colon \calL \to \bbZ_{\ge 0}$.
Therefore another finite-dimensional left $\calC$-module $N$
is isomorphic to $M$ if and only if $d_M = d_N$.
In this sense, the function $d_M$ is a complete invariant
of $M$ under isomorphisms.
\end{thm}
{In one-parameter persistent homology, $d_M$ determines the persistence diagram of $M$: each interval corresponds to a birth--death point whose multiplicity is $d_M(L)$, where $L$ is the corresponding interval module.}

\subsection{Incidence algebra and M{\"o}bius inversion}

{We relate incidence categories to incidence algebras, so that we can use results from the representation theory of finite-dimensional algebras. We also specify our convention for M{\"o}bius inversion, as stated in \cref{thm:MIF}. For more details on M{\"o}bius functions we refer the reader to \cite{stanley2011enumerative}.}

{We will apply the following results to the interval poset $S=(\bbI,\le)$ with real coefficients in \cref{sec3,sec4}.}
{Thus, in the applications in \cref{sec3,sec4}, $\Seg(S)$ denotes $\Seg(\bbI)$, the set of segments of the poset $(\bbI, \le)$ of all intervals of $\bfP$ ordered by inclusion.}
Throughout this subsection, $S$ denotes a \emph{finite poset}.

\begin{dfn}[Incidence algebra of $S$]
\label{dfn:incid-alg}
We define the \emph{incidence algebra} %
$\k S$ of $S$ by using the incidence category $\k[S]$ as
\[
\k S\coloneqq \Ds_{(x,y) \in S \times S}\k[S](x,y) = \Ds_{(x,y)\in [\le]_S}\k p_{y,x},
\]
with the multiplication
\begin{align}
\label{eq:mult-inc-alg}
\sum_{(y,z)\in [\le]_S} k_{z,y}\, p_{z,y}\cdot \sum_{(w,x)\in [\le]_S} k'_{x,w}\, p_{x,w}
& \coloneqq \sum_{(y,z)\in [\le]_S}\sum\limits_{(w,x)\in [\le]_S} \de_{x,y}k^{}_{z,y}k'_{x,w}\ p_{z,w} \nonumber\\
&\; = \sum_{\substack{w,x,z\in S \\ w\leq x\leq z}} k^{}_{z,x}k'_{x,w}\ p_{z,w},
\end{align}
{where the coefficients} $k^{}_{z,y}$ and $k'_{x,w}$ are elements of $\k$. Here and subsequently, $\de_{x,y}$ denotes Kronecker's delta symbol.
\end{dfn}

\begin{rmk}
\label{rmk:inc-alg}
{We describe $\k S$ in terms of matrices and functions, and specify its right action on $\k^S$.}
\begin{enumerate}
\item (As a matrix algebra with blocks $\k$ or $\bfzero$)
To express each element of $\k S$ as a matrix, we fix a total order on $S$ extending the original partial order. By regarding the isomorphism $\k \to \k p_{y,x}$ sending $1$ to $p_{y,x}$ as the identity map, we can regard $\k S$ as a matrix algebra over $\k$ with the set of $(y,x)$-entries $\k$ if $x \le y$, and $\bfzero$ otherwise.

\item (As a set of functions from $\Seg(S)$ to $\k$) Note that we have a bijection $\Seg(S) \to \setc*{p_{y,x}}{x, y \in S, x \le y}$ sending $[x,y]$ to $p_{y,x}$, and that each element $m$ of $\k S$ can be regarded as a function $\setc*{p_{y,x}}{x, y \in S, x \le y} \to \k$ sending each $p_{y,x}$ to the $(y,x)$-entry $m_{y,x}$ of $m$. By combining these we can also regard $\k S$ as the set $\k^{\Seg(S)}$ of functions $\Seg(S) \to \k$, namely, by identifying an element $m \in \k S$ with the function sending each segment $[x,y]$ to the $(y,x)$-entry $m_{y,x}$ of $m$.

\item (Right action on $\k^S$)
Let $\k^S$ be the vector space of all $\k$-valued functions with the domain $S$. Then $\k^S$ has a right $\k S$-module structure, the explicit definition of which is given as follows: Let $f \in \k^S$ and $(x, y) \in [\leq]_S$. Then
\begin{equation}
\label{eq:rt-action}
(f\cdot p_{y,x})(z) \coloneqq \de_{x,z}f(y)
\end{equation}
for all $z \in S$.
\end{enumerate}
\end{rmk}

We denote by $\mod \k S$ the category of finite-dimensional left $\k S$-modules.

{The following equivalence allows us to use results on modules over finite-dimensional algebras for persistence modules.}

\begin{lem}
There exists an equivalence $\ps \colon \mod \k[S] \to \mod \k S$ defined as follows:
For each $V \in \mod \k[S]$, $\ps(V)\coloneqq \Ds_{x \in S}V(x)$
with a left $\k S$-action defined by
$$
[a_{y,x}]_{x,y \in S} \cdot (v_x)_{x \in S}
\coloneqq \left[\sum_{x \in S}V(a_{y,x})(v_x)\right]_{y \in S}
$$
for all $[a_{y,x}]_{x,y \in S} \in \k S$ {\em (\cref{rmk:inc-alg} (1))}
and $(v_x)_{x \in S} \in \Ds_{x \in S}V(x)$.

For each morphism $f = (f_x \colon V(x) \to W(x))_{x \in S} \colon V \to W$ in $\mod \k[S]$,
$\ps(f)\coloneqq \Ds_{x\in S}f_x \colon \Ds_{x\in S}V(x) \to \Ds_{x\in S}W(x)$.
\end{lem}

\begin{proof}
A quasi-inverse $\ph \colon \mod \k S \to \mod \k[S]$ is defined as follows:
For each $M \in \mod \k S$, $\ph(M) \colon \k[S] \to \mod \k$ is a functor
given by $\ph(M)(x)\coloneqq \id_x M$ for all $x \in S$, and for each morphism
$a \colon x \to y$ in $\k[S]$, a linear map
$\ph(M)(a) \colon \id_x M \to \id_y M$ is defined by
the left multiplication by $a = \id_y a \id_x$.

For each morphism $f \colon M \to N$ in $\mod \k S$,
$\ph(f) \colon \ph(M) \to \ph(N)$ is given as a natural transformation
defined by the restriction maps
$\ph(f)_x\coloneqq f|_{\id_x M} \colon \id_x M \to \id_x N$ for all $x \in S$.

It is easy to verify that both $\ps$ and $\ph$ are
well-defined functors and that $\ph$ is a quasi-inverse of $\ps$.
\end{proof}

{By these equivalences $\ph,\ps$, we identify $\mod\k[S]$ with $\mod\k S$; thus we also call a persistence module over $S$ a (left) $\k S$-\emph{module}.}

\begin{dfn}[Zeta and M{\"o}bius functions]
We set
\[
\ze\coloneqq \sum_{x\le y}p_{y,x} \in\Ds_{x\le y}\k p_{y,x} =  \k S \iso \k^{\Seg(S)}
\]
(see \cref{rmk:inc-alg} (2)), and call it the \emph{zeta function} (on $S$). We note that $\ze$ is expressed as a lower triangular matrix with all diagonal entries $1$ in $\k S$ as a matrix algebra (see~\cref{rmk:inc-alg} (1)). {Thus $\ze$ is invertible in $\k S$. Its inverse $\mu\coloneqq\ze^{-1}$ is called the \emph{M{\"o}bius function} on $S$.}

Note that for any $f \in \k^S$, we have
\begin{equation}
\label{eq:zeta}
(f\cdot \ze)(z) = \sum_{x\le y}\de_{x,z}f(y) = \sum_{z \le y}f(y)
\end{equation}
for all $z \in S$ by \eqref{eq:rt-action}.
{See \cref{rmk:stanley-conventions} for a comparison with Stanley's conventions.}
\end{dfn}

{We will use the following inversion formula to recover interval summand multiplicities in \cref{sec3} and to define signed interval multiplicities in \cref{sec4}.}

\begin{samepage}
\begin{thm}[M{\"o}bius inversion formula]
\label{thm:MIF}
{For any $f, g \in \k^S$, the following statements are equivalent:}
\begin{enumerate}[label=(\arabic*),font=\normalfont]
\item
$f(x) = \sum_{x\le y \in S} g(y)$ {for all $x \in S$}; and
\item
$g(x) = \sum_{x\le y \in S} f(y) \mu([x,y])$ {for all $x \in S$}.
\end{enumerate}
\end{thm}
\end{samepage}

\begin{proof}
Since $\mu = \sum_{x\le y}\mu([x,y])p_{y,x}$,
we have
\begin{equation}
(f\cdot \mu)(z) = \sum_{x\le y} \de_{x,z}f(y)\mu([x,y])
= \sum_{z \le y}f(y)\mu([z,y]). \label{formula-mu}
\end{equation}
By equality~\eqref{formula-mu} together with equality~\eqref{eq:zeta},
the equivalence follows from 
the fact that $f = g \cdot \ze$ if and only if $f \cdot \mu = g$.
\end{proof}

\begin{rmk}[Comparison with the conventions in \cite{stanley2011enumerative}]
\label{rmk:stanley-conventions}
{Under the identification in \cref{rmk:inc-alg} (2), $\k S$ is the opposite algebra of Stanley's incidence algebra $I(S)$ \cite[Section 3.6]{stanley2011enumerative}; their zeta and M{\"o}bius functions coincide.
Our right action \eqref{eq:rt-action} corresponds to the left action in the proof of \cite[Proposition 3.7.2]{stanley2011enumerative}.
\cref{thm:MIF} is the dual M{\"o}bius inversion formula there, with $f$ and $g$ interchanged.
The sums run over $y \ge x$, as required for the sums over $J \supseteq I$ in \cref{prp:cdelta}.
}
\end{rmk}

\section{{Compressions and interval ranks}}\label{sec3}

{A compression system assigns a restriction functor to each interval. We show that the resulting summand multiplicities recover the usual rank on segments and use them to define interval ranks. After establishing the properties needed for interval replacement, we compare different compressions and identify the total interval rank with the generalized rank.}

\subsection{Compression systems}

{A compression system assigns an order-preserving map to each interval, and hence a restriction functor. We specify the conditions on these maps and give examples, including the total and source-sink compression systems.}

\begin{dfn}
\label{dfn:comp-sys-simplified-ver}
A \emph{compression system} for $A$ ($\coloneqq \k[\bfP]$) is a family $\xi = (\xi_I)_{I \in \bbI}$ of
order-preserving maps $\xi_I \colon I^{\xi} \to \bfP$ from a connected finite poset $I^{\xi}$
satisfying the following conditions for all $I \in \bbI$:
\begin{enumerate}
\item
$\xi_I$ factors through the poset inclusion $I \hookrightarrow \bfP$. Namely, $\Im \xi_I \subseteq I$.
\item
$\src(I) \cup \snk(I) \subseteq \Im \xi_I$.
\end{enumerate}
{When $A$ (or the poset $\bfP$) is clear from the context, we simply call $\xi$ a compression system.}
\end{dfn}

Let $I \in \bbI$.
Then %
the restriction functor
$R_I^\xi \colon \mod A \to \mod \k[I^\xi]$ is defined by
sending $M$ to $M \circ \k[\xi_I]$ for all $M \in \mod A$.
{Throughout the rest of the paper, $\xi$ denotes a compression system for $A$, and we write $R_I$ for $R_I^\xi$ when the choice of $\xi$ is clear.}

{The following lemma identifies the indecomposable module whose multiplicity will be counted after compression.}

\begin{lem}
\label{lem:com-mult-full-int}
Let $\xi$ be a compression system for $A$.
Then for each $I \in \bbI$, we have $R_I(V_I) = V_{I^\xi}$ as a persistence module over $I^\xi$.
Here, we slightly abuse notation by using the same symbol $V_{(\blank)}$ for interval modules in different categories: $V_I$ denotes the interval module in $\mod A$, whereas $V_{I^\xi}$ denotes the interval module in $\mod \k[I^\xi]$.
\end{lem}

\begin{proof}
Let $x \in I^\xi$. Then since $\xi_I(x) \in I$ holds by definition of $\xi$, we have $R_I(V_I)(x) = V_I(\k[\xi_I](x)) = V_I(\xi_I(x)) = \k$. Now let ${p_{y,x}\in\k[I^\xi](x,y)}$ for $(x, y)\in [\leq]_{I^\xi}$. Then $\xi_I(x) \le \xi_I(y)$ in $\bfP$, where $\xi_I(x), \xi_I(y) \in I$. Hence $R_I(V_I)(p_{y, x}) = V_I(\k[\xi_I](p_{y, x})) = V_I(p_{\xi_{I}(y), \xi_{I}(x)}) = \id_\k$. As a consequence, we have $R_I(V_I) = V_{I^\xi}$.
\end{proof}

\begin{exm}[tot]
\label{exm:total}
For each $I \in \bbI$, set $I^\tot\coloneqq I$, and let $\tot_I$ be the inclusion $I \hookrightarrow \bfP$.
This defines a compression system $\tot\coloneqq (\tot_I)_{I \in \bbI}$ for $A$,
which is called the \emph{total} compression system for $A$.
\end{exm}

\begin{exm}[ss]
\label{exm:ss}
For each $I \in \bbI$, {let $I^{\ss}$ be} the full subposet $\snk(I)\cup \src(I)$ of $I$,
and let 
$\ss_I \colon I^{\ss} \hookrightarrow \bfP$ be the inclusion.
This defines a compression system $\ss\coloneqq (\ss_I)_{I \in \bbI}$ for $A$,
which is called the \emph{source-sink} compression system for $A$.
\end{exm}

\begin{exm}
Let $\bfP\coloneqq G_{5,2}$ as in {\eqref{eq:G52}}, and
$I$ be the interval with $\src(I){=} \{(1,2),\allowbreak (2,1)\}$ and $\snk(I){=} \{(3,2)\}$.
Take
$I^{\xi}$ %
to be the full subposet $\{(1,2), (2,2), (3,2), (2,1)\}$ of $I$, and let
$\xi_I \colon I^{\xi} \to \bfP$ {be the inclusion. Then}
this $\xi_I$ can be taken as a component of a compression system $\xi$ for $A$,
which satisfies $\ss_I \ne \xi_I \ne \tot_I$.
\end{exm}

\subsection{{Interval ranks}}

{The multiplicity of $R_I(V_I)$ in $R_I(M)$ extends the usual rank from segments to intervals. We first prove this assertion for a segment; this will motivate the definition below.}

\begin{prp}
\label{prp:seg-rank}
{Let $\xi$ be a compression system, $M\in\mod A$, and $I=[x,y]\in\Seg(\bfP)$. Then
\begin{equation}
\label{eq:rank-equality}
d_{R_I(M)}(R_I(V_I))=\rank M_{y,x}.
\end{equation}
In particular, this multiplicity does not depend on $\xi$.}
\end{prp}

\begin{proof}
{Set $r\coloneqq \rank M_{y,x}$ and $c\coloneqq d_{R_I(M)}(R_I(V_I))$.
Choose linear maps $s\colon\k^r\to M(x)$ and $t\colon M(y)\to\k^r$ such that $tM_{y,x}s=\id_{\k^r}$.
For every $u\in I^\xi$, define
\[
{\al}_u\coloneqq M_{\xi_I(u),x}s,\quad
\pi_u\coloneqq tM_{y,\xi_I(u)}.
\]
These maps are well defined because $x\le\xi_I(u)\le y$.
By the functoriality of $M$, they give morphisms
${\al}\colon V_{I^\xi}^{r}\to R_I(M)$ and $\pi\colon R_I(M)\to V_{I^\xi}^{r}$ with $\pi{\al}=\id$.
Thus $V_{I^\xi}^{r}$ is a direct summand of $R_I(M)$, so $r\le c$ by \cref{lem:com-mult-full-int}.
It remains to prove $c\le r$.}
\begingroup
By \cref{lem:com-mult-full-int}, $R_I(V_I)=V_{I^\xi}$ in $\mod\k[I^\xi]$.
Here $V_{I^\xi}$ is supported on the whole poset $I^\xi$.
Thus
${V_{I^\xi}^{c}}(u) = \k^c$ for all $u \in I^\xi$,
and all its structure linear maps are $\id_{\k^c}$.
By the definition of $c$, we have a decomposition $R_I(M)\iso V_{I^\xi}^{c}\ds N$ for some $N \in \mod\k[I^\xi]$,
{and hence a section}
$\io\colon {V_{I^\xi}^{c}}\to R_I(M)$.
In particular,
each component $\io_u\colon\k^c \to R_I(M)(u) = M(\xi_I(u))$ is injective.

Condition (1) in \cref{dfn:comp-sys-simplified-ver} gives $x\le\xi_I(u)\le y$ for every $u\in I^\xi$.
Since $\src(I)=\{x\}$ and $\snk(I)=\{y\}$, condition (2) gives vertices $u_x,u_y\in I^\xi$ with
$\xi_I(u_x)=x$ and $\xi_I(u_y)=y$.
Note that these two vertices {need not be comparable} in $I^\xi$.

Let $v,w \in I^\xi$ with $v\le w$ in $I^\xi$.
By the naturality of $\io$ and the functoriality of $M$, the following diagram commutes:
\[\begin{tikzcd}[column sep=20pt, ampersand replacement=\&]
	{{V_{I^\xi}^{c}}(v)} \& R_I(M)(v) \& M(\xi_I(v)) \&[20pt] {M(y)} \\
	{{V_{I^\xi}^{c}}(w)} \& R_I(M)(w) \& M(\xi_I(w)) \& {M(y)}
	\arrow["{\iota_v}", from=1-1, to=1-2]
    \Ar{1-2}{1-3}{equal}
    \Ar{1-3}{1-4}{"{M_{y,\xi_I(v)}}"}
	\arrow["{\iota_w}"', from=2-1, to=2-2]
    \Ar{2-2}{2-3}{equal}
    \Ar{2-3}{2-4}{"{M_{y,\xi_I(w)}}"'}
	\arrow["\id_{\k^c}"', from=1-1, to=2-1]
    \Ar{1-2}{2-2}{"R_I(M)(p_{w,v})"}
    \Ar{1-3}{2-3}{"M(p_{\xi_I(w),\xi_I(v)})"}
	\arrow[equal, from=1-4, to=2-4]
\end{tikzcd}.
\]
Hence we have $M_{y,\xi_I(v)}\io_v=M_{y,\xi_I(w)}\io_w$.
Since $I^\xi$ is connected, any two of its vertices can be joined by a finite sequence of vertices whose consecutive terms are comparable.
Applying the equality along such a sequence shows that the map $M_{y,\xi_I(u)}\io_u\colon\k^c \to M(y)$ is independent of $u$.
In particular, evaluating it at $u_x$ and $u_y$ gives
\[
M_{y,x}\io_{u_x}=M_{y,y}\io_{u_y}=\io_{u_y}.
\]
The injectivity of $\io_{u_y}$ now yields
\[
c = \rank\io_{u_y}
=\rank(M_{y,x}\io_{u_x})
\le\rank M_{y,x}.
\]
This completes the proof.
\endgroup
\end{proof}

Let $\calL_I\coloneqq \calL_{\k[I^\xi]}$ be a complete set of representatives of isoclasses of indecomposable left $\k[I^\xi]$-modules ({see the notation in} \cref{thm:KS}).
Since $R_I(V_I)$ is indecomposable by \cref{lem:com-mult-full-int},
we may assume that $R_I(V_I) \in \calL_I$.

\begin{dfn}
\label{dfn:int-rk}
{Let $M\in\mod A$ and $I\in\bbI$.
We define the $I$-\emph{rank} of $M$ under $\xi$ by
\[
\rank_I^\xi M\coloneqq d_{R_I(M)}(R_I(V_I)).
\]
The family $\rank_\bbI^\xi M\coloneqq (\rank_I^\xi M)_{I\in\bbI}$ is called the
\emph{interval rank invariant} of $M$ under $\xi$.
We also regard this family as the function $\bbI\to\bbZ$ given by $I\mapsto\rank_I^\xi M$.
In particular, $\rank_I^\tot M$ is called the \emph{total} $I$-rank of $M$.}
\end{dfn}

{By \cref{prp:seg-rank}, we have $\rank_{[x,y]}^\xi M=\rank_{[x,y]}M$ for every compression system $\xi$. Thus this definition extends the usual rank invariant.}

{We next compute interval ranks on interval modules and show additivity. These two properties will be used to relate interval ranks to interval summand multiplicities.}

\begin{prp}
\label{prp:IJ}
Let $I, J \in \bbI$. Then
\[
{\rank_{I}^{\xi} V_J} =
\begin{cases}
    1, & \textnormal{if } I \le J,\\
    0, & \textnormal{otherwise}.
\end{cases}
\]
\end{prp}

\begin{proof}
If $I \le J$, then ${\rank_{I}^{\xi} V_J} = d_{R_I(V_J)}(R_I(V_I)) = d_{R_I(V_I)}(R_I(V_I)) = 1$. Otherwise, \cref{cor:EV} ensures the existence of a vertex $x \in \src(I) \cup \snk(I)$ that is not in $J$. By the defining condition (2) of the compression system $\xi$,
there exists $x' \in I^{\xi}$ such that $\xi_I(x')=x$. By definition, $x'$ satisfies $R_I(V_I)(x') = \k$ and $R_I(V_J)(x') = 0$. Hence in particular, $R_I(V_I)$ is not a direct summand of $R_I(V_J)$. Thus, ${\rank_{I}^{\xi} V_J} = 0$.
\end{proof}

\begin{prp}
\label{prp:AD}
Let $M, N \in \mod A$, and $I \in \bbI$.
Then
\[
{\rank_{I}^{\xi} (M \oplus N)} = {\rank_{I}^{\xi} M} + {\rank_{I}^{\xi} N}.
\]
\end{prp}

\begin{proof}
This is a direct consequence of the additivity of $R_I$ and the uniqueness of $d_{M \oplus N}$ in \cref{thm:KS}.
\end{proof}

{For an interval decomposable module, the next proposition expresses its interval ranks as sums of the multiplicities of interval summands whose supports contain $I$. M{\"o}bius inversion then recovers the summand multiplicities.}

\begin{prp}
\label{prp:cd}
Let $M \in \mod A$ and $I \in \bbI$. If $M$ is interval decomposable, then
\[
{\rank_{I}^{\xi} M} = \sum_{I \le J \in \bbI}{d_M(V_J)}.
\]
This can be rewritten as
\[
{\rank_{\bbI}^{\xi} M} =  d_M \zeta,
\]
where $d_M(I) \coloneqq d_M(V_I)$.
\end{prp}

\begin{proof}
By assumption, $M$ can be decomposed as a direct sum of interval modules: $M = \bigoplus_{J \in \bbI} V_J^{d_M(V_J)}$. Now, \cref{prp:AD} yields
\[
{\rank_{I}^{\xi} M} = \sum_{J \in \bbI}{d_M(V_J) \ {\rank_{I}^{\xi} V_J}}.
\]
\cref{prp:IJ} leads to the desired formula.
\end{proof}

{In particular, for a fixed interval decomposable module $N$,
$\rank_I^\xi N$ does not depend on $\xi$.}

\begin{cor}
\label{cor:dc}
Let $M \in \mod A$. If $M$ is interval decomposable, then
\[
d_M = {(\rank_{\bbI}^{\xi} M)\mu}.
\]
\end{cor}

\begin{proof}
This follows directly from \cref{thm:MIF}.
\end{proof}

\subsection{{Comparing interval ranks}}

{For general modules, interval ranks may depend on the compression system. We first show that the total compression gives the least $I$-rank, and then consider monotonicity as $I$ varies.}

\begin{ntn}
\label{ntn:tot}
Let $\xi$ be a compression system and $I \in \bbI$.
Then
{by condition (1)} in \cref{dfn:comp-sys-simplified-ver} %
(the commutativity of the following diagram on the left),
the functor $\k[\xi_I] \colon \k[I^\xi] \to A$ factors through $\k[\tot_I]$ as in the following diagram on the right:
\[
\begin{tikzcd}
I^{\xi} && \bfP\\
& I
\Ar{1-1}{1-3}{"\xi_I"}
\Ar{1-1}{2-2}{"\xi'"'}
\Ar{2-2}{1-3}{"\tot_I"'}
\end{tikzcd},
\qquad
\begin{tikzcd}
\k[I^\xi] && \k[\bfP]\\
& \k[I]
\Ar{1-1}{1-3}{"{\k[\xi_I]}"}
\Ar{1-1}{2-2}{"{\k[\xi']}"'}
\Ar{2-2}{1-3}{"{\k[\tot_I]}"'}
\end{tikzcd}.
\]
{Here $\xi^{\prime}\colon I^\xi\to I$ is the corestriction of $\xi_I$.}
Hence we have the corresponding factorization {of $R^\xi_I$ through $R^\tot_I$} as in the diagram
\[
\begin{tikzcd}
\mod \k[I^\xi] && \mod \k[\bfP]\\
& \mod \k[I]
\Ar{1-3}{1-1}{"R_I^\xi"'}
\Ar{2-2}{1-1}{"{R'}"}
\Ar{1-3}{2-2}{"R^\tot_I"}
\end{tikzcd}.
\]
{Here $R^{\prime}$ is the restriction functor induced by $\k[\xi^{\prime}]$.}
\end{ntn}

{For any $M\in\mod A$ and $I\in\bbI$, the total $I$-rank of $M$ is the least of the $I$-ranks obtained from all compression systems, as shown in the following proposition.}

\begin{prp}
\label{prp:least-comp-mult}
Let $\tot$ be the total compression system for $A$. Then for any compression system $\xi$ for $A$, any interval $I\in \bbI$, and any $M\in \mod A$, we have
\[
{\rank_{I}^{\xi} M} \geq {\rank_{I}^{\tot} M}.
\]
\end{prp}

\begin{proof}
We write $s = {\rank_{I}^{\tot} M}$ and $r = {\rank_{I}^{\xi} M}$. By the definition of {$I$-rank} under $\tot$, we have the decomposition
\begin{equation}
\label{eq:decomp_R_M}
	R^{\tot}_I(M)\cong \left[R^{\tot}_I(V_{I})\right]^{s}\ds N \text{ in } \mod \k[I],
\end{equation}
where $N$ has no direct summand that is isomorphic to $R^{\tot}_I(V_{I})$. Applying the functor $R'$ (given in \cref{ntn:tot}) to~\eqref{eq:decomp_R_M} yields
\[
R^{\xi}_{I}(M)\cong \left[R^{\xi}_{I}(V_{I})\right]^{s}\ds R'(N) \text{ in } \mod \k[I^\xi].
\]
Therefore $r\geq s$, and the assertion follows.
\end{proof}

Under a certain condition on a compression system $\xi$, {interval ranks} have a ``monotonically decreasing'' property, namely, for any $I, J\in \bbI$,
$I \le J$ implies
 ${\rank_{I}^{\xi} M}\geq {\rank_{J}^{\xi} M}$ for all $M\in \mod A$.
We will give an example below.

\begin{dfn}
\label{dfn:mono-compression}
Let $\xi = (\xi_{I})_{I\in \bbI}$ be a compression system for $A$. Then $\xi$ is said to be \emph{monotonic} if for all intervals $I, J\in \bbI$ with $I \leq J$, there exists an order-preserving map $\Phi\colon I^\xi \to J^\xi$ such that $\xi_{I} = \xi_{J} \circ \Phi$ (this can be read as $\xi_I$ is ``weaker'' than $\xi_J$).
\end{dfn}

\begin{exm}
\label{exm:tot-monoic}
The total compression system $\tot$ is monotonic. Indeed, $I^\tot = I$, $J^\tot = J$ and we take $\Phi$ in \cref{dfn:mono-compression} to be the inclusion map $I\hookrightarrow J$.
\end{exm}

\begin{prp}
\label{prp:mono-decrease}
Let $\xi = (\xi_{I})_{I\in \bbI}$ be a monotonic compression system for $A$. Then for any intervals $I, J\in \bbI$ such that $I \leq J$ and for every $M\in \mod A$, we have
\[
{\rank_{I}^{\xi} M} \geq {\rank_{J}^{\xi} M}.
\]
\end{prp}

\begin{proof}
We write $s = {\rank_{J}^{\xi} M}$ and $r = {\rank_{I}^{\xi} M}$. By \cref{dfn:int-rk}, we have the decomposition
\begin{equation}
\label{eq:decomp_R_J_M}
	R_{J}(M)\cong \left[R_{J}(V_{J})\right]^{s}\ds N \text{ in } \mod \k[J^\xi],
\end{equation}
where $N$ has no direct summand that is isomorphic to $R_{J}(V_{J})$. Applying the restriction functor $R_{\Phi}$ induced by $\Phi$ to~\eqref{eq:decomp_R_J_M} yields
\[
R_{I}(M)\cong \left[R_{I}(V_{J})\right]^{s}\ds R_{\Phi}(N) \text{ in } \mod \k[I^\xi].
\]
Since $R_{I}(V_{J}) = R_{I}(V_{I})$, we have $R_{I}(M)\cong \left[R_{I}(V_{I})\right]^{s}\ds R_{\Phi}(N)$. Hence $r\geq s$.
\end{proof}

{The following corollary is immediate from~\cref{exm:tot-monoic} and \cref{prp:mono-decrease}.}

\begin{cor}
\label{prp:mono-decrease-tot}
	Let $\xi = \tot$. Then for all intervals $I, J\in \bbI$ with $I \leq J$ and for every $M\in \mod A$, we have
\begin{equation*}
{\rank_{I}^{\tot} M} \geq {\rank_{J}^{\tot} M}.
\end{equation*}
\end{cor}

\subsection{{The total interval rank and the generalized rank}}

{The generalized rank invariant of Kim and M{\'e}moli~\cite{kim2021generalized} is defined using limits and colimits. We recall its definition and identify it with the total interval rank invariant.}

\begin{dfn}
\label{dfn:concrete-dfn-GRI}
Let $I$ be a finite connected poset, and $M \in \mod \k[I]$. Since $I$ is finite and $M \in \mod \k[I]$, both $\varprojlim M$ and $\varinjlim M$ are easily constructed in $\mod \k$. {Let $\pi_x$ and $\si_x$ denote the limit projections and the colimit structure maps, respectively. By definition, we have a commutative diagram}
\[
\begin{tikzcd}
& \Nname{Mx}M(x)\\
\Nname{lim}\varprojlim M && \Nname{colim} \varinjlim M\\
& \Nname{My} M(y)
\Ar{lim}{Mx}{"\pi_x"}
\Ar{lim}{My}{"\pi_y" '}
\Ar{Mx}{colim}{"\si_x"}
\Ar{My}{colim}{"\si_y" '}
\Ar{Mx}{My}{"M(p_{y,x})"}
\end{tikzcd}
\]
for any $(x, y) \in [\le]_I$, which shows that for any $x,y \in I$, we have $\si_x\pi_x = \si_y \pi_y$ if $x$ and $y$ are in the same connected component of $I$. But since $I$ is connected, the equality above holds for all $x, y \in I$. The common linear map is denoted by $\thf_M \colon \varprojlim M \to \varinjlim M$.

Now, for a (locally finite) poset $\bfP$, a finite interval subposet $I$ of $\bfP$, and $M \in \mod \k[\bfP]$, the $\rank$ of the linear map $\thf_{R_I^\tot(M)}$ for the module $R_I^\tot(M) \in \mod \k[I]$ is called the \emph{generalized rank} of $M$ at $I$.

The family {$(\rank \thf_{R_I^\tot(M)})_{I \in \bbI}$} is called
the \emph{generalized rank invariant} of $M$.
\end{dfn}

{The following lemma is \cite[Lemma~3.1]{chambersPersistentHomologyDirected2018}. We give a complete proof, including the splitting as modules needed for the identification.}

\begin{lem}
\label{lem:gen-rk-inv}
Let $I$ be a finite connected poset, and $M \in \mod \k[I]$. Then $M$ has a direct sum decomposition
\[
M \iso V_{I}^{s} \ds N
\]
as $\k[I]$-modules for some $N$,
where $\rank \thf_{M} = s,\, \rank \thf_{N} = 0$.
Hence in particular, we have $d_M(V_I) = \rank \thf_{M}$.
\end{lem}

\begin{proof}
There exist some vector subspaces $P \subseteq \varprojlim M$ and $T \subseteq \varinjlim M$ such that $P\ds \Ker \thf_M = \varprojlim M$ and $\Im \thf_M \ds T = \varinjlim M$. Let $\si \colon P \to {\varprojlim M}$ be the inclusion and $\pi \colon \varinjlim M \to \Im \thf_M$ the projection with respect to this decomposition. We set $\ph_x\coloneqq \pi_x\si$, $\ro_x\coloneqq \pi \si_x$ for all $x \in I$. Then we have the following commutative diagram:
\[
\begin{tikzcd}
&& \Nname{Mx}M(x)\\
\Nname{P}P &\Nname{lim}\varprojlim M && \Nname{colim} \varinjlim M & \Nname{I}\Im \thf_M\\
&& \Nname{My} M(y)
\Ar{lim}{Mx}{"\pi_x"}
\Ar{lim}{My}{"\pi_y" '}
\Ar{Mx}{colim}{"\si_x"}
\Ar{My}{colim}{"\si_y" '}
\Ar{Mx}{My}{"M(p_{y,x})", pos=0.65}
\Ar{lim}{colim}{"\thf_M", pos=0.3, crossing over}
\Ar{P}{Mx}{"\ph_x", bend left=10pt}
\Ar{P}{My}{"\ph_y" ', bend right=10pt}
\Ar{P}{lim}{"\si"}
\Ar{Mx}{I}{"\ro_x", bend left=10pt}
\Ar{My}{I}{"\ro_y" ', bend right=10pt}
\Ar{colim}{I}{"\pi"}
\end{tikzcd}.
\]
Since $\thf_M$ restricts to an isomorphism $\mu' \colon P \to \Im \thf_M$, we have $\mu' = \pi \thf_M \si$. Thus $\dim P = \dim \Im \thf_M = \rank \thf_M$. Set $s\coloneqq \rank \thf_M$, the common value. Then there exists an isomorphism $\al \colon \k^{s} \to P$, which gives an isomorphism $\be\coloneqq (\mu'\al)\inv \colon \Im \thf_M \to \k^{s}$. Set $\ph_x'\coloneqq \ph_x \al \colon \k^{s} \to M(x)$ and $\ro'_x\coloneqq \be \ro_x \colon M(x) \to \k^{s}$. Then $\Cok \ph'_x = \Cok \ph_x$, $\Ker \ro'_x = \Ker \ro_x$ for all $x \in I$, and {we have the following two commutative diagrams with exact rows, where $\ps_x\colon M(x)\to\Cok\ph_x$ is the canonical cokernel map:}
\begin{equation}
\label{eq:2-sh-ex-seq}
\begin{aligned}
&\begin{tikzcd}
0 & \k^{s} & M(x) & \Cok \ph_x & 0\\
0 & \k^{s} & M(y) & \Cok \ph_y & 0
\Ar{1-1}{1-2}{}
\Ar{1-2}{1-3}{"\ph'_x"}
\Ar{1-3}{1-4}{"\ps_x"}
\Ar{1-4}{1-5}{}
\Ar{2-1}{2-2}{}
\Ar{2-2}{2-3}{"\ph'_y" '}
\Ar{2-3}{2-4}{"\ps_y" '}
\Ar{2-4}{2-5}{}
\Ar{1-2}{2-2}{equal}
\Ar{1-3}{2-3}{"M(p_{y,x})"}
\Ar{1-4}{2-4}{"f_{y,x}", dashed}
\end{tikzcd}
,\text{ and }\\
&\begin{tikzcd}
0 & \Ker \ro_x & M(x) & \k^{s} & 0\\
0 & \Ker \ro_y & M(y) & \k^{s} & 0
\Ar{1-1}{1-2}{}
\Ar{1-2}{1-3}{"\ta_x"}
\Ar{1-3}{1-4}{"\ro'_x"}
\Ar{1-4}{1-5}{}
\Ar{2-1}{2-2}{}
\Ar{2-2}{2-3}{"\ta_y" '}
\Ar{2-3}{2-4}{"\ro'_y" '}
\Ar{2-4}{2-5}{}
\Ar{1-2}{2-2}{"g_{y,x}", dashed}
\Ar{1-3}{2-3}{"M(p_{y,x})"}
\Ar{1-4}{2-4}{equal}
\end{tikzcd},
\end{aligned}
\end{equation}
where the vertical map $f_{y,x}$ (\textnormal{resp.}\ $g_{y,x}$) is the unique linear map making the diagram commutative, and $\ta_z \colon \Ker \ro_z \to M(z)$ is the inclusion map for all $z \in I$. The uniqueness of $f_{y,x}$ (\textnormal{resp.}\ $g_{y,x}$) for all $(x,y) \in [\le]_I$ defines a $\k[I]$-module $C$ (\textnormal{resp.}\ $K$) by setting $C(x)\coloneqq \Cok \ph_x,\, C(p_{y,x})\coloneqq f_{y,x}$ (\textnormal{resp.}\ $K(x)\coloneqq \Ker \ro_x,\, K(p_{y,x})\coloneqq g_{y,x}$) for all $x \in I,\, (x,y) \in [\le]_I$. Set $\ga\coloneqq (\ga_x)_{x \in I}$ for all $\ga \in \{\ph', \ps, \ta, \ro'\}$. Then the commutative diagrams above show that $\ph', \ps, \ta, \ro'$ are morphisms of $\k[I]$-modules, and give us the following short exact sequences of $\k[I]$-modules:
\begin{equation}
\label{eq:ex-seq-PC-KP}
\begin{gathered}
0 \to V_I^{s} \ya{\ph'} M \ya{\ps} C \to 0 \text{ and }\\
0 \to K \ya{\ta} M \ya{\ro'} V_I^{s} \to 0.
\end{gathered}
\end{equation}

We claim $\ro' \ph' = {\id_{V_I^s}}$. Indeed, for each $x \in I$, we have $\ro'_x\ph'_x = \be \ro_x \ph_x \al = \be \mu' \al = \id_{\k^{s}}$. As a consequence, the short exact sequences above split, and hence $M$ has direct sum decompositions $V_I^{s} \ds C \iso M\, (\iso K \ds V_I^{s})$ as $\k[I]$-modules. By the additivity of both $\varprojlim$ and $\varinjlim$, we have $\rank \thf_M = s \rank \thf_{V_I} + \rank \thf_{C}$. Here note that $\thf_{V_I}$ is given by the identity $\id_\k \colon \k \to \k${; thus} $\rank \thf_{V_I} = 1$, which together with $\rank \thf_M = s$ shows that $\rank \thf_C = 0$. Therefore the assertion holds for $N\coloneqq C$.

Note that $N$ {does not have a direct summand} isomorphic to $V_I$ because {$\rank\thf_N=0$}. Hence we have $d_M(V_I) = s = \rank \thf_M$.
\end{proof}

\begin{rmk}
\label{rmk:gap-chambers}
In the proof of \cite[Lemma 3.1]{chambersPersistentHomologyDirected2018},
{the authors state} that the decomposition $M(x) \iso P \ds \Cok \ph_x$
is preserved by $M(p_{y,x})$.
{They then conclude} that this fact establishes a direct sum
$M \iso V_I^{s} \ds C$.
{However, what follows from the naturality of the maps $\ph_x$ is the first diagram in \eqref{eq:2-sh-ex-seq}, with the induced map $f_{y,x}$ between the cokernels.}
Each row splits as vector spaces, but this alone does not show
that the splittings can be chosen compatibly with all structure maps of $M$.
This fact was not proved in their paper.
{Namely, the missing part is to show that the first short exact sequence in \eqref{eq:ex-seq-PC-KP} splits over $\k[I]$.}
{For this, we use the second short exact sequence in \eqref{eq:ex-seq-PC-KP},}
which provides the necessary retraction $\ro'$ for $\ph'$.

\end{rmk}

\begin{rmk}
\label{rmk:concide-gri-int-rk}
{By \cref{lem:gen-rk-inv}, $\rank\thf_{R_I^\tot(M)}=d_{R_I^\tot(M)}(V_I)=\rank_I^\tot M$ for all $I\in\bbI$.
Hence the generalized rank invariant coincides with the total interval rank invariant.}
\end{rmk}

\begin{rmk}
{Together, \cref{lem:gen-rk-inv} and \cref{prp:mono-decrease-tot} also imply \cite[Proposition~3.8]{kim2021generalized}.}

\end{rmk}

\section{Interval replacement and interval rank invariant}\label{sec4}

{We define interval replacement under a compression system and prove that it preserves the corresponding interval rank invariant, and hence also the usual rank invariant and the dimension vector (\cref{thm:rank}).}

\subsection{{Signed interval multiplicities}}

{\cref{cor:dc} recovers interval summand multiplicities by M{\"o}bius inversion. Applying the same operation to an arbitrary module gives the coefficients of its interval replacement.}

\begin{dfn}
\label{dfn:delta}
{Let $M\in\mod A$. The \emph{signed interval multiplicity} of $M$ under $\xi$ is the function $\de_M^\xi\colon\bbI\to\bbZ$ defined by}
\[
{\de_M^\xi\coloneqq (\rank_\bbI^\xi M)\mu.}
\]
\end{dfn}

\begin{rmk}
{By \cref{cor:dc}, if $M$ is interval decomposable, then $\de_M^\xi(I)=d_M(V_I)$ for all $I\in\bbI$. The definition also applies when $M$ is not interval decomposable.}
\end{rmk}

{The following identity recovers interval ranks from signed interval multiplicities and will be used to prove preservation under interval replacement.}

\begin{prp}
\label{prp:cdelta}
Let $M \in \mod A$. For all $I \in \bbI$, we have
\[
{\rank_{I}^{\xi} M} = \sum_{I \le J \in \bbI}{\de^\xi_M(J)},
\]
that is to say
\[
{\rank_{\bbI}^{\xi} M} =  \de^\xi_M \ze.
\]
\end{prp}
\begin{proof}
This is a direct consequence of \cref{thm:MIF}.
\end{proof}

\subsection{Interval replacement}

{Signed interval multiplicities may be negative, so we regard the interval replacement as a formal difference of interval decomposable modules in the split Grothendieck group. We extend additive invariants to this group in order to compare a module with its replacement.}

Recall that the \emph{split Grothendieck group} $K^\ds(A)$ of $A$ is defined
to be the factor group $F_A/R_A$ of the free abelian group $F_A$ with basis the set of isomorphism classes $[M]$ of all $M\in\mod A$ modulo
the subgroup $R_A$ generated by the elements $[M \ds N] - [M] - [N]$
for all $M, N \in \mod A$.
We set $\br{M}\coloneqq [M] + R_A \in K^\ds(A)$ for all $M \in \mod A$.
Let $\calL$ be {a complete set of representatives} of isoclasses of
indecomposable $A$-modules.
Then we have
\begin{equation}
\label{eq:sp-Gr-gp}
K^\ds(A) = \Ds_{L \in \calL} \bbZ \br{L}.
\end{equation}

\begin{dfn}
\label{dfn:pm-map}
{\begin{enumerate}[label=(\arabic*),font=\normalfont]
\item
By \eqref{eq:sp-Gr-gp}, for each $x \in K^\ds(A)$, there exists a {unique finitely supported family}
$(a_L)_{L \in \calL} \in \bbZ^\calL$ such that
$x = \sum_{L \in \calL} a_L \br{L}$.
We set
\[
x_+\coloneqq \Ds_{\smat{L \in \calL\\a_L > 0}}L^{a_L}
\
\text{and}\
x_-\coloneqq \Ds_{\smat{L \in \calL\\a_L < 0}}L^{-a_L}
\]
and call them
the \emph{positive} and \emph{negative} \emph{parts} of $x$, {respectively. Both are $A$-modules.}
Then we have $x = \br{x_+} - \br{x_-}$.

\item
Let $f \colon \mod A \to \bbZ$ be invariant under isomorphisms.
Using the positive and negative parts, we define the corresponding
invariant on $K^\ds(A)$, denoted by the same letter $f$, by setting
$f(x)\coloneqq f(x_+) - f(x_-)$ for all $x \in K^\ds(A)$.

If $f$ is \emph{additive}, namely if $f(M \ds N)=f(M)+f(N)$
for all $M,N \in \mod A$, then the induced map is a group
homomorphism satisfying $f(\br{M})=f(M)$ for all $M \in \mod A$.
In this case, $f(x)=\sum_{L \in \calL}a_L f(L)$
for $x=\sum_{L \in \calL}a_L\br{L}$.
\end{enumerate}
}
\end{dfn}

\begin{dfn}
\label{dfn:int-repl}
Let $M \in \mod A$. We set
\[
\begin{aligned}
\de^\xi(M)_+ &\coloneqq
\bigoplus\limits_{\substack{I\in\bbI\\\de^\xi_M(I)>0}} {V_I}^{\de^\xi_M(I)},
\quad
\de^\xi(M)_- \coloneqq
\bigoplus\limits_{\substack{I\in\bbI\\\de^\xi_M(I)<0}} {{V_I}^{-\de^\xi_M(I)}},
 \text{ and}\\
\de^\xi(M)&\coloneqq \br{\de^\xi(M)_+} - \br{\de^\xi(M)_-}
\end{aligned}
\]
in $K^\ds(A)$.
We call $\de^\xi(M)$ the \emph{interval replacement} of $M$.
Note that $\de^\xi(M)$ is not a module, just an element of the split Grothendieck group, while both $\de^\xi(M)_+$ and $\de^\xi(M)_-$ are interval decomposable modules, and that $\de^\xi(M)$ can be presented by the pair
of these interval decomposable modules.
\end{dfn}

{The interval replacement $\de^\xi(M)$ is built from interval modules and is intended to be a substitute for $M$.
To make precise in what sense it approximates $M$, we compare their invariants.
Since the replacement is a formal difference of modules, we first extend interval ranks to the split Grothendieck group using \cref{dfn:pm-map} (2).
Their values on segments and singleton intervals recover the usual rank invariant and dimension vector, respectively.}

\begin{dfn}
\label{dfn:rank-de}
{For $X\in K^\ds(A)$ and $I\in\bbI$, we define
\[
\rank_I^\xi X\coloneqq \rank_I^\xi X_+-\rank_I^\xi X_-.
\]
We use the same construction for $\rank_{[x,y]}$ and, coordinatewise, for $\udim$.
By \cref{prp:AD}, the map $\rank_I^\xi$ is additive.
Thus \cref{dfn:pm-map} (2) gives
$\rank_I^\xi\de^\xi(M)=\sum_{J\in\bbI}\de_M^\xi(J)\rank_I^\xi V_J$
for all $M\in\mod A$.}
\end{dfn}

\begin{rmk}
\label{rmk:mult-subsumes}
{\cref{prp:IJ} gives $\rank_{[x,y]}V_J=\rank_{[x,y]}^\xi V_J$ for every interval $J$.
Since the positive and negative parts of $\de^\xi(M)$ are interval decomposable, additivity yields
\[
\rank_{[x,y]}\de^\xi(M)=\rank_{[x,y]}^\xi\de^\xi(M),
\quad
(\udim\de^\xi(M))(x)=\rank_{[x,x]}^\xi\de^\xi(M).
\]
Hence its interval rank invariant determines its usual rank invariant and dimension vector.}
\end{rmk}

{We now show that interval replacement preserves all interval ranks under the given compression system.
The usual rank invariant and dimension vector are therefore preserved as well.}

\begin{thm}
\label{thm:rank}
{Let $\xi$ be a compression system and $M\in\mod A$.
Then for every $I\in\bbI$, we have}
\[
{\rank_I^\xi\de^\xi(M)=\rank_I^\xi M.}
\]
{In particular, for every $[x,y]\in\Seg(\bfP)$, we have}
\[
{\rank_{[x,y]}\de^\xi(M)=\rank_{[x,y]}M,
\quad
\udim\de^\xi(M)=\udim M.}
\]
\end{thm}

\begin{proof}
{By \cref{dfn:rank-de} and \cref{prp:IJ,prp:cdelta}, we have
\begin{equation*}
\rank_I^\xi\de^\xi(M)=\sum_{J\in\bbI}\de_M^\xi(J)\rank_I^\xi V_J=\sum_{I\le J\in\bbI}\de_M^\xi(J)=\rank_I^\xi M.
\end{equation*}
The remaining assertions follow by taking $I=[x,y]$ and $I=[x,x]$,
using \cref{prp:seg-rank} and \cref{rmk:mult-subsumes}.}
\end{proof}

\subsection{{Computing signed interval multiplicities}}

{We now give a formula for $\de_M^\xi(I)$ using the covers of $I$ in $\bbI$. The following proposition ensures that the joins occurring in this formula are well defined.}

\begin{prp}
\label{prp:JOIN}
Let $U$ be a {nonempty} subset of $\bbI$.
If $U$ has a lower bound, then the join of $U$ exists.
\end{prp}

\begin{proof}
Let $I$ be a lower bound of $U$. Let us write ${U = \{I_1,\dots,I_n\}}$ with $n \ge 1$. Then the subset of $\bfP$ defined by $\bigcup_{k\in [n]} I_k$ is connected since $I \le I_k$ for all $k\in [n]$. It follows that $\conv \left( \bigcup_{k\in [n]} I_k \right)$ {is connected and convex and contains every $I_k$}, and hence it is an upper bound of $U$. Now, let $W$ be an upper bound of $U$. Since $\bigcup_{k\in [n]} I_k \subseteq W$ and $W$ is convex, we have $\conv \left( \bigcup_{k\in [n]} I_k \right) \subseteq W$. Thus $\conv \left( \bigcup_{k\in [n]} I_k \right) = \bigvee U$.
\end{proof}

{The following formula computes the M{\"o}bius function from subsets of covers. Its proof adapts the argument in \cite[Theorem 4.23]{ASASHIBA2023100007}.}

\begin{thm}
\label{thm:mu}
Let us define $\mu' \in \bbR\bbI$ by
\[
\mu'([I,J]) \coloneqq \sum_{S \in {\pE}}{(-1)^{\lvert S \rvert}},
\]
for $I, J \in \bbI$ with $I \le J$, {where} ${\pE}$ is the set of all sets $S$ such that $S \subseteq \Cov(I)$ and $\bigvee S = J$. Note that if $S$ is nonempty, then $\bigvee S$ is well defined by \cref{prp:JOIN}. {By convention, we set $\bigvee\varnothing\coloneqq I$ to simplify notation.} Then $\mu = \mu'$.
\end{thm}

\begin{proof}
Let us prove that $\zeta \mu' = 1_{\bbR\bbI}$. Let $I,J \in \bbI$ with $I \le J$. We have
\begin{multline*}
    (\zeta \mu')([I,J])
     = \sum_{I \le L \le J} \mu'([I,L])
     = \sum_{I \le L \le J} \ \sum_{S \in {\pE}} (-1)^{\lvert S \rvert} \\
     = \sum_{\substack{S \subseteq \Cov(I) \\ \bigvee S \le J}} (-1)^{\lvert S \rvert}
     = 1 - \sum_{\substack{\varnothing \neq S \subseteq \Cov(I) \\ \bigvee S \le J}} (-1)^{\lvert S \rvert - 1}.
\end{multline*}
Now define the function $f$ as follows:
\begin{align}
  f \colon 2^{\uset_{\bbI} I}\to \mathbb{R},\quad Z\mapsto \sum_{L \in Z} d_{V_J}(V_L),
\end{align}
where $2^{\uset_{\bbI} I}$ is the power set of $\uset_{\bbI} I$. Note that by definition $\bigcap_{L \in S} \uset_{\bbI} L = \uset_{\bbI} \bigvee S$. Therefore, we have
\begin{multline*}
    f \left(\bigcap_{L \in S} \uset_{\bbI} L \right)= f\left(\uset_{\bbI} \bigvee S\right)
     = \sum_{\bigvee S \le L} d_{V_J}(V_L) = {\rank_{\bigvee S}^{\xi} V_J}
     = \begin{cases}
        1, & \textnormal{if } \bigvee S \le J,\\
        0, & \textnormal{otherwise},
    \end{cases}
\end{multline*}
where the last two equalities come from \cref{prp:cd,prp:IJ}, respectively. Thus we can write:
\[
(\zeta \mu')([I,J]) = 1 - \sum_{\varnothing \neq S \subseteq \Cov(I)} (-1)^{\lvert S \rvert - 1} \ f \left(\bigcap_{L \in S} \uset_{\bbI} L\right).
\]
It is easily seen that $(\uset_{\bbI} I,2^{\uset_{\bbI} I},f)$ is a finite measure space. So, by the inclusion-exclusion principle
\begin{multline*}
    (\zeta \mu')([I,J])
     = 1 - f\left(\bigcup_{L \in \Cov(I)} \uset_{\bbI} L\right)
     = 1 - \sum_{I < L} d_{V_J}(V_L) = 1 - \left(\sum_{I \le L} d_{V_J}(V_L) - d_{V_J}(V_I)\right) \\
     = 1 - ({\rank_{I}^{\xi} V_J} - d_{V_J}(V_I))
     = d_{V_J}(V_I),
\end{multline*}
where the last two equalities also come from {\cref{prp:cd,prp:IJ}, respectively}. Finally, since $d_{V_J}(V_I) = 1$ if and only if $I = J$, we have $(\zeta \mu')([I,J]) = 1_{\bbR\bbI}([I,J])$, and hence $\zeta \mu' = 1_{\bbR\bbI}$.
Hence %
$\mu = \mu'$.
\end{proof}

\begin{cor}
\label{cor:dc2}
{Let $M\in\mod A$ and $I\in\bbI$. Then
\[
\de_M^\xi(I)=\sum_{S\subseteq\Cov(I)}(-1)^{|S|}\rank_{\bigvee S}^\xi M,
\]
where $\bigvee\varnothing\coloneqq I$. If $M$ is interval decomposable, the left-hand side is $d_M(V_I)$.}
\end{cor}

\begin{proof}
{By \cref{dfn:delta} and \cref{thm:mu}, we have
\begin{align*}
\de_M^\xi(I)=\sum_{I\le J}\rank_J^\xi M
  \sum_{\substack{S\subseteq\Cov(I)\\\bigvee S=J}}(-1)^{|S|}=\sum_{S\subseteq\Cov(I)}(-1)^{|S|}\rank_{\bigvee S}^\xi M.
\end{align*}
The last assertion follows from \cref{cor:dc}.}
\end{proof}

\begin{rmk}
\label{rmk:altdefdelta}
{\cref{prp:cdelta} allows us to compute the signed interval multiplicity $\de^\xi_M$ recursively, without explicitly computing the M{\"o}bius function. For every maximal interval $I$, we have $\de^\xi_M(I) = {\rank_{I}^{\xi} M}$. For the remaining intervals, we compute $\de^\xi_M(I) = {\rank_{I}^{\xi} M} - \sum_{I < J}\de^\xi_M(J)$ inductively, after the values for all $J > I$ have been computed.}
\end{rmk}

{For the W-shaped zigzag compression in \cite[Section~5.1]{ASASHIBA2023107397}, the signed interval multiplicities of a module {$M$} over a commutative ladder are the alternating sums of {the relative Betti numbers of $M$} in a minimal interval resolution \cite[Corollary~5.7]{ASASHIBA2023107397}.
Thus the corresponding interval ranks can also be recovered using \cref{prp:cdelta}; see \cite[Theorem~5.5]{ASASHIBA2023107397}.
By \cite[Corollary~5.9]{ASASHIBA2023107397}, these signed interval multiplicities give the homological invariant relative to the interval modules.
}

\section{A formula for the \texorpdfstring{$I$}{I}-rank}
\label{sec5}

{Our aim is to express $\rank_I^\xi M$ as a combination of ranks of matrices of structure linear maps (\cref{thm:general w/o conditions}).
We first construct the required projective presentations over a general finite poset $S$.
We then specialize to $S=I^\xi$ and $N=R_I(M)$ to obtain the interval-rank formulas, and simplify them when pairwise joins and meets exist.}

\subsection{{Projective presentations and multiplicities}}

{We recall the representation-theoretic tools used in the computation. Throughout the general constructions below, $S$ is a nonempty finite poset. We identify $\k[S]$-modules with $\k S$-modules as in \cref{sec2}.}

\begin{dfn}
\label{dfn:rad-top-soc}
Let $\bfr$ be the Jacobson radical of $\k S$.  Namely, in this case, we have
$$
\bfr =\Ds_{x < y \text{ in } S}\k p_{y,x}.
$$
Then for any left $\k S$-module $M$,
$\rad M\coloneqq \bfr M$ is called the \emph{radical} of $M$, $\top M\coloneqq M/\bfr M$
is called the \emph{top} of $M$, and $\soc M\coloneqq \{m \in M \mid \bfr m = 0\}$
is called the \emph{socle} of $M$.
{A nonzero module is called \emph{simple} if it has no proper nonzero submodules.}
{It is well known} that $\soc M$ is given by the sum of all simple submodules of $M$,
and both $\top M$ and $\soc M$ are semisimple{, i.e., direct sums of simple modules}.

For any $\k[S]$-module $V$, these are interpreted as follows:
$\bfr$ is the ideal of $\k[S]$ generated (or spanned) by all morphisms
$p_{y,x}$ with $x < y$ in $S$, and for any $x \in S$ we have
$$
(\rad V)(x) = \sum_{z < x} \Im V(p_{x,z}), \quad\text{and}\quad
(\soc V)(x) = \bigcap_{x < y} \Ker V(p_{y,x}).
$$
\end{dfn}

\begin{exm}
Let $\bfP = G_{5,2}$ as in \eqref{eq:G52}, and $M$ an interval module given
by the left diagram below:
\[
\begin{tikzcd}
	\k & \k & \k & \k & 0 \\
	0 & \k & \k & \k & \k
	\arrow["1", from=1-1, to=1-2]
	\arrow["1", from=1-2, to=1-3]
	\arrow["1", from=1-3, to=1-4]
	\arrow[from=1-4, to=1-5]
	\arrow[from=2-1, to=1-1]
	\arrow[from=2-1, to=2-2]
	\arrow["1", from=2-2, to=1-2]
	\arrow["1"', from=2-2, to=2-3]
	\arrow["1", from=2-3, to=1-3]
	\arrow["1"', from=2-3, to=2-4]
	\arrow["1", from=2-4, to=1-4]
	\arrow["1"', from=2-4, to=2-5]
	\arrow[""', from=2-5, to=1-5]
\end{tikzcd},
\quad
\begin{tikzcd}
	0 & \k & \k & \k & 0 \\
	0 & 0 & \k & \k & \k
	\arrow["", from=1-1, to=1-2]
	\arrow["1", from=1-2, to=1-3]
	\arrow["1", from=1-3, to=1-4]
	\arrow[from=1-4, to=1-5]
	\arrow[from=2-1, to=1-1]
	\arrow[from=2-1, to=2-2]
	\arrow["", from=2-2, to=1-2]
	\arrow[""', from=2-2, to=2-3]
	\arrow["1", from=2-3, to=1-3]
	\arrow["1"', from=2-3, to=2-4]
	\arrow["1", from=2-4, to=1-4]
	\arrow["1"', from=2-4, to=2-5]
	\arrow[""', from=2-5, to=1-5]
\end{tikzcd}.
\]
Then $\rad M$ is given by the right diagram above; and
$\top M$ and $\soc M$ are given by the left and the right
diagrams below, respectively:
\[
\begin{tikzcd}
	\k & 0 & 0 & 0 & 0 \\
	0 & \k & 0 & 0 & 0
	\arrow["", from=1-1, to=1-2]
	\arrow["", from=1-2, to=1-3]
	\arrow["", from=1-3, to=1-4]
	\arrow[from=1-4, to=1-5]
	\arrow[from=2-1, to=1-1]
	\arrow[from=2-1, to=2-2]
	\arrow["", from=2-2, to=1-2]
	\arrow[""', from=2-2, to=2-3]
	\arrow["", from=2-3, to=1-3]
	\arrow[""', from=2-3, to=2-4]
	\arrow["", from=2-4, to=1-4]
	\arrow[""', from=2-4, to=2-5]
	\arrow[""', from=2-5, to=1-5]
\end{tikzcd},
\quad
\begin{tikzcd}
	0 & 0 & 0 & \k & 0 \\
	0 & 0 & 0 & 0 & \k
	\arrow["", from=1-1, to=1-2]
	\arrow["", from=1-2, to=1-3]
	\arrow["", from=1-3, to=1-4]
	\arrow[from=1-4, to=1-5]
	\arrow[from=2-1, to=1-1]
	\arrow[from=2-1, to=2-2]
	\arrow["", from=2-2, to=1-2]
	\arrow[""', from=2-2, to=2-3]
	\arrow["", from=2-3, to=1-3]
	\arrow[""', from=2-3, to=2-4]
	\arrow["", from=2-4, to=1-4]
	\arrow[""', from=2-4, to=2-5]
	\arrow[""', from=2-5, to=1-5]
\end{tikzcd}.
\]
\end{exm}

\begin{rmk}
\label{rmk:prj-cov}
We can use the top (resp.\ socle) of $M$ in order to reduce the computation of a
projective cover (resp.\ injective hull) of $M$ to that of the semisimple module
$\top M$ (resp.\ $\soc M$) as follows:
Let $\pi \colon M \to \top M$ be the canonical epimorphism,
and $\si \colon \soc M \to M$ the inclusion.
{If} $f \colon P \to \top M$ is a projective cover of $\top M$,
then any lift $f' \colon P \to M$ of $f$ is a projective cover of $M$;
and if $g \colon \soc M \to E$ is an injective hull of $\soc M$,
then any extension $g' \colon M \to E$ is an injective hull of $M$.
Moreover, top and socle are used to compute almost split sequences
as seen in the proof of \cref{prp:n-m-ge-2 w/o conditions}.
\end{rmk}

\begin{ntn}
\label{ntn:Yoneda}
Let $S$ be a finite poset, $M \in \mod \k[S]$, and $x, y \in S$.
\begin{enumerate}
\item
{We denote $\k[S](x,\blank)$ by $P_x$ and $D(\k[S](\blank,x))$ by $I_x$.
Similarly, we denote $\k[S\op](x,\blank)$ by $P'_x$ and $D(\k[S\op](\blank,x))$ by $I'_x$.

Then $P_x$ and $I_x$ are, respectively, the projective and injective indecomposable $\k[S]$-modules corresponding to the vertex $x$.
The same holds for $P'_x$ and $I'_x$ over $\k[S\op]$.}

\item
By the Yoneda lemma, we have an isomorphism
\[
M(x) \to \Hom_{\k[S]}(P_x, M),\quad
m \mapsto \ro_m\ (m \in M(x)),
\]
{where $\ro_m \colon P_x \to M$ is the morphism in $\mod \k[S]$ with components $\ro_{m,y} \colon P_x(y) \to M(y)$ defined by}
\begin{equation}
\label{eq:Yoneda_corres}
	\ro_{m,y}(p)\coloneqq p\cdot m = M(p)(m)
\end{equation}
for all $y \in S$ and $p \in P_x(y) = \k[S](x,y)$. Sometimes we just write $\ro_{m}(p)\coloneqq M(p)(m)$ by omitting $y$.

Similarly, for the opposite poset $S\op$ of $S$,
by considering {a} $\k[S\op]$-module $N$ to be a right $\k[S]$-module,
we have an isomorphism
\[
N(x) \to \Hom_{\k[S\op]}(P'_x, N),\quad
m \mapsto \ro'_m\ (m \in N(x)),
\]
{where $\ro'_m \colon P'_x \to N$ is the morphism in $\mod \k[S\op]$ with components $\ro'_{m,y} \colon\allowbreak P'_x(y) \to N(y)$ defined by}
\begin{equation}
\label{eq:Yoneda_corres_op}
	\ro'_{m,y}(p)\coloneqq m\cdot p = N(p)(m)
\end{equation}
for all $y \in S$ and $p \in P'_x(y) = \k[S\op](x,y)$. Sometimes we just write $\ro'_{m}(p)\coloneqq N(p)(m)$ by omitting $y$.

\item
Since $p_{y,x} \in \k[S](x,y) = P_x(y)$,
we can set $\sfP_{y,x}\coloneqq \ro_{p_{y,x}} \colon P_y \to P_x$.
Similarly, we set $p\op_{x,y}\coloneqq p_{y,x} \in {\k[S\op]}(y,x) = {\k[S]}(x,y)$
for all $(x,y) \in [\le]_{{S}} $.
It induces a morphism $\sfP'_{x,y}\coloneqq {\ro'_{p\op_{x,y}}} \colon P'_x \to P'_y$ in $\mod \k[S\op]$.
\item
Let $B=\k$ or {$\k[S]$}, and {suppose that $V,W\in\mod B$ have decompositions $V=\Ds_{i\in[m]}V_i$ and $W=\Ds_{j\in[n]}W_j$,}
say with the canonical projections $\pi^V_i \colon V \to V_i$
(resp.\ $\pi^W_j \colon W \to W_j$) and the canonical injections
$\si^V_i \colon V_i \to V$ (resp.\ $\si^W_j \colon W_j \to W$).
Then recall that a morphism $f \colon W \to V$ in $\mod B$ {is expressed as an}
$m \times n$ matrix $f = [f_{i,j}]_{(i,j) \in [m]\times [n]}$,
where $f_{i,j}\coloneqq \pi^V_i \circ f \circ \si^W_j$ for all $(i,j) \in [m]\times [n]$.
Note that if $f' \colon W \to V$ is another morphism in $\mod B$,
then $f = f'$ if and only if $f_{i,j} = f'_{i,j}$ for all  $(i,j) \in [m]\times [n]$,
which justifies {the matrix notation} $f = [f_{i,j}]_{(i,j) \in [m]\times [n]}$.
Suppose further that $U \in \mod B$ is decomposed as
$U = \Ds_{h\in [l]} U_h$ and let $e \colon V \to U$ be in $\mod B$ with
a matrix expression $e = [e_{h,i}]_{(h,i) \in [l] \times [m]}$
with respect to these decompositions of $V, U$.
{The matrix of $ef$ is the product $[e_{h,i}]_{(h,i)}[f_{i,j}]_{(i,j)}$.}
We denote by ${}^t(\blank)$ the formal transpose:
${}^t[f_{i,j}]_{(i,j) \in [m]\times [n]}\coloneqq [f_{i,j}]_{(j,i) \in [n]\times [m]}$.
See \cref{lem:blockwise-transpose} and \cref{rmk:transpose}.
\end{enumerate}
\end{ntn}

For any finite poset $S$ and any $C, M \in \mod \k[S]$,
the following lemma makes it possible to compute
the dimension of $\Hom_{\k[S]}(C,M)$
by using a projective presentation of $C$ and the module structure of $M$.

\begin{lem}
\label{lem:dim-Hom-coker}
Let $S$ be a finite poset and $C, M \in \mod \k[S]$.
Assume that $C$ has a projective presentation
\[
\Ds_{j\in [n]} P_{y_j} \ya{{\ep}} \Ds_{i\in [m]} P_{x_i} \mathrel{\to} C
 \to 0\]
{where $\ep = [\ep_{i, j}]_{(i,j) \in [m]\times [n]}\coloneqq [a_{ji}\sfP_{y_j,x_i}]_{(i,j) \in [m]\times [n]}$ for some
$x_1, x_2,\dots, x_m,\allowbreak y_1,\allowbreak y_2,\allowbreak \dots, y_n \in S$ and $a_{ji} \in \k$. In this lemma, if $x_i\nleq y_j$, we interpret $\sfP_{y_j,x_i}$ and $M_{y_j,x_i}$ as the zero maps $P_{y_j}\to P_{x_i}$ and $M(x_i)\to M(y_j)$, respectively.}
Then we have
\[
\dim \Hom_{\k[S]}(C, M) = \sum_{i\in [m]} \dim M(x_i) -\rank{}^t\!\!\left([a_{ji}M_{y_j, x_i}]_{(i,j) \in [m]\times [n]}\right).
\]
\end{lem}

\begin{proof}
Set $Y\coloneqq \Ds_{j\in [n]} P_{y_j}, X\coloneqq \Ds_{i\in [m]} P_{x_i}$ for short.
Then we have an exact sequence $Y \ya{{\ep}} X \mathrel{\to} C \to 0$, which yields an exact sequence
\[
0 \to \Hom_{\k[S]}(C, M) \to \Hom_{\k[S]}(X, M) \ya{\Hom_{\k[S]}({\ep},M)} \Hom_{\k[S]}(Y,M).
\]
Hence $\Hom_{\k[S]}(C, M) \iso \Ker\Hom_{\k[S]}({\ep},M)$.
Now we have
\[
\begin{aligned}
&\Ker\Hom_{\k[S]}({\ep},M) = \{f \in \Hom_{\k[S]}(X,M) \mid f{\ep} = 0\}\\
&\iso \left.\left\{(f_1,\dots,f_m) \in \Ds_{i\in [m]} \Hom_{\k[S]}(P_{x_i},M) \right| (f_1,\dots,f_m)[a_{ji}\sfP_{y_j, x_i}]_{(i,j)} = 0 \right\}\\
&= \left.\left\{(f_1,\dots,f_m) \in \Ds_{i\in [m]} \Hom_{\k[S]}(P_{x_i},M) \right|
\left(\sum_{i\in [m]} a_{ji} f_i \sfP_{y_j, x_i}\right)_{j\in [n]} = 0 \right\}\\
&\iso \left.\left\{\sbmat{b_1\\\vdots\\b_m} \in \Ds_{i\in [m]} M(x_i) \right|
\left(\sum_{i\in [m]} a_{ji}M_{y_j, x_i}(b_i)\right)_{j\in [n]} = 0 \right\}\\
&= \left.\left\{\sbmat{b_1\\\vdots\\b_m} \in \Ds_{i\in [m]} M(x_i) \right| {}^t\!\!\left([a_{ji}M_{y_j, x_i}]_{(i,j)}\right)\sbmat{b_1\\\vdots\\b_m} = 0 \right\}\\
& = \Ker\left(\,{}^t\!\!\left([a_{ji}M_{y_j, x_i}]_{(i,j)}\right) \colon \Ds_{i\in [m]} M(x_i) \to \Ds_{j\in [n]} M(y_j)\right).
\end{aligned}
\]
Hence $\dim \Hom_{\k[S]}(C, M) = \sum_{i\in [m]} \dim M(x_i) -\rank{}^t\!\!\left([a_{ji}M_{y_j,x_i}]_{(i,j)\in [m]\times [n]}\right)$.
\end{proof}

The following theorem reduces the multiplicity of an indecomposable module to Hom dimensions.
Originally, it was proved under the condition that the base field $\k$ is algebraically closed.
However, the argument of the proof shows that the formula is valid 
if $\End_B (L)/\rad\End_B(L) \iso \k$ (this holds for example if $\k$ is algebraically closed).
Together with \cref{lem:dim-Hom-coker}, it explains which projective presentations are needed.

\begin{thm}[{cf.\ \cite[Theorem~3]{Asashiba2017}}]
\label{thm:Thm3-asashiba}
Let $B$ be a finite-dimensional algebra over $\k$, and $L, N \in \mod B$, where $L$ is indecomposable with $\End_B (L)/\rad\End_B(L) \iso \k$.
When $L$ is not injective, let
$$
0 \to L \to E_L \to \ta\inv L \to 0
$$
{be an almost split sequence starting at $L$, where $\ta\inv$ denotes the inverse Auslander--Reiten translation.}
Then we have
{\small
$$
d_N(L) = \begin{cases}
    \dim\Hom_B(L, N) - \dim\Hom_B(L/\soc L, N), & \textnormal{if } L \textnormal{ is injective},\\
    \begin{aligned}[t]
    &\dim\Hom_B(L, N) - \dim\Hom_B(E_L, N)\\
    &\qquad{}+ \dim\Hom_B(\ta\inv L, N),
    \end{aligned} & \textnormal{otherwise}.
\end{cases}
$$
}
\end{thm}

\removelastskip

\begin{rmk}
\label{rmk:presentation-strategy}
{For a connected poset $S$, we will apply \cref{thm:Thm3-asashiba} with $B=\k[S]$ and $L=V_S$, the interval module supported on all of $S$.
Since $S$ is connected, every endomorphism of $V_S$ acts by the same scalar at every vertex.
Hence $\End_{\k[S]}(V_S)\iso \k$,
so the hypothesis on the endomorphism division algebra is satisfied over an arbitrary field.
If $V_S$ is injective, we need projective presentations of $V_S$ and $V_S/\soc V_S$.
Otherwise, we need projective presentations of $V_S$, $\ta\inv V_S$, and the middle term $E$ of the almost split sequence starting at $V_S$.
Our presentations need not be minimal; we keep a common projective summand in the latter two terms, and the contribution of this common projective summand cancels in the multiplicity formula.}
\end{rmk}

\subsection{{Projective presentations over a finite poset}}
\label{subsec:presentations-finite-poset}

{Let $S$ be a nonempty finite poset. We denote by $V_S$ the $\k[S]$-module with $V_S(x)=\k$ and $V_S(p_{y,x})=\id_\k$ for every $x\le y$ in $S$.
When $S$ is connected, this is the interval module supported on the whole of $S$ and is indecomposable.
We use \cref{ntn:Yoneda} for $S$.}

{Sources index the degree-zero projective term, and minimal common upper bounds of pairs of sources index the degree-one term. We first specify these index sets.}

\begin{dfn}
\label{dfn:n-m-type-QI}
The poset $S$ is said to be of $(n,m)$-type if $|\src(S)|=n$ and $|\snk(S)|=m$, where $n,m$ are positive integers because $S \ne \varnothing$.
{We number the elements of $\src(S)$ and $\snk(S)$ by choosing bijections
$a\colon[n]\to\src(S)$, $i\mapsto a_i$, and
$b\colon[m]\to\snk(S)$, $j\mapsto b_j$.
The numberings are used to arrange the rows and columns of matrices and to fix the signs in \cref{ntn: notations for general case}.
They do not change the partial order on $S$.}
\end{dfn}

\removelastskip
\begin{dfn}
\label{dfn:pre-join}
Let $S$ be a finite poset and $a \ne b$ in $S$. The set
\[
\vee'\{a, b\}\coloneqq \src(\uset_S a \cap \uset_S b)\quad
(\text{resp.\ } \wedge'\{a, b\}\coloneqq \snk(\dset_S a \cap \dset_S b))
\]
of minimal upper (resp.\ maximal lower) bounds of $a$ and $b$ in $S$ is called the
\emph{pre-join} (resp.\ \emph{pre-meet}) of $\{a, b\}$.
It may be empty, and it is a singleton if and only if the join $a \vee b$ (resp.\ the meet $a \wedge b$) exists,
in which case $\vee'\{a, b\} = \{a \vee b\}$ (resp.\ $\wedge'\{a, b\} = \{a \wedge b\}$).
\end{dfn}

\begin{ntn}
\label{ntn: notations for general case}
Let $S$ be a nonempty finite poset, and use \cref{ntn:Yoneda} for $S$.
Number $\src(S) = \{a_1, \dots, a_n\}$ and $\snk(S) = \{b_1, \dots, b_m\}$ as in \cref{dfn:n-m-type-QI}.
\begin{enumerate}
\item
For a set $T$, we denote by $\sub_2 T$ the set of two-element subsets of $T$.
For $\bfa = \{a_i, a_j\} \in \sub_2 \src(S)$ with $i < j$, we write
$\udl{\bfa}\coloneqq a_i$ and $\ovl{\bfa}\coloneqq a_j$; similarly, $\udl{\bfb}\coloneqq b_i$ and $\ovl{\bfb}\coloneqq b_j$
for $\bfb = \{b_i, b_j\} \in \sub_2 \snk(S)$ with $i < j$.

\item
We set
\[
\begin{aligned}
\src_1(S)&\coloneqq \{\bfa_c\coloneqq (\bfa, c) \mid \bfa \in \sub_2 \src(S),\ c \in \vee'\bfa\},\\
\snk_1(S)&\coloneqq \{\bfb_d\coloneqq (\bfb, d) \mid \bfb \in \sub_2 \snk(S),\ d \in \wedge'\bfb\},
\end{aligned}
\]
where the pre-join $\vee'\bfa$ and the pre-meet $\wedge'\bfb$ are taken in the poset $S$.
Note that no element of $\vee'\bfa$ lies in $\src(S)$, and that the same $c$ may occur
in several $\bfa_c$, which is why $c$ is recorded together with $\bfa$.
The sets $\src(S)$ and $\src_1(S)$ index the two projective terms in \cref{prp:prjpres-VI w/o conditions};
$\snk(S)$ and $\snk_1(S)$ give the dual index sets over $\k[S\op]$; see (3).

\item
For $X$ one of the sets $\src(S)$, $\src_1(S)$, $\snk(S)$, $\snk_1(S)$, we set
$P_X\coloneqq \Ds_{t \in X} P_t$ and $P'_X\coloneqq \Ds_{t \in X} P'_t$, where $P_{\bfa_c}\coloneqq P_c$ and $P_{\bfb_d}\coloneqq P_d$,
and likewise for $P'$. Thus $P_X = 0$ if $X = \varnothing$, and $P_{\src_1(S)}$ contains
one copy of $P_c$ for each $\bfa_c \in \src_1(S)$.
\end{enumerate}
Number the elements of each pre-join and pre-meet in (2) arbitrarily.
For $\bfa_c=(\{a_i,a_j\},c)$ with $i<j$, we order $\src_1(S)$ lexicographically by $(i,j,c)$,
where the third coordinates are compared only for the same pair $(i,j)$, using the order induced by the numbering of $\vee'\{a_i,a_j\}$.
We denote this order by $\plex$ and order $\snk_1(S)$ similarly.
Matrices indexed by these sets follow the source and sink numberings and these lexicographic orders.
Changing these numberings amounts to permuting block rows and block columns of $\tdbfM$ and $\hatbfM$ in \cref{thm:n-1-case w/o conditions,thm:1-n-case w/o conditions} and possibly changing their signs.
Their ranks are therefore unchanged.
\end{ntn}

\begin{exm}
\label{exm:src1-illustration}
Let $\bfP\coloneqq G_{3,3}$ and take the interval $I\coloneqq \bfP\setminus\{(1,1),(3,3)\}$.
Number its sources by $a_1\coloneqq (1,2)$, $a_2\coloneqq (2,1)$ and its sinks by $b_1\coloneqq (2,3)$, $b_2\coloneqq (3,2)$.
The source-sink compression gives the full subposet $I^\ss=\{a_1,a_2,b_1,b_2\}$, as shown below:
\[
I =
\begin{tikzcd}[ampersand replacement=\&, row sep=small, column sep=small]
(1,3) \& (2,3) \& \\
(1,2) \& (2,2) \& (3,2) \\
\& (2,1) \& (3,1)
\arrow[from=2-1, to=1-1] \arrow[from=1-1, to=1-2] \arrow[from=2-1, to=2-2]
\arrow[from=2-2, to=1-2] \arrow[from=2-2, to=2-3] \arrow[from=3-2, to=2-2]
\arrow[from=3-2, to=3-3] \arrow[from=3-3, to=2-3]
\end{tikzcd},
\qquad
I^{\ss} =
\begin{tikzcd}[ampersand replacement=\&, row sep=small, column sep=small]
\& (2,3) \& \\
(1,2) \& \& (3,2) \\
\& (2,1) \&
\arrow[from=2-1, to=1-2] \arrow[from=2-1, to=2-3]
\arrow[from=3-2, to=1-2] \arrow[from=3-2, to=2-3]
\end{tikzcd}.
\]
Put $\bfa\coloneqq \{a_1,a_2\}$ and $\bfb\coloneqq \{b_1,b_2\}$.
In $I^\ss$, we have $\vee'\bfa=\{b_1,b_2\}$ and $\wedge'\bfb=\{a_1,a_2\}$,
whereas $a_1\vee a_2=(2,2)$ in $I$.

Using \cref{ntn: notations for general case} with $S=I^\ss$,
place $b_1$ before $b_2$ on $\vee'\bfa$ and $a_1$ before $a_2$ on $\wedge'\bfb$.
Since there is only one source pair and one sink pair, the induced lexicographic orders are
$\bfa_{b_1}\plex\bfa_{b_2}$ and $\bfb_{a_1}\plex\bfb_{a_2}$, respectively.
The index sets and the corresponding projective terms are
\[
\begin{aligned}
\src_1(I^{\ss}) &= \{\bfa_{b_1}, \bfa_{b_2}\},&
P_{\src_1(I^{\ss})} &= P_{b_1} \ds P_{b_2},\\
\snk_1(I^{\ss}) &= \{\bfb_{a_1}, \bfb_{a_2}\},&
P'_{\snk_1(I^{\ss})} &= P'_{a_1} \ds P'_{a_2},
\end{aligned}
\]
where $\bfa_{b_j}=(\bfa,b_j)$, and similarly for $\bfb_{a_i}$.
\end{exm}

{For connected $S$ of $(n,m)$-type, we have $\top V_S\iso\Ds_{i\in[n]}V_{\{a_i\}}$ and $\soc V_S\iso\Ds_{j\in[m]}V_{\{b_j\}}$.
Thus $V_S$ is projective exactly when $n=1$, and injective exactly when $m=1$.
Connectedness also ensures that the index sets $\src_1(S)$ and $\snk_1(S)$ are nonempty when $n\ge2$ and $m\ge2$, respectively.}

\begin{prp}
\label{prp:nonempty-relation-indices}
{Let $S$ be a nonempty finite connected poset of $(n,m)$-type.
If $n\ge2$, then $\src_1(S)\ne\varnothing$; if $m\ge2$, then $\snk_1(S)\ne\varnothing$.}
\end{prp}

\begin{proof}
{Suppose that $n\ge2$ and $\src_1(S)=\varnothing$.
The nonempty upsets $\uset_S a_i$ cover $S$, and any two of them are disjoint: otherwise their intersection, being finite, has a minimal element and gives an element of $\src_1(S)$.
No two elements in distinct upsets are comparable, since an element above a point of $\uset_S a_i$ also belongs to $\uset_S a_i$.
This contradicts the connectedness of $S$.
The assertion about $\snk_1(S)$ follows by applying the same argument to $S\op$.}
\end{proof}

\subsubsection{{A projective presentation of \texorpdfstring{$V_S$}{VS}}}

{We now use the index sets $\src(S)$ and $\src_1(S)$ to construct a projective presentation of $V_S$.}

\begin{samepage}
\begin{prp}
\label{prp:prjpres-VI w/o conditions}
{Let $S$ be a nonempty finite poset, and use \cref{ntn:Yoneda,ntn: notations for general case} for $S$.}
Then we have the following {(not necessarily minimal)} projective presentation of $V_{{S}}$ in $\mod \k[{S}]$:
\begin{equation}
\label{eq:min-proj-resol-V_I-nsrccase-lem w/o conditions}
P_{\src_1({S})} \ya{\ep_{1}} P_{\src({S})} \ya{\ep_{0}} V_{{S}} \to 0.
\end{equation}
Here $\ep_{0}, \ep_{1}$ are given by
\begin{align}\label{eq:matrix form of ep_0 w/o conditions}
    \ep_{0}\coloneqq (\ro_{1_{a_{1}}}, \ro_{1_{a_{2}}}, \dots, \ro_{1_{a_{n}}}),
\end{align}
where we set
$1_u \coloneqq 1_\k \in \k = V_{{S}}(u)$ for all $u \in {S}$, and
\begin{align}\label{eq:matrix form of ep_1 w/o conditions}
\ep_{1}\coloneqq
\bmat{\tilde{\sfP}_{a,\bfa_c}}_{(a,\bfa_c) \in \src({S}) \times \src_1({S})},
\end{align}
where the entry is given by
\begin{align}\label{eq:entries of ep_1 w/o conditions}
\tilde{\sfP}_{a,\bfa_c}\coloneqq
\begin{cases}
    \sfP_{c,a}, & \textnormal{if } a = \udl{\bfa},\\
    -\sfP_{c,a}, & \textnormal{if } a = \ovl{\bfa},\\
    \mathbf{0}, & \textnormal{if } a \not\in \bfa,
\end{cases}
\end{align}
for all $\bfa_c \in \src_1({S})$ and $a \in \src({S})$. {Here and subsequently, the row and column indices follow the numberings and lexicographic orders in \cref{ntn: notations for general case}.}
\end{prp}
\end{samepage}

\begin{proof}
{If $n=1$, then $\src_1(S)=\varnothing$ and $\ro_{1_{a_1}}\colon P_{a_1}\to V_S$ is an isomorphism, so the assertion holds. Assume that $n\ge2$.}
We verify {the exactness of the sequence} \eqref{eq:min-proj-resol-V_I-nsrccase-lem w/o conditions} in the following steps:

(a) showing that $\ep_{0}$ is surjective;
(b) showing that $\ep_{0}\ep_{1} = 0$;
(c) showing that $\dim\Im\ep_{1} \geq   \dim\Ker\ep_{0}$.

(a) It is enough to show that for all {$x\in {S}$},
\[
\left(\ep_{0}\right)_{x} = \bmat{\left(\ro_{1_a}\right)_{x}}_{a\in \src({S})}\colon \Ds_{i\in [n]}P_{a_{i}}(x)\to V_{{S}}(x)
\]
is surjective.
{For any $x\in S$, the space $V_S(x)=\k$ has basis $\{1_x\}$. Thus} it suffices to show that $1_{x}$ has a preimage in $\Ds_{i\in [n]}P_{a_{i}}(x)$. In this case, it is evident that there is $a_{j}\in \src ({S})$ such that $a_j\leq x$. Set $\left(m_{i}\right)_{i\in [n]}\in \Ds_{i\in [n]}P_{a_{i}}(x)$ with all entries zero except for the $j$-th entry $m_j \coloneqq p_{x, a_{j}}$. Then
\[
\left(\ep_{0}\right)_{x} \left(\left(m_{i}\right)_{i\in [n]}\right)= \sum_{i\in [n]}\left(\ro_{1_{a_{i}}}\right)_{x} \left(m_{i}\right) = \left(\ro_{1_{a_{j}}}\right)_{x} \left(m_{j}\right) = 1_{x},
\]
{where the last equality follows from the Yoneda correspondence \eqref{eq:Yoneda_corres}}:
\[
\left(\ro_{1_{a_{j}}}\right)_{x}\colon P_{a_{j}}(x) \to V_{{S}}(x),\quad
m_j = p_{x, a_{j}} \mapsto p_{x, a_{j}}\cdot 1_{a_{j}} = V_{{S}}(p_{x, a_{j}})(1_{a_{j}}) = 1_{x}.
\]
Therefore for all $x\in {S}$, $1_{x}$ has the preimage and thus $\ep_{0}$ is an epimorphism.

(b) It suffices to show that the composition
\begin{equation}
\label{eq:compositionzero-lem w/o conditions}
P_{\bfa_c} \ya{\ep'_{1}} \bigoplus_{i\in[n]} P_{a_{i}} \ya{\ep_{0}} V_{{S}}
\end{equation}
is zero for all $\bfa_c \in \src_1({S})$,
where
\[
\ep'_{1}\coloneqq
{}^t\!\bmat{
    \tilde{\sfP}_{a_1,\bfa_c} & \cdots &
    \tilde{\sfP}_{a_n,\bfa_c}
}
\]
is the $\bfa_c$-column of $\ep_1$.
Its only nonzero entries are $\sfP_{c,\udl{\bfa}}$
and $-\sfP_{c,\ovl{\bfa}}$ in the rows indexed by
$\udl{\bfa}$ and $\ovl{\bfa}$, respectively.
Let $c\leq t\in {S}$. Then
\begin{align*}
(\ep_{0} \ep'_{1})(p_{t,c})
& =
(\ro_{1_{\udl{\bfa}}}\sfP_{c,\udl{\bfa}})(p_{t,c})- (\ro_{1_{\ovl{\bfa}}} \sfP_{c,\ovl{\bfa}})(p_{t,c})\\
&=
\ro_{1_{\udl{\bfa}}}(p_{t,\udl{\bfa}})
- \ro_{1_{\ovl{\bfa}}} (p_{t,\ovl{\bfa}})\\
& =
V_{{S}}(p_{t,\udl{\bfa}})(1_{\udl{\bfa}})
- V_{{S}}(p_{t,\ovl{\bfa}})(1_{\ovl{\bfa}})\\
&= 1_t - 1_t = 0.
\end{align*}

(c) It suffices to show that
$\dim\Im\, (\ep_{1})_x \geq \dim\Ker\, (\ep_{0})_x$
for all $x\in {S}$.
Fix $x\in {S}$, and let
$w\coloneqq |\src({S})\cap \dset_{{S}} x|$. Then $w \ge 1$.
{After renumbering the sources, we may assume that $\{a_1,\dots,a_w\}=\src(S)\cap\dset_S x$.}

{\bf Case 1.}
Consider the case where $w = 1$. If $P_{\src_1({S})}(x) \ne 0$, then there exists some
$\bfa_c \in \src_1({S})$
such that $P_{\bfa_c}(x) \ne 0$, which shows that
$\udl{\bfa} \le c,\, \ovl{\bfa} \le c,\, c \le x$.
Thus $w \ge |\bfa| = 2$, a contradiction.
Therefore in this case, $P_{\src_1({S})}(x) = 0$.
Hence the evaluation of
\eqref{eq:min-proj-resol-V_I-nsrccase-lem w/o conditions} at $x$ becomes
\begin{equation*}
0 \to \k p_{x,a_1} \ya{\al} \k \to 0,
\end{equation*}
where $\al$ is an isomorphism defined by
$\al(p_{x,a_1})\coloneqq 1_\k$.
The claim follows since
$\Im\, (\ep_{1})_x = 0 = \Ker\, (\ep_{0})_x$.

{\bf Case 2.}
Consider the case where $w \ge 2$. By \cref{ntn: notations for general case},
$P_{\src({S})}(x)
= \bigoplus_{a \in \src({S})} P_a(x)
= \bigoplus_{a \in \src({S}) \cap \dset_{{S}} x}\k p_{x,a}$,
and
$$
P_{\src_1({S})}(x)
= \bigoplus_{\bfa_c \in \src_1({S})} P_{\bfa_c}(x)
= \bigoplus_{\smat{\bfa\in \sub_2\src({S})\\
c \in \vee'\bfa \cap \dset_{{S}} x}} P_{c}(x)
= \bigoplus_{\smat{\bfa\in \sub_2\src({S})\\
c \in \vee'\bfa \cap \dset_{{S}} x}} \k p_{x,c}.
$$
Then the evaluation of
\eqref{eq:min-proj-resol-V_I-nsrccase-lem w/o conditions} at $x$ becomes
the first row of the diagram
$$
\begin{tikzcd}
\Nname{P1}
\DDs_{\smat{\bfa\in \sub_2\src({S})\\
c \in \vee'\bfa \cap \dset_{{S}} x}} \k p_{x,c}
& \Nname{P0}
\DDs_{a \in \src({S}) \cap \dset_{{S}} x}\k p_{x,a}
& \Nname{k}\k  & \Nname{0} 0\\
\Nname{k1}
\DDs_{\smat{\bfa\in \sub_2\src({S})\\
c \in \vee'\bfa \cap \dset_{{S}} x}} \k
& \Nname{k0}
\DDs_{a \in \src({S}) \cap \dset_{{S}} x}\k
& \Nname{k'}\k
\Ar{P1}{P0}{"(\ep_1)_x"}
\Ar{P0}{k}{"(\ep_0)_x"}
\Ar{k1}{k0}{"{\ep_1^x}"}
\Ar{k0}{k'}{"{\ep_0^x}"}
\Ar{k}{0}{}
\Ar{k1}{P1}{"\al"}
\Ar{k0}{P0}{"\be"}
\Ar{k'}{k}{equal}
\end{tikzcd},
$$
where $\al = \Ds_{\bfa_c}\al_c$ (resp.\ $\be = \Ds_{a}\be_a$) is the isomorphism defined by $\al_c(1_\k)\coloneqq p_{x,c}$ (resp.\ $\be_a(1_\k)\coloneqq p_{x,a}$)
for all $\bfa \in \sub_2\src({S}),\, c \in \vee'\bfa \cap \dset_{{S}} x$
(resp.\ $a \in \src({S}) \cap \dset_{{S}} x$).
{We define the linear maps $\ep_1^x$ and $\ep_0^x$ above
in such a way that the diagram commutes.}
Then we can compute their entries
by the equality
$\sfP_{a,b}(p_{x,a}) = p_{x,b}$ {for $b\le a\le x$ in $S$}.

We note here that
\begin{equation}
\label{eq:cap-x}
\src(\uset_{{S}}\udl{\bfa} \cap \uset_{{S}}\ovl{\bfa} \cap \dset_{{S}} x) \subseteq \vee'\bfa \cap \dset_{{S}} x
\end{equation}
holds for all $\bfa \in \sub_2\src({S})$.
Indeed,
assume that $y$ is an element of the left-hand side.
Let $z < y$ in ${S}$. Then by this assumption,
$z \not\in \uset_{{S}}\udl{\bfa} \cap \uset_{{S}}\ovl{\bfa} \cap \dset_{{S}} x$.
Since $z < y \le x$, we have $z \in \dset_{{S}} x$.
Thus $z \not\in \uset_{{S}}\udl{\bfa} \cap \uset_{{S}}\ovl{\bfa}$.
Hence $y$ {belongs to the set on the right-hand side}, as required.

Since $w \ge 2$,
$\sub_2(\src({S}) \cap \dset_{{S}}  x) \ne \varnothing$.
Note that for each $\bfa \in \sub_2(\src({S}) \cap \dset_{{S}}  x)$,
we have $\uset_{{S}}\udl{\bfa} \cap \uset_{{S}}\ovl{\bfa} \cap \dset_{{S}} x
\ne \varnothing$ because the left-hand side contains $x$.
Hence
$\src(\uset_{{S}}\udl{\bfa} \cap \uset_{{S}}\ovl{\bfa} \cap \dset_{{S}} x)
\ne \varnothing$.
Take any $c_{\bfa,x}$ from this set.
Then by \eqref{eq:cap-x}, we have
$c_{\bfa,x} \in \vee'\bfa \cap \dset_{{S}} x$.
The pair $\bfa_{c_{\bfa,x}} = (\bfa, c_{\bfa,x})$ is denoted by $\bfa_{c_x}$ for short.
{With columns in the lexicographic order of \cref{ntn: notations for general case},
the matrix $\ep_{1}^x$ has a $w\times\binom{w}{2}$ submatrix $\tilde{\ep}_{1}^x$ given by}
\begin{align}\label{ep1matrixform w/o conditions}
\begin{blockarray}{cccccccccc}
& \scalebox{0.7}{$\{a_{1},a_{2}\}_{c_{x}}$} & \scalebox{0.7}{$\{a_{1},a_{3}\}_{c_{x}}$} & \scalebox{0.7}{$\cdots$} & \scalebox{0.7}{$\{a_{1},a_{w}\}_{c_{x}}$} & \scalebox{0.7}{$\{a_{2},a_{3}\}_{c_{x}}$} & \scalebox{0.7}{$\cdots$} & \scalebox{0.7}{$\{a_{2},a_{w}\}_{c_{x}}$} & \scalebox{0.7}{$\cdots$} & \scalebox{0.7}{$\{a_{w-1},a_{w}\}_{c_{x}}$} \\
\begin{block}{c[ccccccccc]}
  \scalebox{0.7}{$a_{1}$} & 1 & 1  & \cdots & 1 & 0 & \cdots & 0 &\cdots & 0 \\
  \scalebox{0.7}{$a_{2}$} & -1  & 0 & \cdots & 0 & 1 & \cdots & 1 & \cdots & 0 \\
  \scalebox{0.7}{$a_{3}$} & 0 & -1 & \cdots & 0 & -1 & \cdots & 0 & \cdots & 0 \\
  \scalebox{0.7}{$\vdots$} & \vdots & \vdots & & \vdots & \vdots & & \vdots & & \vdots \\
  \scalebox{0.7}{$a_{w-1}$} & 0 & 0 & \cdots & 0 & 0 & \cdots & 0 & \cdots & 1 \\
  \scalebox{0.7}{$a_{w}$} & 0 & 0 & \cdots & -1 & 0 & \cdots & -1 & \cdots & -1 \\
\end{block}
\end{blockarray}
\end{align}
and the matrix $\ep_{0}^x$ is given by $(\underbrace{1, 1, {\dots}, 1}_{w})$. It is clear that $\rank\ep_{0}^x = 1$.
For the matrix $\tilde{\ep}_1^x$,
note that the last $w-1$ rows are linearly independent, and that the sum of all rows is a zero row vector. This shows that $\rank\tilde{\ep}_{1}^x = w-1$. Thus $\dim\Im\ep_{1}^x = \rank\ep_{1}^x \ge \rank{\tilde{\ep}_{1}^x} = w - 1 = \dim\Ker\ep_{0}^x$.
\end{proof}

\begin{rmk}
\label{rmk:prjpres-general-poset}
{Let $U$ be a nonempty upset of a finite poset $S$.
Extension by zero outside $U$ preserves exactness and sends $\k[U](u,\blank)$ to $\k[S](u,\blank)$ for $u\in U$.
Thus extension by zero takes the presentation of $V_U$ over $\k[U]$ in \cref{prp:prjpres-VI w/o conditions} to a projective presentation over $\k[S]$.
This construction is used in \cite[Propositions~3.13 and~3.14]{asashiba2024interval}; Proposition~3.18 there treats arbitrary intervals of finite posets.}
\end{rmk}

\begin{rmk}[Relations with previously known results]
\label{rmk:known-constructions}
{Let $S$ be a nonempty finite poset with $\src(S)=\{a_1,\dots,a_n\}$, and use \cref{ntn:Yoneda} for each poset below.
We compare the presentation of $V_S$ in \cref{prp:prjpres-VI w/o conditions} with two Koszul-type constructions.

(1) Suppose that $S$ is an upset of a finite lattice $\bfL$.
Extending modules by zero outside $S$ identifies each $P_x$, for $x\in S$, with the corresponding projective $\k[\bfL]$-module.
For $f\colon[n]\to\bfL$, $i\mapsto a_i$, \cite[Proposition~5.1]{botnanBottleneckStabilityRank2024} gives a resolution whose first two projective terms form the presentation
\[
\bigoplus_{1\le i<j\le n}P_{a_i\vee a_j}
\ya{\ep_1}\bigoplus_{i=1}^{n}P_{a_i}
\ya{\ep_0}V_S\to0.
\]
Here $\ep_0=(\ro_{1_{a_1}},\dots,\ro_{1_{a_n}})$, where $1_{a_i}\coloneqq 1_\k\in V_S(a_i)$, and for $c=a_i\vee a_j$, the component $P_c\to P_{a_i}\ds P_{a_j}$ of $\ep_1$ is $\sbmat{\sfP_{c,a_i}\\-\sfP_{c,a_j}}$ up to sign.
Since $\vee'\{a_i,a_j\}=\{a_i\vee a_j\}$, this is the presentation in \cref{prp:prjpres-VI w/o conditions}.

(2) Suppose that every nonempty subset of $\src(S)$ bounded above in $S$ has a join in $S$, and adjoin a new minimum $0$ to form $\widehat S\coloneqq S\sqcup\{0\}$.
Apply the Koszul construction of Chach{\'o}lski et al.\ \cite[Section~3.7]{chacholski2025koszul} at $0$ in the opposite poset $\widehat S\op$ to the representable functors $P'_x=\k[\widehat S](\blank,x)$, for $x\in\widehat S$.
In $\widehat S\op$, the elements strictly below $0$ with no element between them and $0$ are precisely the sources of $S$, and the meets of their nonempty subsets bounded below are their joins in $S$.
By \cite[Theorem~3.8]{chacholski2025koszul} and the projectivity of $P'_x$, these complexes, as $x$ varies, form a projective left $\k[\widehat S]$-module resolution of $\top P_0$ with degree-zero term $P_0$.
Its first syzygy $\rad P_0$ is $V_S$ extended by zero at $0$; the induced presentation restricts to (1), with the same signed maps and the sum over $i<j$ restricted to pairs with a common upper bound in $S$.

(3) Without the join assumption in (2), a pair of sources may have several minimal upper bounds.
\cref{prp:prjpres-VI w/o conditions} uses one summand $P_c$ for each $c\in\vee'\{a_i,a_j\}$, with the same component maps; pairs without a common upper bound contribute no summand.
Thus it recovers the preceding presentations when the relevant joins exist and remains valid without them.
These presentations need not be minimal, and this comparison does not assert an extension of the full resolutions to arbitrary finite posets.}
\end{rmk}

\removelastskip

\subsubsection{{The quotient by the socle}}

{For this construction, assume that $S$ is of $(n,1)$-type with $n\ge2$ and $\snk(S)=\{b\}$.
Then $b$ is the maximum of $S$, so $V_S\iso I_b$ is injective. To use the first case of \cref{thm:Thm3-asashiba}, we construct a projective presentation of $V_S/\soc V_S$.
Let $\pi\colon V_S\to V_S/\soc V_S$ be the canonical projection.}

\begingroup
The following lemma is immediate from \cref{lem:intv-as-K-L}.

\begin{lem}
\label{lem:topandatom}
For each $x\in S$, there exists $a_{i}\in \src(S)$ such that $a_{i}\leq x\leq b$. \qed
\end{lem}

Define a morphism ${\la\colon} P_{b}\to P_{\src(S)}$ by setting
$$
\la\coloneqq {}^t\!\bmat{\sfP_{b,a_{1}} &
\underbrace{\mathbf{0}\ \cdots\ \mathbf{0}}_{n-1}}.
$$
{Since $\soc V_S$ is simple and $\ep_0(P_{a_1})\ne0$,}
we have $\soc V_S \subseteq \ep_0(P_{a_1})$.
Moreover there exists an isomorphism $P_{b} \to \soc V_S$.
Since $P_b$ is a projective $\k[S]$-module, the composite
$P_{b} \to \soc V_S \hookrightarrow \ep_0(P_{a_1})$
factors through the epimorphism $\ep_0|_{P_{a_1}} \colon P_{a_1} \to \ep_0(P_{a_1})$.
Hence $\ep_0(\la({p_{b,b}})) = \ep_0(p_{b,a_1})$
is a nonzero element of the simple module $\soc V_S$.
Thus 
we have $\Im (\ep_{0}\circ \la)=\soc V_S$,
which shows the exactness of the right column in diagram \eqref{eq:snake-lem-nsrc} below.

\begin{prp}
\label{prp:quotient-socle-presentation}
Assume that $S$ is a nonempty finite poset of $(n,1)$-type with $n \ge 2$
and $\snk(S) = \{b\}$, and use the maps $\la$ and $\pi$ above.
Then we have the following projective presentation of $V_S/\soc V_S$ in $\mod \k[S]$:
\begin{equation}
\label{eq:min-proj-resol-V_I/socV_I-nsrccase w/o conditions}
P_{\src_1(S)}\ds P_{b} \ya{\ep'_{1}} P_{\src(S)} \ya{\ep'_{0}} V_S/\soc V_S \to 0,
\end{equation}
where $\ep'_{1}$, $\ep'_{0}$ {are} given by
\[
\ep'_{1}\coloneqq \bmat{\ep_{1}\ \la},\ \ep'_{0}\coloneqq \pi \circ \ep_{0}.
\]
\end{prp}

\begin{proof}
Since the first and second {terms of the sequence}
\eqref{eq:min-proj-resol-V_I/socV_I-nsrccase w/o conditions}
are projective, it is enough to show the exactness of this sequence.
We start {from the projective presentation} \eqref{eq:min-proj-resol-V_I-nsrccase-lem w/o conditions} of $V_S$, and express $\ep_1$ as the composite
$\ep_1 = \iota_0 \ta_1$ {as in the following commutative diagram:}
\[
\begin{tikzcd}
P_{\src_1(S)} && P_{\src(S)} & V_S & 0, \\
& \Ker\ep_{0}\\
\Ar{1-1}{1-3}{"\ep_{1}"}
\Ar{1-3}{1-4}{"{\ep_{0}}"}
\Ar{1-4}{1-5}{}
\Ar{2-2}{1-3}{"\iota_{0}"', tail}
\Ar{1-1}{2-2}{"{\tau_{1}}"', two heads}
\end{tikzcd}
\]
where $\iota_0$ is the kernel of $\ep_0$ and
$\ta_1$ is an epimorphism obtained from $\ep_1$ by restricting the codomain.
This yields the diagram
\[
\begin{tikzcd}
P_{\src_1(S)} \ds P_b && P_{\src(S)} & V_S/\soc V_S & 0, \\
& \Ker\ep_{0} \ds P_b\\
\Ar{1-1}{1-3}{"\ep'_{1}"}
\Ar{1-3}{1-4}{"{\ep'_{0}}"}
\Ar{1-4}{1-5}{}
\Ar{2-2}{1-3}{"\iota'_{0}"', tail}
\Ar{1-1}{2-2}{"{\tau'_{1}}"', two heads}
\end{tikzcd}
\]
where $\ta'_1\coloneqq \sbmat{\tau_{1}\ \mathbf{0}\\ \mathbf{0}\ \id}$ and
$\iota'_{0}\coloneqq \bmat{\iota_{0}\ \la}$.
Then $\ta'_1$ is an epimorphism because so is $\ta_1$, and
the diagram is commutative.
Indeed,
$\iota'_{0} \circ \tau'_{1} = \bmat{\iota_{0}\ \la}\sbmat{\tau_{1}\ \mathbf{0}\\ \mathbf{0}\ \id} = \bmat{\iota_{0}\tau_{1}\ \la} = \bmat{\ep_{1}\ \la} = \ep'_{1}$.
Therefore, it remains to show that the sequence
\begin{equation}
\label{eq:necess-ex}
0\to \Ker\ep_{0}\ds P_{b}\ya{\iota'_{0}} P_{\src(S)}\ya{\ep'_{0}} V_S/\soc V_S\to 0
\end{equation}
is exact.
Consider the following commutative diagram of solid arrows
with exact rows surrounded by dashed lines:
\begin{equation}
\label{eq:snake-lem-nsrc}
\begin{tikzcd}[ampersand replacement=\&]
0 \& 0 \& \Ker\iota'_{0}\& 0\\
\Nname{lu}0 \& \Ker\ep_{0} \& \Ker\ep_{0} \ds P_b \& P_b \&\Nname{ru} 0\\
\Nname{ld}0 \& \Ker\ep_{0} \& P_{\src(S)} \& V_S \& \Nname{rd}0\\
\& 0 \& \Cok \iota'_{0} \& V_S/\soc V_S \& 0
\Ar{1-1}{1-2}{}
\Ar{1-2}{1-3}{}
\Ar{1-3}{1-4}{}
\Ar{2-1}{2-2}{}
\Ar{2-2}{2-3}{"{\sbmat{\id\\\mathbf{0}}}"}
\Ar{2-3}{2-4}{"{{\sbmat{\mathbf{0}&\id}}}"}
\Ar{2-4}{2-5}{}
\Ar{3-1}{3-2}{}
\Ar{3-2}{3-3}{"\iota_{0}"}
\Ar{3-3}{3-4}{"\ep_{0}"}
\Ar{3-4}{3-5}{}
\Ar{4-2}{4-3}{}
\Ar{4-3}{4-4}{"\overline{\ep_{0}}"}
\Ar{4-4}{4-5}{}
\Ar{1-2}{2-2}{}
\Ar{1-3}{2-3}{}
\Ar{1-4}{2-4}{}
\Ar{2-2}{3-2}{equal}
\Ar{2-3}{3-3}{"\iota'_{0}"}
\Ar{2-4}{3-4}{"\ep_{0}\circ \la"}
\Ar{2-4}{3-3}{"\la" ', dashed}
\Ar{3-2}{4-2}{}
\Ar{3-3}{4-3}{"\cok \iota'_{0}"}
\Ar{3-4}{4-4}{"\pi"}
\ar[to path={([yshift=8pt]lu.north east)--([yshift=8pt]ru.north east)--
(rd.south east)--(ld.south west)--([yshift=8pt]lu.north west)--([yshift=8pt]lu.north east)}, dash, dashed, rounded corners]
\end{tikzcd}
\end{equation}
By applying the snake lemma to this diagram, we obtain that $\Ker\iota'_{0} = 0$
and that $\overline{\ep_{0}}\colon \Cok\iota'_{0}\to V_S/\soc V_S$ is an isomorphism.
Since $\overline{\ep_{0}}\circ \cok \iota'_{0} = \pi \circ \ep_{0} = \ep'_{0}$,
the center column {yields the exact sequence} \eqref{eq:necess-ex}.
\end{proof}
\removelastskip
\endgroup

\subsubsection{{Duality and the Auslander--Reiten translation}}

{We return to a nonempty finite poset $S$ of $(n,m)$-type. The duality $D\coloneqq \Hom_\k(\blank,\k)\colon\mod\k[S]\to\mod\k[S\op]$ identifies $DV_S$ with $V_{S\op}$.
The opposite poset is of $(m,n)$-type, with $\src(S\op)=\snk(S)$ and $\src_1(S\op)=\snk_1(S)$.
We record its effect on multiplicities and matrices, and then use a projective presentation of $DV_S$ to construct one of $\ta\inv V_S$.}

\begin{lem}
\label{lem:dual-comprs}
{Let $S$ be a finite poset and $L,N\in\mod\k[S]$, with $L$ indecomposable. Then $d_N(L)=d_{DN}(DL)$.}
\end{lem}

\begin{proof}
{Write $N\iso L^{c}\ds N'$ with $N'$ having no direct summand isomorphic to $L$.
Applying $D$ gives $DN\iso (DL)^{c}\ds DN'$.
Since $D$ is a duality, $DL$ is indecomposable and $DN'$ has no direct summand isomorphic to $DL$. Hence $d_{DN}(DL)=c=d_N(L)$.}
\end{proof}

\begin{lem}
\label{lem:rank-D}
Let $f \colon V \to W$ be a linear map in $\mod \k$.
Then $\rank D(f) = \rank f$.
\end{lem}

\begin{proof}
The linear map $f$ is expressed as the composite
$f = f_1 \circ f_2$ for some epimorphism $f_2 \colon V \to \Im f$
and some monomorphism $f_1 \colon \Im f \to W$.
Then $D(f)$ is expressed as $D(f) = D(f_2) \circ D(f_1)$,
where $D(f_1) \colon D(W) \to D(\Im f)$ is an epimorphism and $D(f_2) \colon D(\Im f) \to D(V)$ is a monomorphism.
Hence we have $\Im D(f) \iso D(\Im f)$.
Then the assertion follows from
$\dim \Im f = \dim D(\Im f) = \dim \Im D(f)$.
\end{proof}

\begin{lem}
\label{lem:blockwise-transpose}
{Let $f\colon V\to W$ be a morphism in $\mod\k$, and let $V=\Ds_{i\in I}V_i$ and $W=\Ds_{j\in J}W_j$ be direct sum decompositions.}
If $f = [f_{j,i}]_{(j,i) \in J \times I}$ with $f_{j,i} \colon V_i \to W_j$ is a matrix expression
of $f$ with respect to these direct sum decompositions, then $D(f)$ has a matrix expression
$D(f) = [D(f_{j,i})]_{(i,j) \in I \times J}$
with $D(f_{j,i}) \colon D(W_j) \to D(V_i)$
with respect to the direct sum decompositions
$D(V) \iso \Ds_{i\in I} D(V_i)$ and
$D(W) \iso \Ds_{j\in J} D(W_j)$.
Hence by \cref{lem:rank-D}, we have
\[
\rank\, [D(f_{j,i})]_{(i,j) \in I \times J}
= \rank\, [f_{j,i}]_{(j,i) \in J \times I},
\]
where $[f_{j,i}]_{(j,i) \in J \times I}$ can be seen
as the %
transpose of the formal matrix
$[f_{j,i}]_{(i,j) \in I \times J}$.
\end{lem}

\begin{proof}
Let $(\si_i^V \colon V_i \to V)_{i \in I}$ be the 
family of canonical injections,
and let $(\pi_j^W \colon W \to W_j)_{j \in J}$ be the family of canonical projections with respect to the decompositions of $V, W$ above,
respectively.
Then $f_{j,i} = \pi_j^W \circ f \circ \si_i^V$ for all
$i \in I, j \in J$.
Now $(D(\si_i^V) \colon D(V) \to D(V_i))_{i \in I}$ forms
the family of the canonical projections, and $(D(\pi_j^W) \colon D(W_j) \to D(W))_{j \in J}$ the canonical injections with respect to the decompositions
$D(V) \iso \Ds_{i\in I} D(V_i)$ and
$D(W) \iso \Ds_{j\in J} D(W_j)$, respectively.
Hence $D(f)$ has the matrix expression
$D(f) = [D(f)_{i,j}]_{(i,j) \in I \times J}$,
where
$D(f)_{i,j}  = D(\si_i^V) \circ D(f) \circ D(\pi_j^W)
= D(\pi_j^W \circ f \circ \si_i^V) = D(f_{j,i})$.
\end{proof}

\begingroup
Following the usual convention in representation theory, we denote by $(\blank)^t$ the contravariant functors
\[
\begin{aligned}
\Hom_{\k[S]}(\blank, \k[S](\cdot, ?)) &\colon \mod \k[S] \to \mod \k[S\op],\\
&M \mapsto \Hom_{\k[S]}({}_{?}M, \k[S](\cdot, ?)), \text{ and}\\
\Hom_{\k[S\op]}(\blank, \k[S\op](\cdot, ?)) &\colon \mod \k[S\op] \to \mod \k[S],\\
&M \mapsto \Hom_{\k[S\op]}(M_{?}, \k[S](?, \cdot)),
\end{aligned}
\]
{which restrict to dualities} between
$\prj \k[S]$ and $\prj \k[S\op]$, where $\prj B$ denotes the full subcategory
of $\mod B$ consisting of projective modules for any finite $\k$-category $B$.
We use the notation $P'_x$ provided in \cref{ntn:Yoneda} for $S$. Then by the Yoneda lemma, we have
\[
P^t_x = \Hom_{\k[S]}(\k[S](x, ?), \k[S](\cdot, ?)) \iso \k[S](\cdot, x) = \k[S\op](x, \cdot) = P'_x
\]
for all $x \in S$.
By this natural isomorphism, we usually identify $P'_x$ with $P^t_x$, and $\sfP'_{x,y}$ with $(\sfP_{y,x})^t$ for all $x, y \in S$. For this reason, we write $P^t$ instead of $P'$ in the sequel if there is no confusion.

Remember that $\ta\inv M = \Tr D M$ for all $M \in \mod \k[S]$,
where for each $N \in \mod \k[S\op]$,
the {\em transpose} $\Tr N$ of $N$ is defined as the cokernel of some $f^t$
with $P_1 \ya{f} P_0 \to N \to 0$ a minimal projective presentation of $N$.
By applying \cref{prp:prjpres-VI w/o conditions} {to $S\op$ and using $V_{S\op}\iso DV_S$}, we first obtain a projective presentation
of $DV_S$ as follows.

\begin{prp}
\label{prp:prjpres-DVI w/o conditions}
Let $S$ be a nonempty finite poset of $(n,m)$-type. Then we have the following (not necessarily minimal) projective presentation of $D V_S$ in $\mod \k[S\op]$:
\begin{equation}
\label{eq:proj-resol-DV_I-nsrccase-lem w/o conditions}
P'_{\snk_1(S)} \ya{\psi_{1}} P'_{\snk(S)} \ya{\psi_{0}} DV_S \to 0.
\end{equation}
Here $\psi_{0}, \psi_{1}$ are given by
\begin{align*}
    \psi_{0}\coloneqq (\ro'_{1_{b_{1}}}, \ro'_{1_{b_{2}}}, {\dots}, \ro'_{1_{b_{m}}})
\end{align*}
$($see \cref{ntn:Yoneda} for $\ro')$,
where $1_u \coloneqq 1_\k \in \k = V_{S\op}(u)$ for all $u \in S$, and
\begin{align}\label{eq:matrix form of ep_1 w/o conditions, dual version}
\psi_{1}\coloneqq
\bmat{\vec{\sfP}_{b,\bfb_d}}_{(b, \bfb_d) \in \snk(S) \times \snk_1(S)},
\end{align}
where the entry is given by
\begin{align}\label{eq:entries of ep_1 w/o conditions, dual version}
\vec{\sfP}_{b,\bfb_d}\coloneqq
\begin{cases}
    \sfP'_{d,b}, & \textnormal{if } b = \udl{\bfb},\\
    -\sfP'_{d,b}, & \textnormal{if } b = \ovl{\bfb},\\
    \mathbf{0}, & \textnormal{if } b \not\in \bfb,
\end{cases}
\end{align}
for all $\bfb_d \in \snk_1(S)$ and $b \in \snk(S)$.
\qed
\end{prp}

\begin{rmk}
\label{rmk:transpose}
{Let $f\colon V\to W$ be a morphism in $\mod\k[S]$, and let $V=\Ds_{i\in I}V_i$ and $W=\Ds_{j\in J}W_j$ be direct sum decompositions.}
If $f = [f_{j,i}]_{(j,i) \in J \times I}$ with $f_{j,i} \colon V_i \to W_j$ is a matrix expression of $f$ with respect to these direct sum decompositions, then it is clear that
$f^{t}$ has a matrix expression
$f^{t}=[f_{j,i}]^{t}_{(j,i) \in J \times I}
=[f^t_{j,i}]_{(i,j) \in I \times J}
= {}^t\!\left([f^t_{j,i}]_{(j,i) \in J \times I}\right)$
(see \cref{ntn:Yoneda} (4))
with entries $f^t_{j,i} \colon W_j^t \to V_i^t$
with respect to the direct sum decompositions
$V^t \iso \Ds_{i \in I} V_i^t$ and $W^t \iso \Ds_{j \in J} W_j^t$.
\end{rmk}

From now on, assume that $S$ is connected and of $(n,m)$-type with $m\ge2$, so that $V_S$ is not injective.
We note here that $\psi_0$ is a projective cover of $DV_S$ in
\eqref{eq:proj-resol-DV_I-nsrccase-lem w/o conditions}
because it induces an isomorphism $\top P^t_{\snk(S)} \iso \top DV_S$
(see \cref{dfn:rad-top-soc} and \cref{rmk:prj-cov}),
but $\psi_1 \colon P^t_{\snk_1(S)}\to \Im \psi_1$ is not always a projective cover.
Choose a decomposition
\begin{equation}
\label{eq:decomp-P'snk1-I}
P^t_{\snk_1({S})} = P^t_1 \ds P^t_2
\end{equation}
with $\psi_{11} \colon P^t_1 \to \Im \psi_1$ a projective cover, where
$\psi_1 = (\psi_{11}, \mathbf{0})$ is a matrix expression of $\psi_1$ with respect to
this decomposition of $P^t_{\snk_1(S)}$. {For $i=1,2$, put $P_i\coloneqq(P_i^t)^t$, using the natural identification of a projective module with its double dual.}

{Keeping the additional projective summand gives the explicit presentation needed in the general multiplicity formula.}

\begin{lem}
\label{lem:translation-presentation}
In the setting above, we can give a projective presentation of
$\ta\inv V_S \ds P_2$ as follows:
\begin{equation}
\label{eq:proj-resol-ta-invV_I+P-nsrccase-lem w/o conditions}
P_{\snk(S)} \ya{{\psi^t_{1}}=\bmat{\psi^t_{11}\\ \mathbf{0}}} P_1 \ds P_2 = P_{\snk_1(S)} \ya{\cok{\psi^t_{11}}\ds \id_{P_2}} \ta\inv V_S \ds P_2 \to 0.
\end{equation}
Here by 
\eqref{eq:matrix form of ep_1 w/o conditions, dual version}
and \eqref{eq:entries of ep_1 w/o conditions, dual version},
the precise form of $\ps_1^t$ is given as follows:
\begin{align*}
\psi^t_{1} &= 
\bmat{\vec{\sfP}_{b,\bfb_d}}_{(b, \bfb_d) \in \snk(S) \times \snk_1(S)}^t
= \bmat{\vec{\sfP}_{b,\bfb_d}^t}_{(\bfb_d, b) \in \snk_1(S) \times \snk(S)}\\
&\eqqcolon  \bmat{\hat{\sfP}_{b,\bfb_d}}_{(\bfb_d, b) \in \snk_1(S) \times \snk(S)},
\end{align*}
where the entry is given by
\begin{align}\label{eq:entries of psi_1 w/o conditions, dual version}
\hat{\sfP}_{b,\bfb_d}\coloneqq
    \begin{cases}
        \sfP_{b,d}, & \textnormal{if } b = \udl{\bfb},\\
        -\sfP_{b,d}, & \textnormal{if } b = \ovl{\bfb},\\
        \mathbf{0}, & \textnormal{if } b \not\in \bfb,
    \end{cases}
\end{align}
for all $b \in \snk(S)$ and $\bfb_d \in \snk_1(S)$.
\end{lem}

\begin{proof}
By the construction above, $DV_S$ has a minimal projective presentation
\begin{equation}
\label{eq:min-proj-resol-DV_I-nsrccase-lem w/o conditions}
P^t_1 \ya{\psi_{11}} P^t_{\snk(S)} \ya{\psi_{0}} DV_S \to 0.
\end{equation}
Hence by applying $(\blank)^t$
to $\psi_{11}$ in \eqref{eq:min-proj-resol-DV_I-nsrccase-lem w/o conditions},
we have a minimal projective presentation
\begin{equation}
\label{eq:min-proj-resol-ta-invV_I-nsrccase-lem w/o conditions}
P_{\snk(S)} \ya{\psi^t_{11}} P_1 \ya{\cok{\psi^t_{11}}} \ta\inv V_S \to 0
\end{equation}
of $\ta\inv V_S = \Tr D V_S$ in $\mod \k[S]$.
Hence the assertion follows.
\end{proof}
\removelastskip
\endgroup

\subsubsection{{The middle term of the almost split sequence}}

{Assume that $S$ is connected and of $(n,m)$-type with $n,m\ge2$.
Since $V_S$ is indecomposable and not injective, there is an almost split sequence
\begin{equation}
\label{eq:ass-VI w/o conditions}
0\to V_S\to E\to\ta\inv V_S\to0.
\end{equation}
The presentation in \cref{lem:translation-presentation} has explicitly given projective terms, at the cost of the summand $P_2$.
We now construct a presentation of $E\ds P_2$ using the same $P_2$.
These presentations will allow us to compute the two Hom dimensions together, without determining $P_2$ separately.}

\begingroup
\begin{prp}
\label{prp:n-m-ge-2 w/o conditions}
Let $S$ be a nonempty finite connected poset
of $(n,m)$-type with $m,n \ge 2$, $\src(S)=\{a_1,\dots, a_n\}$, $\snk(S) =\{b_1,\dots, b_m\}$, $E$ the middle term in \eqref{eq:ass-VI w/o conditions},
and $P_2$ a direct summand of $P_{\snk_1(S)}$ as in \eqref{eq:decomp-P'snk1-I}. Assume that $a_1 \le b_1$ without loss of generality.
Then $E \ds P_2$ has the following projective presentation:
\[
P_{\src_1(S)} \ds P_{\snk(S)} \ya{\mu_E} P_{\src(S)}\ds P_{\snk_1(S)}
\ya{\ep_E} E \ds P_2  \to 0.
\]
Here $\mu_E$ is given by
\[
\mu_E\coloneqq \bmat{
\ep_1 & \bmat{\sfP_{b_1,a_1} & \mathbf{0}\\\mathbf{0}&\mathbf{0}}\\
\mathbf{0} & \psi^t_1},
\]
where $\ep_1 \colon P_{\src_1(S)}\to P_{\src(S)}$ is given in \eqref{eq:matrix form of ep_1 w/o conditions}, and $\psi^t_1 \colon P_{\snk(S)}\to P_{\snk_1(S)}$ is given in \eqref{eq:proj-resol-ta-invV_I+P-nsrccase-lem w/o conditions}.
\end{prp}

\begin{proof}
By \cite[Section 3.6]{gabriel2006auslander}, an almost split sequence
in \eqref{eq:ass-VI w/o conditions} can be obtained as {a pushout of the sequence}
\eqref{eq:min-proj-resol-ta-invV_I-nsrccase-lem w/o conditions}
along a morphism $\et \colon P_{\snk(S)} \to V_S$ as follows:
\begin{equation}
\label{eq:pushout w/o conditions}
\begin{tikzcd}
\Nname{P_1}P_{\snk(S)} & \Nname{P_0}P_{1} & \Nname{ta1}\ta\inv V_S & \Nname{01}0\\
\Nname{VI}V_S & \Nname{E}E & \Nname{ta2}\ta\inv V_S & \Nname{02}0
\Ar{P_1}{P_0}{"\psi^t_{11}"}
\Ar{P_0}{ta1}{}
\Ar{ta1}{01}{}
\Ar{VI}{E}{}
\Ar{E}{ta2}{}
\Ar{ta2}{02}{}
\Ar{P_1}{VI}{"\et"}
\Ar{P_0}{E}{}
\Ar{ta1}{ta2}{equal}
\end{tikzcd}.
\end{equation}
Here, $\et$ is the composite of morphisms
\[
P_{\snk(S)} \ya{\text{can.}} \top P_{\snk(S)} \isoto \soc \nu P_{\snk(S)}
\isoto \soc V_S \ya{\al} L \hookrightarrow \soc V_S \hookrightarrow V_S
\]
(see \cref{dfn:rad-top-soc}),
where $\nu$ is the Nakayama functor $\nu\coloneqq D\circ(\blank)^t$, $L$ is any simple $\k[{S}]$-$\End_{\k[{S}]}(V_S)$-subbimodule of
$\soc V_S$, and $\al$ is a retraction.

Here we claim that
any simple $\k[{S}]$-submodule of $\soc V_S$ is automatically
a simple $\k[{S}]$-$\End_{\k[{S}]}(V_S)$-subbimodule of $\soc V_S$. Indeed, this follows from the fact that
$\soc V_S = \Ds_{i \in [m]}V_{\{b_i\}}$, where $V_{\{b_i\}}$ are mutually non-isomorphic simple $\k[{S}]$-modules.
More precisely, it is enough to show that $f(L) \subseteq L$ for any $f \in \End_{\k[{S}]}(V_S)\op$
because if this is shown, then $L$ turns out to be a right $\End_{\k[{S}]}(V_S)$-submodule
and a simple $\k[{S}]$-$\End_{\k[{S}]}(V_S)$-subbimodule of $\soc V_S$.
Let $T$ be any simple $\k[{S}]$-submodule {of $\soc V_S$. Then}
by the fact above $T \iso V_{\{b_i\}}$ for a unique $i \in [m]$, and hence
$\pr_j(T) = 0$ for all $j \in [m]\setminus \{i\}$,
where $\pr_j \colon \soc V_S \to V_{\{b_j\}}$ is the canonical projection.
Thus $T \subseteq V_{\{b_i\}}$, which shows that $T = V_{\{b_i\}}$ {because both sides are simple}.
Now there exists a unique $i \in [m]$ such that $L = V_{\{b_i\}}$.
If $f = 0$, then $f(L) = 0 \subseteq L$; otherwise $f(L) \iso L$,
and then $f(L) = V_{\{b_i\}} = L$ by the argument above. This proves our claim.

Therefore, we may take $L\coloneqq V_{\{b_1\}}$, and
\[
\et\coloneqq [\ro_{1_{b_1}}, \mathbf{0}, \dots, \mathbf{0}] \colon P_{\snk({S})} = P_{b_1} \ds \cdots \ds P_{b_m} \to V_S.
\]
By assumption, $a_1 \le b_1$ in ${S}$.
Hence we have a commutative diagram
\[
\begin{tikzcd}
& P_{\src({S})}= P_{a_1} \ds \cdots \ds P_{a_n}\\
P_{\snk({S})} & V_S
\Ar{1-2}{2-2}{"{\ep_0 = (\ro_{1_{a_{1}}},\dots, \ro_{1_{a_{n}}})}"}
\Ar{2-1}{2-2}{"\et" '}
\Ar{2-1}{1-2}{"\et'\coloneqq {\Pzero}"}
\end{tikzcd}
\]
We recall that $\ep_0 \colon P_{\src({S})}\to V_S$ is given in~\eqref{eq:matrix form of ep_0 w/o conditions}. The above diagram commutes because for each $p \in P_{b_1}$, we have
\[
\ro_{1_{a_1}}(\sfP_{b_1,a_1}(p)) = \ro_{1_{a_1}}(p\cdot p_{b_1,a_1}) = V_S(p\cdot p_{b_1,a_1})(1_{a_1}) = V_S(p)(1_{b_1}) = \ro_{1_{b_1}}(p).
\]
{The pushout diagram} \eqref{eq:pushout w/o conditions} yields the following exact sequences:
\[
P_{\snk({S})} \ya{\bmat{\et\\ \psi^t_{11}}} V_S \ds P_{1} \to E \to 0,\text{ and }
P_{\snk({S})} \ya{\bmat{\et\\ \psi^t_{1}}} V_S \ds P_{\snk_1({S})} \ya{\pi} E \ds P_2 \to 0.
\]
The latter is extended to the following commutative diagram with the bottom row exact:
\[
\begin{tikzcd}[row sep=35pt, ampersand replacement=\&]
P_{\src_1({S})} \ds P_{\snk({S})} \&[35pt] P_{\src({S})}\ds P_{\snk_1({S})} \& E \ds P_2 \& 0\\
P_{\src_1({S})} \ds P_{\snk({S})} \& V_S\ds P_{\snk_1({S})} \& E \ds P_2 \& 0
\Ar{1-1}{1-2}{"\mu_E\coloneqq \bmat{\ep_1 & \et'\\\mathbf{0} & \psi^t_1}"}
\Ar{1-2}{1-3}{"\ep_E"}
\Ar{1-3}{1-4}{}
\Ar{2-1}{2-2}{"\bmat{\mathbf{0} & \et\\ \mathbf{0} & \psi^t_1}" '}
\Ar{2-2}{2-3}{"\pi" '}
\Ar{2-3}{2-4}{}
\Ar{1-1}{2-1}{equal}
 \Ar{1-2}{2-2}{"\bmat{\ep_0 & \mathbf{0}\\ \mathbf{0} & \id}"}
 \Ar{1-3}{2-3}{equal}
\end{tikzcd},
\]
where we set $\ep_E\coloneqq \pi \circ \sbmat{\ep_0 & \mathbf{0}\\ \mathbf{0} & \id}$,
which is an epimorphism as the composite of epimorphisms.
\par
It remains to show that $\ep_E$ is a cokernel morphism of $\mu_E$.
By the commutativity of the diagram and the exactness of
the bottom row, we see that $\ep_E \mu_E = 0$.
Let $(f,g) \colon P_{\src({S})}\ds P_{\snk_1({S})} \to X$ be a morphism
with $(f,g)\mu_E = 0$.
Then $f \ep_1 = 0$.
Since $\ep_0$ is a cokernel morphism of $\ep_1$,
there exists some $f' \colon V_S \to X$ such that
$f = f' \ep_0$.
Then we have $(f,g) = (f', g)\sbmat{\ep_0 &\mathbf{0}\\ \mathbf{0}&\id}$.
Now $(f',g)\sbmat{\mathbf{0}& \et\\ \mathbf{0}& \psi^t_1} = (f',g) \sbmat{\ep_0 &\mathbf{0}\\ \mathbf{0}&\id} \mu_E = (f,g)\mu_E = 0$.
Hence $(f',g)$ factors through $\pi$, that is,
$(f',g) = h \pi$ for some $h \colon E\ds P_2 \to X$.
Therefore, we have $(f,g) = h \pi \sbmat{\ep_0 &\mathbf{0}\\ \mathbf{0}&\id} = h\, \ep_E$.
The uniqueness of $h$ follows from the fact that $\ep_E$ is an
epimorphism.
As a consequence, $\ep_E$ is a cokernel morphism of $\mu_E$.
\end{proof}
\removelastskip
\endgroup

\subsection{{The interval-rank formula}}
\label{General case}

{We now fix a compression system $\xi=(\xi_I\colon I^\xi\to\bfP)_{I\in\bbI}$ and an interval $I$.
In the constructions above, take $S=I^\xi$ and $N=R_I(M)$.
By \cref{lem:com-mult-full-int}, the required multiplicity is $d_N(V_S)=\rank_I^\xi M$.
We use \cref{ntn:Yoneda,ntn: notations for general case} with this choice of $S$.
The numbers of sources and sinks determine which case of \cref{thm:Thm3-asashiba} applies.}

\subsubsection{{\rm (1,1)}-type}
{When $I^\xi$ has a unique source and a unique sink, the $I$-rank is the rank of a single structure linear map.}

\begin{prp}
\label{prp:mult-for-1-1-case}
{Let $M\in\mod A$, and assume that $I^\xi$ is of $(1,1)$-type.
Write $\src(I^\xi)=\{a\}$ and $\snk(I^\xi)=\{b\}$.
Then $I=[\xi_I(a),\xi_I(b)]$ and}
\[
{\rank_I^\xi M=\rank M_{\xi_I(b),\xi_I(a)}.}
\]
\end{prp}

\begin{proof}
{Since $I^\xi$ is finite, $a\le u\le b$ for all $u\in I^\xi$.
Condition (2) in \cref{dfn:comp-sys-simplified-ver} then gives
$\xi_I(a)\le x$ for every $x\in\src(I)$ and $y\le\xi_I(b)$ for every $y\in\snk(I)$.
As $\xi_I(a),\xi_I(b)\in I$, minimality and maximality imply
$\src(I)=\{\xi_I(a)\}$ and $\snk(I)=\{\xi_I(b)\}$.
Thus $I=[\xi_I(a),\xi_I(b)]$ by \cref{lem:intv-as-K-L}.
The rank equality follows from \cref{prp:seg-rank} and \cref{dfn:int-rk}.}
\end{proof}

\begin{rmk}
{An interval $I$ of $(1,1)$-type need not have $I^\xi$ of $(1,1)$-type.}
For example, let $\bfP = G_{2,2} = I$, and
$I^{\xi}\coloneqq \{(1,1), (1,2), (2,2)\}$ with the order relations
defined by the Hasse quiver $(1,1) \to (2,2) \leftarrow (1,2)$.
{Thus $I^\xi$ is not a full subposet of $\bfP$.}
Define $\xi_I \colon I^{\xi} \to I$ as the inclusion.
{For every other interval $J\in\bbI(\bfP)$, set $J^\xi\coloneqq J$ and let $\xi_J\colon J\hookrightarrow\bfP$ be the inclusion.}
Then this $\xi$ is a {compression system}.
In this case, $I$ is of $(1,1)$-type, but $I^{\xi}$ is of $(2,1)$-type.
\end{rmk}

\subsubsection{\texorpdfstring{$(n,1)$}{(n,1)}-type}

{Using \cref{prp:prjpres-VI w/o conditions,prp:quotient-socle-presentation} with $S=I^\xi$, we apply the injective case of \cref{thm:Thm3-asashiba} to obtain a formula involving two matrix ranks.}

\begin{thm}
\label{thm:n-1-case w/o conditions}
Let $M\in \mod A$, and assume that $I^{\xi}$ is of $(n,1)$-type with $n \ge 2$, $\src(I^{\xi})=\{a_1,\dots, a_n\}$, and $\snk(I^{\xi}) =\{b\}$. Then we have
\begin{equation}
\label{eq:nsrc1snkrank w/o conditions}
{\rank_I^\xi M} =
\rank
\bmat{\tdbfM \\
\tilde{\beta}
}
- \rank \tdbfM.
\end{equation}
Here $\tilde{\beta}$ is given by
\[\tilde{\beta} =
 \big[M_{\xi_I(b),\xi_I(a_{1})}\ \underbrace{\mathbf{0}\ \cdots\ \mathbf{0}}_{n-1}\big],
\]
and $\tdbfM$ is given by
$$
\tdbfM = \bmat{\tdM_{\bfa_c,a}}_{(\bfa_c,a) \in \src_1(I^{\xi})\times \src(I^{\xi})},
$$
where
\begin{align}\label{eq:entries of M w/o conditions, dual version}
    \tdM_{\bfa_c,a}\coloneqq
    \begin{cases}
        M_{\xi_I(c),\xi_I(a)}, & \textnormal{if } a = \udl{\bfa},\\
        -M_{\xi_I(c),\xi_I(a)}, & \textnormal{if } a = \ovl{\bfa},\\
        \mathbf{0}, & \textnormal{if } a \not\in \bfa,
    \end{cases}
\end{align}
for all $\bfa_c \in \src_1(I^{\xi})$ and $a \in \src(I^{\xi})$.
\end{thm}

\begin{proof}
Since $I^{\xi}$ has a unique sink $b$,
$V_{I^{\xi}}$ is isomorphic to an injective indecomposable $\k[I^{\xi}]$-module
{$I_b$ from \cref{ntn:Yoneda}}.
{Then by applying \cref{thm:Thm3-asashiba}
to $V_{I^{\xi}}$, we have}
\begin{equation}
\label{eq:d_R(M)-nsrc w/o conditions}
d_{R_I(M)}(V_{I^{\xi}}) = \dim\Hom_{\k[I^{\xi}]}(V_{I^{\xi}}, R_I(M)) - \dim\Hom_{\k[I^{\xi}]}(V_{I^{\xi}}/\soc V_{I^{\xi}}, R_I(M)).
\end{equation}
A projective presentation of $V_{I^{\xi}}$ in $\mod \k[I^{\xi}]$
is given by \eqref{eq:min-proj-resol-V_I-nsrccase-lem w/o conditions} in
\cref{prp:prjpres-VI w/o conditions} {with $S=I^\xi$}.
Hence by \cref{lem:dim-Hom-coker},
we have
\begin{equation}
\label{eq:1st-eq-nsrc w/o conditions}
\dim \Hom_{\k[I^{\xi}]}(V_{I^{\xi}} ,R_I(M)) = \sum_{i\in [n]}\dim M(\xi_I(a_{i})) - \rank \tdbfM
\end{equation}
because $R_I(M)(a_i) = M(\xi_I(a_{i}))$.

On the other hand, a projective presentation of $V_{I^{\xi}}/\soc V_{I^{\xi}}$ is given by
\eqref{eq:min-proj-resol-V_I/socV_I-nsrccase w/o conditions} {in \cref{prp:quotient-socle-presentation}, again with $S=I^\xi$}.
Hence by \cref{lem:dim-Hom-coker}, we have
\begin{equation}
\label{eq:2nd-eq-nsrc w/o conditions}
\dim \Hom_{\k[I^{\xi}]}(V_{I^{\xi}}/\soc V_{I^{\xi}}, R_I(M)) = \sum_{i\in [n]}\dim M(\xi_I(a_{i})) - \rank \bmat{\tdbfM \\ \tilde{\beta}}.
\end{equation}
By equations \eqref{eq:d_R(M)-nsrc w/o conditions}, \eqref{eq:1st-eq-nsrc w/o conditions}, {and} \eqref{eq:2nd-eq-nsrc w/o conditions}, we
obtain \eqref{eq:nsrc1snkrank w/o conditions}.
\end{proof}

\subsubsection{\texorpdfstring{$(1,m)$}{(1,m)}-type}

{Applying duality to \cref{thm:n-1-case w/o conditions} gives the corresponding formula for a unique source and several sinks.}

\begin{thm}
\label{thm:1-n-case w/o conditions}
Let $M\in \mod A$, and $I^{\xi}$ be of $(1,m)$-type with $m \ge 2$, {$\src(I^\xi)=\{a\}$}, and
{$\snk(I^\xi) =\{b_1,\dots, b_m\}$}. Then we have
\begin{equation}
\label{eq:(1,n)rank w/o conditions}
{\rank_I^\xi M} =
\rank
\bmat{\hatbfM\ \hat{\beta}}
- \rank \hatbfM.
\end{equation}
Here $\hat{\beta}$ is given by
\[\hat{\beta} =
{}^t\!\bmat{M_{\xi_I(b_{1}),\xi_I(a)} &
\underbrace{\mathbf{0}\ \cdots\ \mathbf{0}}_{m-1}},
\]
and $\hatbfM$ is given by
\begin{align*}
    \hatbfM = \bmat{\hatM_{b, \bfb_d}}_{(b, \bfb_d) \in \snk(I^{\xi}) \times \snk_1(I^{\xi})},
\end{align*}
where
\begin{align}\label{eq:entries of N w/o conditions, dual version}
    \hatM_{b, \bfb_d}\coloneqq
    \begin{cases}
        M_{\xi_I(b),\xi_I(d)}, & \textnormal{if } b = \udl{\bfb},\\
        -M_{\xi_I(b),\xi_I(d)}, & \textnormal{if } b = \ovl{\bfb},\\
        \mathbf{0}, & \textnormal{if } b \not\in \bfb,
    \end{cases}
\end{align}
for all $b \in \snk(I^{\xi})$ and $\bfb_d \in \snk_1(I^{\xi})$.
\end{thm}

\begin{proof}
{Let $R_I\op\colon\mod\k[\bfP\op]\to\mod\k[(I^\xi)\op]$ be the restriction functor induced by $\k[\xi_I\op]$.
Restriction commutes with duality, so $DR_I(M)=R_I\op(DM)$, and $DV_{I^\xi}\iso V_{(I^\xi)\op}$.}
By \cref{lem:dual-comprs}, we have
${\rank_I^\xi M} = d_{R_I(M)}(R_I(V_I)) = d_{R_I\op(DM)}(V_{(I^{\xi})\op})$.
{Here} $(I^{\xi})\op$ is of $(m,1)$-type
with {$\src((I^\xi)\op) = \{b_1,\dots, b_m\}$}
and {$\snk((I^\xi)\op) = \{a\}$}. Hence to compute $d_{R_I\op(DM)}(V_{(I^{\xi})\op})$,
we can apply \cref{thm:n-1-case w/o conditions} to the following setting:
poset $\bfP\op$, module $DM$, the interval $I\op$,
the poset $(I^{\xi})\op$, and the order-preserving map
$\xi_I\op \colon (I^{\xi})\op \to \bfP\op$ that is defined by
$\xi_I\op(x)\coloneqq \xi_I(x)$ for all $x \in (I^{\xi})\op$. {The theorem gives}
\[
d_{R_I\op(DM)}(V_{(I^{\xi})\op}) =
\rank
\bmat{\vec{\bfM} \\
{\vec{\beta}}
}
- \rank \vec{\bfM}.
\]
Here
\begin{align*}
    \vec{\bfM} \coloneqq \bmat{\vec{M}_{\bfb_d, b}}_{(\bfb_d, b) \in \snk_1(I^{\xi}) \times \snk(I^{\xi})},
\end{align*}
where
\begin{align}\label{eq:entries of DM w/o conditions}
    \vec{M}_{\bfb_d, b}\coloneqq
    \begin{cases}
        (DM)_{\xi_I(d),\xi_I(b)}, & \textnormal{if } b = \udl{\bfb},\\
        -(DM)_{\xi_I(d),\xi_I(b)}, & \textnormal{if } b = \ovl{\bfb},\\
        \mathbf{0}, & \textnormal{if } b \not\in \bfb,
    \end{cases}
\end{align}
for all $\bfb_d \in \snk_1(I^{\xi})$ and $b \in \snk(I^{\xi})$,
and
\[\vec{\beta} =
 \big[(DM)_{\xi_I(a),\xi_I(b_{1})}\ \underbrace{\mathbf{0}\ \cdots\ \mathbf{0}}_{m-1}\big].
\]
Now for any $x, y \in \bfP\op$ with $x \le\op y$ in $\bfP\op$,
let $p\op_{y,x}$ be the
unique morphism in $\bfP\op(x,y)$.
Then we have $y \le x$ in $\bfP$, and $p_{x,y} = p\op_{y,x}$,
which is the unique morphism in $\bfP(y,x) = \bfP\op(x,y)$.
Hence we have
\[
(DM)_{y,x} = (DM)(p\op_{y,x}) = (DM)(p_{x,y}) = D(M(p_{x,y})) = D(M_{x,y}).
\]
Then \eqref{eq:(1,n)rank w/o conditions} follows by
\cref{lem:blockwise-transpose}.
\end{proof}

\subsubsection{\texorpdfstring{$(n,m)$}{(n,m)}-type with \texorpdfstring{$m, n \ge 2$}{m, n >= 2}}

{For several sources and sinks, we use the almost split sequence starting at $V_{I^{\xi}}$ to combine the matrices $\tdbfM$ and $\hatbfM$ from the preceding cases into a single block matrix. The next theorem gives the resulting $I$-rank formula.}

\begin{thm}
\label{thm:n-m-ge-2 w/o conditions}
Let $M \in \mod A$ and $I^{\xi}$ be
of $(n,m)$-type with $m,n \ge 2$, $\src(I^{\xi})=\{a_1,\dots, a_n\}$, and
$\snk(I^{\xi}) =\{b_1,\dots, b_m\}$.
Obviously, for $b_1 \in \snk(I^{\xi})$, there exists some $a_{i} \in \src(I^{\xi})$
such that $a_{i} \le b_1$.
Hence we may assume that $a_1 \le b_1$ without loss of generality.
Then we have
\begin{equation}
\label{eq:nsrcmsnk multiplicity w/o conditions}
    {\rank_I^\xi M} =
\rank \bmat{\tdbfM & {\mathbf{0}}\\
\bmat{M_{\xi_I(b_1),\xi_I(a_1)}&\mathbf{0}\\\mathbf{0}&\mathbf{0}} & \hatbfM\\
}
- \rank \tdbfM
-\rank \hatbfM,
\end{equation}
where $\tdbfM, \hatbfM$ are defined in \cref{thm:n-1-case w/o conditions,thm:1-n-case w/o conditions}.
\end{thm}

\begin{proof}
{Apply the constructions of \cref{lem:translation-presentation} and \cref{prp:n-m-ge-2 w/o conditions} with $S=I^\xi$.
Let $E$ be the middle term of the almost split sequence~\eqref{eq:ass-VI w/o conditions} for this $S$, and use the same projective summand $P_2$.}
The value of ${\rank_I^\xi M}=d_{R_I(M)}(V_{I^{\xi}})$ can be computed from the three terms
of this almost split sequence
by using \cref{thm:Thm3-asashiba} %
as follows:
\begin{equation}
\label{eq:formula-dRV-2-2 w/o conditions}
\begin{aligned}
d_{R_I(M)}(V_{I^{\xi}}) = \dim \Hom_{\k[I^{\xi}]}(V_{I^{\xi}}, R_I(M)) &- \dim \Hom_{\k[I^{\xi}]}(E, R_I(M))\\
&+ \dim \Hom_{\k[I^{\xi}]}(\ta\inv V_{I^{\xi}}, R_I(M))\\
= \dim \Hom_{\k[I^{\xi}]}(V_{I^{\xi}}, R_I(M)) &- \dim \Hom_{\k[I^{\xi}]}(E \ds P_2, R_I(M))\\
&+ \dim \Hom_{\k[I^{\xi}]}(\ta\inv V_{I^{\xi}} \ds P_2, R_I(M)),
\end{aligned}
\end{equation}
where $P_2$ is a direct summand of $P_{\snk_1(I^{\xi})}$ as in \eqref{eq:decomp-P'snk1-I}.
Hence the assertion follows by {\cref{prp:n-m-ge-2 w/o conditions}}
{together with the projective presentation} \eqref{eq:min-proj-resol-V_I-nsrccase-lem w/o conditions}
of $V_{I^{\xi}}$, {the projective presentation} \eqref{eq:proj-resol-ta-invV_I+P-nsrccase-lem w/o conditions} of $\ta\inv V_{I^{\xi}} \ds P_2$, and
\cref{lem:dim-Hom-coker}.
\end{proof}

{We use the convention that a direct sum over an empty index set is zero. A matrix with no rows or no columns represents the corresponding map to or from the zero vector space and has rank zero. With this convention, the preceding cases combine into the following formula, proving \textbf{Main result A}.}

\begin{thm}
\label{thm:general w/o conditions}
Let $M \in \mod A$, and $I^{\xi}$ be of $(n,m)$-type $(m,n \ge 1)$ with $\src({I^{\xi}})=\{a_1,\dots, a_n\}$, $\snk({I^{\xi}}) =\{b_1,\dots, b_m\}$. {After renumbering if necessary, assume that $a_1\le b_1$.} Then we have
\begin{equation}
\label{eqn:general formula for (n,m)-type w/o conditions}
{\rank_I^\xi M} =
\rank \bmat{\tdbfM & \mathbf{0}\\
\bmat{M_{\xi_I(b_1),\xi_I(a_1)}&\mathbf{0}\\\mathbf{0}&\mathbf{0}} & \hatbfM\\
}
- \rank \tdbfM
-\rank \hatbfM,
\end{equation}
where if $m = 1$ $($resp.\ $n = 1)$, then
$\hatbfM$ $($resp.\ $\tdbfM)$ is an empty matrix, and hence the formula
has the form in \cref{thm:n-1-case w/o conditions,thm:1-n-case w/o conditions}
or {\cref{prp:mult-for-1-1-case}}.
\end{thm}

{The formula holds for every numbering satisfying $a_1\le b_1$; in each case its value is $\rank_I^\xi M$.}

\subsection{Under the existence condition of pairwise joins/meets}
{We first simplify the matrices in \cref{thm:general w/o conditions} when pairwise joins of sources and pairwise meets of sinks exist. \cref{lem:degeneratedcase} identifies redundant rows and columns. We then prove the 2D-grid formula in \cref{cor:2Dcase-tot-formula} without assuming that all pairwise joins and meets exist in $I$. This formula will be used to construct zigzag essential covers in the next section.}

\begin{dfn}
\label{dfn: existence condition of pairwise joins/meets}
The poset $I^{\xi}$ is said to {\em satisfy the existence condition of pairwise joins} in $\src(I^{\xi})$ (resp.\ {\em meets} in $\snk(I^{\xi})$) if $a_{i}\vee a_{j}$ (resp.\ $b_{i}\wedge b_{j}$) exists in $I^{\xi}$ for every $i\neq j$. If this is the case, to shorten notation, we set $a_{ij}\coloneqq a_{i}\vee a_{j}$ (resp.\ $b_{ij}\coloneqq b_{i}\wedge b_{j}$).
Note that we have $a_{ij} = a_{ji}$ for all $i \ne j$ in $[n]$
(resp.\ $b_{ij} = b_{ji}$ for all $i \ne j$ in $[m]$).
\end{dfn}

{When pairwise joins and meets exist, the pre-joins and pre-meets are singletons, simplifying the index sets in \cref{ntn: notations for general case}.}
\begin{rmk}
\label{rmk:pair-index-sets}
{Under the respective existence conditions, $\vee'\{a_i,a_j\}=\{a_{ij}\}$ and
$\wedge'\{b_i,b_j\}=\{b_{ij}\}$ for distinct indices. Hence the corresponding index sets are
\[
\begin{aligned}
\src_1(I^{\xi})&=\{(\{a_i,a_j\},a_{ij}) \mid 1\le i<j\le n\},\\
\snk_1(I^{\xi})&=\{(\{b_i,b_j\},b_{ij}) \mid 1\le i<j\le m\}.
\end{aligned}
\]
The map $(\{a_i,a_j\},a_{ij})\mapsto\{i,j\}$ identifies $\src_1(I^{\xi})$ with $\sub_2[n]$,
and similarly for $\snk_1(I^{\xi})$ and $\sub_2[m]$.
Distinct pairs remain distinct indices even when their joins or meets coincide.
In matrix labels, we write $ij$ for $\{i,j\}$ with $i<j$.
In matrices, these indices are arranged lexicographically in $(i,j)$.}
\end{rmk}

\begin{rmk}
\label{rmk:exist-cond-case}
{If $I^{\xi}$ satisfies the existence condition of pairwise joins in $\src(I^{\xi})$ and
meets in $\snk(I^{\xi})$, then the matrices $\tdbfM,\, \hatbfM$ have the following forms,
with the pair indices in the lexicographic order of \cref{ntn: notations for general case},}
where we denote $M_{\xi_I(b),\xi_I(a)}$ simply by ${M^{\xi_I}_{b,a}}$ {for all $a\le b$ in $I^\xi$}:
\[\scalebox{0.9}{$
\begin{blockarray}{ccccccc}
& \scalebox{0.7}{$1$} & \scalebox{0.7}{$2$} & \scalebox{0.7}{$3$} & \scalebox{0.7}{$\cdots$} & \scalebox{0.7}{$n-1$} & \scalebox{0.7}{$n$} \\
\begin{block}{c[cccccc]}
  \scalebox{0.7}{$12$} & M^{\xi_I}_{a_{12},a_{1}} & -M^{\xi_I}_{a_{12},a_{2}} & \mathbf{0} & \cdots & \mathbf{0} & \mathbf{0} \\
  \scalebox{0.7}{$13$} & M^{\xi_I}_{a_{13},a_{1}} & \mathbf{0} & -M^{\xi_I}_{a_{13},a_{3}} & \cdots & \mathbf{0} & \mathbf{0} \\
  \scalebox{0.7}{$\vdots$} & \vdots & \vdots & \vdots &  & \vdots & \vdots \\
  \scalebox{0.7}{$1n$} & M^{\xi_I}_{a_{1n},a_{1}} & \mathbf{0} & \mathbf{0} &\cdots & \mathbf{0} & -M^{\xi_I}_{a_{1n},a_{n}} \\
  \scalebox{0.7}{$23$} & \mathbf{0} & M^{\xi_I}_{a_{23},a_{2}} & -M^{\xi_I}_{a_{23},a_{3}} & \cdots & \mathbf{0} & \mathbf{0} \\
  \scalebox{0.7}{$\vdots$} & \vdots & \vdots & \vdots &  & \vdots & \vdots \\
  \scalebox{0.7}{$2n$} & \mathbf{0} & M^{\xi_I}_{a_{2n},a_{2}} & \mathbf{0} & \cdots & \mathbf{0} & -M^{\xi_I}_{a_{2n},a_{n}} \\
  \scalebox{0.7}{$\vdots$} & \vdots & \vdots & \vdots &  & \vdots & \vdots \\
  \scalebox{0.7}{$n-1,n$} & \mathbf{0} & \mathbf{0} & \mathbf{0} & \cdots & M^{\xi_I}_{a_{n-1,n},a_{n-1}} & -M^{\xi_I}_{a_{n-1,n},a_{n}}\\
\end{block}
\end{blockarray}
$}
\]
 and
\[
\scalebox{0.9}{$
\begin{blockarray}{ccccccccc}
& \scalebox{0.7}{$12$} & \scalebox{0.7}{$\cdots$} & \scalebox{0.7}{$1m$} & \scalebox{0.7}{$23$} & \scalebox{0.7}{$\cdots$} &\scalebox{0.7}{$2m$} & \scalebox{0.7}{$\cdots$} & \scalebox{0.7}{$m-1, m$}\\
\begin{block}{c[cccccccc]}
  \scalebox{0.7}{$1$} & M^{\xi_I}_{b_{1},b_{12}} & \cdots & M^{\xi_I}_{b_{1},b_{1m}} & \mathbf{0} & \cdots & \mathbf{0} &\cdots & \mathbf{0} \\
  \scalebox{0.7}{$2$} & -M^{\xi_I}_{b_{2},b_{12}} & \cdots & \mathbf{0} & M^{\xi_I}_{b_{2},b_{23}} & \cdots & M^{\xi_I}_{b_{2},b_{2m}}& \cdots & \mathbf{0} \\
  \scalebox{0.7}{$3$} & \mathbf{0} & \cdots & \mathbf{0} & -M^{\xi_I}_{b_{3},b_{23}}&\cdots & \mathbf{0} & \cdots & \mathbf{0}\\
  \scalebox{0.7}{$\vdots$} & \vdots & & \vdots & \vdots & &\vdots & & \vdots  \\
  \scalebox{0.7}{$m-1$} & \mathbf{0} & \cdots & \mathbf{0} & \mathbf{0} & \cdots& \mathbf{0} & \cdots & M^{\xi_I}_{b_{m-1},b_{m-1,m}}\\
  \scalebox{0.7}{$m$} & \mathbf{0} & \cdots & -M^{\xi_I}_{b_{m},b_{1m}}  & \mathbf{0} & \cdots &-M^{\xi_I}_{b_{m},b_{2m}} &\cdots & -M^{\xi_I}_{b_{m},b_{m-1,m}}\\
\end{block}
\end{blockarray}
$}
\]
 respectively.
\end{rmk}

\begin{rmk}
\label{rmk:redund-matrix}
Even in the case where $I^{\xi}$ satisfies the existence condition of pairwise joins in
$\src(I^{\xi})$ and meets in $\snk(I^{\xi})$, formula \eqref{eqn:general formula for (n,m)-type w/o conditions}
is still {redundant} because projective
presentations~\eqref{eq:min-proj-resol-V_I-nsrccase-lem w/o conditions} and
\eqref{eq:min-proj-resol-V_I/socV_I-nsrccase w/o conditions}
are not minimal in general if there are order relations between pairwise joins in $I^{\xi}$. {The next lemma identifies redundant rows and columns that can be removed without changing the $I$-rank formula.}
\end{rmk}

{Under the existence condition of pairwise joins, the rows of $\tdbfM$ in \eqref{eqn:general formula for (n,m)-type w/o conditions} are indexed by
the pairs $\{i,j\}\in\sub_2[n]$ as in \cref{rmk:pair-index-sets}.
}

\begin{lem}
\label{lem:degeneratedcase}
We keep the setting of \cref{thm:general w/o conditions}
and assume the {existence} condition of pairwise joins in $\src(I^{\xi})$ and meets in $\snk(I^{\xi})$. Let $\{i,j,k\} \subseteq [n]$
{with $i,j,k$ pairwise distinct}.
{Then the following statements are equivalent:}
\begin{enumerate}[font=\normalfont]
\item
$a_{ij} \le a_{ik}$;
\item
$a_j \le a_{ik}$;
\item
$a_i, a_j, a_k \le a_{ik}$.
\end{enumerate}
If one of the above holds, then formula {\eqref{eqn:general formula for (n,m)-type w/o conditions}} remains valid even if
we replace $\tdbfM$ in \cref{rmk:exist-cond-case}
with the matrix obtained by deleting the $\{i,k\}$ row of $\tdbfM$. {The dual statement holds} for $\hatbfM$.
\end{lem}

\begin{proof}
The equivalence of the three statements is trivial.
Now assume that one of them holds. Then all of them hold.
{Before applying the row operations below, multiply the rows labelled by
$\{i,j\}$, $\{i,k\}$ and $\{j,k\}$ by $-1$ if necessary, so that their blocks in
columns $(i,j)$, $(i,k)$ and $(j,k)$ have signs $(+,-)$, respectively.
Applying the same changes to the corresponding rows of the block matrix in \eqref{eqn:general formula for (n,m)-type w/o conditions} preserves both ranks and leaves $\hatbfM$ unchanged.}
By (3), we have $a_{ij}\leq a_{ik}$ and $a_{jk}\leq a_{ik}$. Thus there exist morphisms $p_{a_{ik},a_{ij}}$ and $p_{a_{ik},a_{jk}}$. {The following row operations preserve the ranks of $\tdbfM$ and the block matrix in \eqref{eqn:general formula for (n,m)-type w/o conditions}, when applied to the corresponding rows of both matrices} {(see $\tdbfM$ in \cref{rmk:exist-cond-case} for $(i,j,k)=(1,2,3)$)}:
\begin{itemize}
\item
To the $\{i,k\}$ row, add the row obtained from the $\{i,j\}$ row by
{left multiplication by} $-M^{\xi_I}_{a_{ik},a_{ij}}$.
\item
To the $\{i,k\}$ row, add the row obtained from the $\{j,k\}$ row by
{left multiplication by} $-M^{\xi_I}_{a_{ik},a_{jk}}$.
\end{itemize}
By these operations, the $\{i,k\}$ row of $\tdbfM$ becomes zero, and we can delete the $\{i,k\}$ row without changing the value of the right-hand side of {\eqref{eqn:general formula for (n,m)-type w/o conditions}. The dual column operations preserve the ranks of $\hatbfM$ and the block matrix, and leave $\tdbfM$ unchanged.}
\end{proof}

{For total compression on a 2D-grid, we simplify the general formula in \cref{thm:general w/o conditions} by keeping only the rows and columns for adjacent sources and sinks. These matrices will be used to construct a zigzag essential cover in \cref{sec:Essential cover relative to compression systems}.}

\begin{cor}[Specialization to 2D-grids]
\label{cor:2Dcase-tot-formula}
Let $\bfP$ be a 2D-grid and let $\xi = \tot$ be the total compression system for $A$ $(\coloneqq \k[\bfP])$. Take an interval $I\in \bbI$ with $\src({I}) = \{a_1,\ldots, a_n\}$ and $\snk({I}) = \{b_1,\ldots,b_m\}$. Without loss of generality, we assume that the first coordinate $($i.e., the $x$-coordinate {as in~\eqref{eq:G52}} of $a_{i}$ $($\textnormal{resp.}\ $b_{j}$$))$ is strictly less than that of $a_{i+1}$ $(i\in [n-1])$ $($\textnormal{resp.}\ $b_{j+1}$ $(j\in [m-1]))$. {Then the joins $a_{i,i+1}\coloneqq a_i\vee a_{i+1}$ and meets $b_{j,j+1}\coloneqq b_j\wedge b_{j+1}$ exist in $I$ for all $i\in[n-1]$ and $j\in[m-1]$, and $a_1\le b_1$. For every $M\in\mod A$,} we have
\begin{equation}
\label{eq:nsrc1snkrank2D}
\rank_{I}^{\tot} M=
\rank \bmat{
  \tdbfM & \mathbf{0} \\
  \checkbfM & \hatbfM \\
}
- \rank \bmat{
  \tdbfM & \mathbf{0} \\
  \mathbf{0} & \hatbfM \\
},
\end{equation}
{where the pair indices are written in the lexicographic order of
\cref{ntn: notations for general case} and} $\tdbfM$ has the form:
\[
\scalebox{0.95}{$
\begin{blockarray}{ccccccc}
& \scalebox{0.7}{$1$} & \scalebox{0.7}{$2$} & \scalebox{0.7}{$3$} & \scalebox{0.7}{$\cdots$} & \scalebox{0.7}{$n-1$} & \scalebox{0.7}{$n$} \\
\begin{block}{c[cccccc]}
  \scalebox{0.7}{$12$} & M_{a_{12},a_{1}} & -M_{a_{12},a_{2}} & \mathbf{0} & \cdots & \mathbf{0} & \mathbf{0} \\
  \scalebox{0.7}{$23$} & \mathbf{0} & M_{a_{23},a_{2}} & -M_{a_{23},a_{3}} & \cdots & \mathbf{0} & \mathbf{0} \\
  \scalebox{0.7}{$\vdots$} & \vdots & \vdots & \vdots &  & \vdots & \vdots \\
  \scalebox{0.7}{$n\!-\!2,n\!-\!1$} & \mathbf{0} & \mathbf{0} & \mathbf{0} & \cdots & -M_{a_{n-2,n-1},a_{n-1}} & \mathbf{0} \\
  \scalebox{0.7}{$n\!-\!1,n$} & \mathbf{0} & \mathbf{0} & \mathbf{0} & \cdots & M_{a_{n-1,n},a_{n-1}} & -M_{a_{n-1,n},a_{n}} \\
\end{block}
\end{blockarray}
$},
\]
$\checkbfM$ has the form:
\[
\scalebox{0.95}{$
\begin{blockarray}{ccccccc}
& \scalebox{0.7}{$1$} & \scalebox{0.7}{$2$} & \scalebox{0.7}{$3$} & \scalebox{0.7}{$\cdots$} & \scalebox{0.7}{$n-1$} & \scalebox{0.7}{$n$} \\
\begin{block}{c[cccccc]}
  \scalebox{0.7}{$1$} & M_{b_{1},a_{1}} & \mathbf{0} & \mathbf{0} & \cdots & \mathbf{0} & \mathbf{0} \\
  \scalebox{0.7}{$2$} & \mathbf{0} & \mathbf{0} & \mathbf{0} & \cdots & \mathbf{0} & \mathbf{0} \\
  \scalebox{0.7}{$\vdots$} & \vdots & \vdots & \vdots &  & \vdots & \vdots \\
  \scalebox{0.7}{$m-1$} & \mathbf{0} & \mathbf{0} & \mathbf{0} & \cdots & \mathbf{0} & \mathbf{0} \\
  \scalebox{0.7}{$m$} & \mathbf{0} & \mathbf{0} & \mathbf{0} & \cdots & \mathbf{0} & \mathbf{0} \\
\end{block}
\end{blockarray}
$},
\]
and $\hatbfM$ has the form:
\[
\scalebox{0.95}{$
\begin{blockarray}{ccccc}
& \scalebox{0.7}{$12$} & \scalebox{0.7}{$23$} & \scalebox{0.7}{$\cdots$} & \scalebox{0.7}{$m-1, m$} \\
\begin{block}{c[cccc]}
  \scalebox{0.7}{$1$} & M_{b_{1},b_{12}} & \mathbf{0} & \cdots & \mathbf{0} \\
  \scalebox{0.7}{$2$} & -M_{b_{2},b_{12}} & M_{b_{2},b_{23}} & \cdots & \mathbf{0} \\
  \scalebox{0.7}{$3$} & \mathbf{0} & -M_{b_{3},b_{23}} & \cdots & \mathbf{0} \\
  \scalebox{0.7}{$\vdots$} & \vdots & \vdots & & \vdots \\
  \scalebox{0.7}{$m-1$} & \mathbf{0} & \mathbf{0} & \cdots & M_{b_{m-1},b_{m-1,m}} \\
  \scalebox{0.7}{$m$} & \mathbf{0} & \mathbf{0} & \cdots & -M_{b_{m},b_{m-1,m}}\\
\end{block}
\end{blockarray}
$}.
\]
\end{cor}

{\begin{proof}
Write $a_i=(x_i,y_i)$ and $b_j=(x'_j,y'_j)$. Since the sources and sinks are antichains, $y_1>\cdots>y_n$ and $y'_1>\cdots>y'_m$.
Fix $i\in[n-1]$. The sets $U\coloneqq\bigcup_{r=1}^{i}\uset_I a_r$ and $V\coloneqq\bigcup_{r=i+1}^{n}\uset_I a_r$ are nonempty upsets whose union is $I$.
Since both sets are upsets, they must intersect: otherwise there would be no comparable elements between them, contrary to the connectedness of $I$.
Take $u\in U\cap V$, so that $a_r,a_s\le u$ for some $r\le i<s$.
The coordinate orders give $a_i\le(x_{i+1},y_i)\le u$, and hence $(x_{i+1},y_i)\in I$ by convexity.
This point is $a_i\vee a_{i+1}$ in $I$.
The dual argument gives $b_j\wedge b_{j+1}=(x'_j,y'_{j+1})\in I$ for all $j\in[m-1]$.
Moreover, choose $a_r\le b_1$ and $a_1\le b_s$. Then $x_1\le x_r\le x'_1$ and $y_1\le y'_s\le y'_1$, so $a_1\le b_1$.

For $i<k$, put $a_{ik}\coloneqq(x_k,y_i)$, the join of $a_i$ and $a_k$ in $\bfP$.
If $u\in I$ is a common upper bound of $a_i$ and $a_k$, then $a_i\le a_{ik}\le u$, so $a_{ik}\in I$ by convexity.
Thus the pre-join $\vee'\{a_i,a_k\}$ in $I$ is $\{a_{ik}\}$ if $a_{ik}\in I$, and is empty otherwise. The dual statement holds for the pre-meets of sinks.

Apply \cref{thm:general w/o conditions} with $\xi=\tot$, and denote its full matrices by $\tdbfM_0$ and $\hatbfM_0$.
For each $i<k$ with $a_{ik}\in I$, denote by $\tdbfM_{0,ik}$ the block row of $\tdbfM_0$ indexed by $(\{a_i,a_k\},a_{ik})$.
Since $a_{\ell,\ell+1}\le a_{ik}$ for $i\le\ell<k$, functoriality gives
\[
\tdbfM_{0,ik}=\sum_{\ell=i}^{k-1}M_{a_{ik},a_{\ell,\ell+1}}\tdbfM_{0,\ell,\ell+1}.
\]
Indeed, the terms in the intermediate columns cancel, leaving $M_{a_{ik},a_i}$ in column $i$ and $-M_{a_{ik},a_k}$ in column $k$.
Consequently, all nonadjacent block rows can be deleted without changing the rank of $\tdbfM_0$ or of the block matrix in \eqref{eqn:general formula for (n,m)-type w/o conditions}, since these rows have zero blocks on the right.
The dual column operations delete all nonadjacent block columns of $\hatbfM_0$ without changing the rank of $\hatbfM_0$ or of the block matrix, and leave $\checkbfM$ unchanged.
The remaining matrices are $\tdbfM$ and $\hatbfM$ displayed above, which proves \eqref{eq:nsrc1snkrank2D}.
The cases $n=1$ or $m=1$ are included by the empty-matrix convention in \cref{thm:general w/o conditions}.
\end{proof}}

{For $(n,m) = (2,2)$, the general formula takes the explicit matrix form displayed below.}

\begin{exm}
\label{exm:2-2case-general}
Let $M \in \mod A$ and $I^{\xi}$ be
of $(2,2)$-type with
$\src({I^{\xi}}) = \{a_1,a_2\}$ and $\snk({I^{\xi}}) = \{b_1, b_2\}$.
Assume that both $x\coloneqq a_1 \vee a_2$ and $y\coloneqq b_1 \wedge b_2$ exist.
Since $\snk({I^{\xi}}) = \{b_1, b_2\}$, we have $a_1 \le b_1$ or $a_1 \le b_2$,
and hence we may assume that $a_1 \le b_1$ without loss of generality. Then we have
\[
\begin{aligned}
{\rank_I^\xi M}
&=
\rank \scalebox{0.85}{$\bmat{{M^{\xi_I}_{x,a_1}} & -{M^{\xi_I}_{x,a_2}} & \mathbf{0}\\
{M^{\xi_I}_{b_1,a_1}} & \mathbf{0} & {M^{\xi_I}_{b_1,y}}\\
\mathbf{0} & \mathbf{0} & -{M^{\xi_I}_{b_2,y}}
}$}
- \rank \scalebox{0.9}{${\bmat{
{M^{\xi_I}_{x,a_1}} & -{M^{\xi_I}_{x,a_2}}}}$} - \rank \scalebox{0.9}{$\bmat{{M^{\xi_I}_{b_1,y}}\\-{M^{\xi_I}_{b_2,y}}}$}\\
&= \rank \scalebox{0.9}{$\bmat{{M^{\xi_I}_{x,a_1}} & {M^{\xi_I}_{x,a_2}} & \mathbf{0}\\
{M^{\xi_I}_{b_1,a_1}} & \mathbf{0} & {M^{\xi_I}_{b_1,y}}\\
\mathbf{0} & \mathbf{0} & {M^{\xi_I}_{b_2,y}}
}$}
- \rank \scalebox{0.9}{$\bmat{
{M^{\xi_I}_{x,a_1}} & {M^{\xi_I}_{x,a_2}}}$}
-\rank \scalebox{0.9}{$\bmat{{M^{\xi_I}_{b_1,y}}\\{M^{\xi_I}_{b_2,y}}}$}.
\end{aligned}
\]
\end{exm}

\begin{rmk}
\label{rmk:code}
{We provide MPIntInv, a program that computes the interval rank invariant and the interval replacement under the total and source-sink compression systems for persistence modules over any $d$D-grid, mainly based on \cref{thm:general w/o conditions,rmk:exist-cond-case}. The code and implementation details are available at \url{https://github.com/Enhao-Liu/MPIntInv}.}
\end{rmk}

\section{Essential covers relative to compression systems}
\label{sec:Essential cover relative to compression systems}

{Essential covers provide a way to compute an interval rank from a restriction of a persistence module to another poset.
\cref{thm:ess-cov-int-rk-inv} gives this reduction, and \cref{cor:sufficient condition when two compression systems induce the same interval rank invariants} uses it to compare the invariants induced by two compression systems.
Both results apply to finite posets in general, including higher-dimensional grids.}

{We first express the interval-rank formula by matrices of morphisms and use them to define essential covers. After proving the two results, we illustrate how to choose covers in small examples.
We then apply the technique to total compression on 2D-grids, where a suitable cover can always be chosen to be a zigzag.}

\subsection{{Matrices and essential covers}}

{We use the formal additive hull of a linear category $B$ to consider matrices whose entries are morphisms in $B$. Allowing empty matrices unifies the formulas for all cases.}

Roughly speaking, the formal additive hull $\Ds B$ of $B$ is defined
as follows:
{The objects are all finite sequences}
$(x_i)_{i \in [l]} = (x_1,\dots, x_l)$ with $x_1, \dots, x_l \in B_0$
and $l \ge 0$.
For any $x = (x_i)_{i \in [l]}$, $y = (y_j)_{j \in [m]}$,
the set of morphisms from $x$ to $y$ is defined as the set of matrices
$\bmat{\al_{ji}}_{(j,i) \in [m]\times [l]}$, where
$\al_{ji} \in B(x_i, y_j)$ for all $(j,i) \in [m]\times [l]$.
The composition is given by the usual matrix multiplication.
See \cref{sec:form-add-hull} for details and empty matrices.

{The following proposition justifies evaluating a persistence module on a matrix of morphisms in the incidence category.}

{For each additive linear category $C$ and each finite family $X=(X_i)_{i\in[n]}$ of objects of $C$, with $n\ge0$,} we choose its biproduct
$\Ds_{i \in [n]}X_i$ with the projections $\pi_i^X \colon \Ds_{i \in [n]}X_i \to X_i$ and the injections $\si_i^X \colon X_i \to \Ds_{i \in [n]}X_i$ once and for all.
Here, if $n = 0$, then $\Ds_{i \in [n]}X_i$ is a chosen zero object $0_{C}$, and if $n=1$, then $\Ds_{i \in [n]}X_i\coloneqq  X_1$, $\pi_1^X\coloneqq  \id_{X_1}$, $\si_1^X\coloneqq  \id_{X_1}$.

{For the formal additive hull $\Ds B$ of a linear category $B$}, we choose the canonical
biproducts given by concatenation of finite sequences,
with the standard injections and projections and zero object $()$.

\begin{prp}
\label{prp:univ-formal-add-hull}
Let $B$ be a linear category and $C$ an additive linear category.
Then each linear functor $F \colon B \to C$ uniquely extends to a linear functor $\ovl{F} \colon \Ds B \to C$
such that, for each finite sequence $(x_i)_{i\in [n]}$ with $n \ge 0$,
\[
\overline{F}((x_i)_{i\in [n]})
=\bigoplus_{i=1}^n F(x_i),
\]
and $\overline{F}$ sends the canonical projections and injections of $(x_i)_{i\in [n]}$ to those of the chosen biproduct.
\end{prp}

\begin{proof}
For each morphism
\begin{equation}
\label{eq:morphism-in-DsB}
\al\coloneqq \bmat{\al_{ji}}_{(j,i)\in [n]\times [m]}
\colon (x_i)_{i\in [m]} \to (y_j)_{j \in [n]}
\end{equation}
in $\Ds B$, we set
$$
\ovl{F}(\al)\coloneqq [F(\al_{ji})]_{(j,i)\in [n]\times [m]} \colon \Ds_{i\in [m]} F(x_i) \to \Ds_{j\in [n]}F(y_j).
$$
This defines a linear functor $\ovl{F}$.

For uniqueness, let $G\colon\Ds B\to C$ be another extension
satisfying the stated conditions.
Put $X=(F(x_i))_{i\in[m]}$ and $Y=(F(y_j))_{j\in[n]}$.
Applying $G$ to the canonical matrix decomposition 
$\al = \sum _{j\in[n],\,i\in[m]}
  \si_j^y \al_{ji} \pi_i^x$
of $\al$ gives
\[
G(\al)
=\sum_{j\in[n],\,i\in[m]}
  \si_j^Y F(\al_{ji})\pi_i^X
=\ovl{F}(\al).
\]
Since the values on objects are also prescribed, $G=\ovl{F}$.
\end{proof}

Different choices of biproducts yield naturally isomorphic extensions, with a unique natural isomorphism restricting to $1_F$ on $B$.
In what follows, we use the extension constructed in the proof
and denote it by $F$ {if no confusion can arise}.

\begin{rmk}
Proposition \ref{prp:univ-formal-add-hull} has the following applications.

(1)
If $B$ is additive, then 
the identity $\id_B \colon B \to B$ is uniquely extended to a functor
$\et_B\coloneqq \ovl{\id_B} \colon \Ds B \to B$, which turns out to be an equivalence.
It sends each morphism $\al$ in \eqref{eq:morphism-in-DsB} to
$\bmat{\al_{ji}}_{(j,i)\in [n]\times [m]}
\colon \Ds_{i\in [m]} x_i \to \Ds_{j \in [n]} y_j$ in $B$.
In particular, it sends $()$ to 0.

(2)
Let $F \colon B \to C$ be a linear functor between linear categories, and $\si_C\colon C \to \Ds C$ the canonical embedding (see Definition \ref{dfn:formal-add-hull}).
Then the functor $\si_C \circ F \colon B \to \Ds C$ is uniquely extended to
a functor $\Ds F\coloneqq \ovl{\si_C \circ F} \colon \Ds B \to \Ds C$.
It sends each morphism $\al$ in \eqref{eq:morphism-in-DsB} to
$$
(\Ds F)(\al)\coloneqq [F(\al_{ji})]_{(j,i)\in [n]\times [m]} \colon (F(x_i))_{i\in [m]} \to (F(y_j))_{j\in [n]}.
$$
In particular, $(\Ds F)(()) = ()$, and 
$(\Ds F)(\sfJ)\coloneqq \sfJ$  for all $\sfJ \in \{\sfJ_{n,0}, \sfJ_{0,m} \mid m, n \ge 0\}$.
For example, $\sfJ_{0,m} \colon (x_i)_{i \in [m]} \to ()$ is sent to
$\sfJ_{0,m} \colon (F(x_i))_{i \in [m]} \to ()$.
If there is no confusion, we denote $\Ds F$ simply by $F$.

(3)
Assume that for any objects $x, y$ in $B$, $B(x,y)$ is finite-dimensional and $B(x, x)$ is a local algebra (for example, $B = \k[S]$ for a poset $S$).
Then $\Ds B$ is equivalent to the category $\prj B\op$ of 
finitely generated projective right $B$-modules.
An equivalence is given by the unique extension 
$\ovl{Y_B} \colon \Ds B \to \prj B\op$ of the Yoneda embedding $Y_B \colon B \to \prj B\op$. It sends each morphism $\al$ in \eqref{eq:morphism-in-DsB}
to
$$
\bmat{B(\blank, \al_{ji})}_{(j,i)\in [n]\times [m]} \colon \Ds_{i\in [m]}B(\blank, x_i) \to \Ds_{j\in [n]}B(\blank, y_j).
$$
\end{rmk}
\removelastskip

\begin{ntn}
\label{ntn:W(g)}
Let $B$ be a linear category, $W$ a $B$-module, and $m, n$ nonnegative integers,
and consider a morphism $\bfg = \bmat{g_{ji}}_{(j,i)\in [n]\times [m]}
\colon (x_i)_{i\in [m]} \to (y_j)_{j \in [n]}$ 
in $\Ds B$.
Then by applying {the convention following} \cref{prp:univ-formal-add-hull}
in the case where $C = \mod\k$,
we write
\[
W(\bfg)\coloneqq {\ovl W(\bfg)} = \bmat{W({g_{ji}})}_{j,i} \colon
\Ds_{i \in [m]} W(x_i) \to \Ds_{j \in [n]} W(y_j).
\]
\end{ntn}

\begin{exm}
\label{exm:Ds-ze}
Let ${\om} \colon \bfZ \to \bfP$ be an order-preserving map between posets.
Then by \eqref{eq:k_linear_ext_ob} and \eqref{eq:k_linear_ext_mor} we have a linear functor $\k[{\om}] \colon \k[\bfZ] \to \k[\bfP]$,
which yields a linear functor
$\Ds \k[{\om}] \colon \Ds \k[\bfZ] \to \Ds \k[\bfP]$.
If $\al\coloneqq [\al_{ji}]_{(j,i)\in [n]\times [m]}$ is a morphism in $\Ds \k[\bfZ]$,
we denote $(\Ds \k[{\om}])(\al)$ simply by ${\om}(\al) = [{\om}(\al_{ji})]_{(j,i)\in [n]\times [m]}$.
\end{exm}

{Recall that $R_I(M)=M\circ\k[\xi_I]$ for $M\in\mod A$.
Using \cref{ntn:W(g)}, we express the formula in \cref{thm:general w/o conditions} by a matrix of morphisms in $\k[I^\xi]$. This provides a matrix to be lifted along an essential cover.}

{With the numberings and lexicographic orders of \cref{ntn: notations for general case},
we regard $\src(I^\xi)$ and $\snk(I^\xi)$ as objects of $\Ds\k[I^\xi]$.
We use the same notation $\src_1(I^\xi)$ (resp.\ $\snk_1(I^\xi)$) for the sequence having
one entry $c$ (resp.\ $d$) for each index $\bfa_c$ (resp.\ $\bfb_d$), in that order.}
Then we can consider a morphism $\bfg \colon \src(I^\xi) \ds \snk_1(I^\xi)
\to \src_1(I^\xi)\ds \snk(I^\xi)$ in $\Ds \k[I^\xi]$ defined in the following theorem.

\begin{thm}
\label{thm:restate-int-rk-inv-formula}
Let $\xi = \left(\xi_I \colon I^{\xi} \to \bfP\right)_{I\in \bbI}$ be a compression system. Fix an interval $I$ of $\bfP$. Choose any $(b, a) \in \snk(I^\xi) \times \src(I^\xi)$ with $(a, b)\in [\leq]_{I^\xi}$, and set 
\[
{\bfg\coloneqq}\left[
\begin{array}{c|c}
\bfg_1 & \bfzero\\
\hline
\bfg_3 & \bfg_2
\end{array}
\right],
\]
where $\bfg_1 \coloneqq \bmat{{\tilde{p}_{\bfa_c,a}}}_{(\bfa_c, a) \in \src_1(I^\xi) \times \src(I^\xi)}$ with the entries given by
\[
\begin{aligned}
{\tilde{p}_{\bfa_c,a}}\coloneqq
\begin{cases}
    p_{c,a}, & \textnormal{if } a = \udl{\bfa},\\
    -p_{c,a}, & \textnormal{if } a = \ovl{\bfa},\\
    \mathbf{0}, & \textnormal{if } a \not\in \bfa,
\end{cases}
\end{aligned}
\]
for all $\bfa_c \in \src_1(I^\xi)$ and $a \in \src(I^\xi)$; and $\bfg_2 \coloneqq \bmat{\hat{p}_{b,\bfb_d}}_{{(b,\bfb_d)} \in \snk(I^\xi) \times \snk_1(I^\xi)}$ with the entries given by
$$
\begin{aligned}
\hat{p}_{b,\bfb_d}\coloneqq
    \begin{cases}
        p_{b,d}, & \textnormal{if } b = \udl{\bfb},\\
        -p_{b,d}, & \textnormal{if } b = \ovl{\bfb},\\
        \mathbf{0}, & \textnormal{if } b \not\in \bfb,
    \end{cases}
\end{aligned}
$$
for all $b \in \snk(I^\xi)$ and $\bfb_d \in \snk_1(I^\xi)${. The matrix $\bfg_3$ has size $|\snk(I^\xi)|\times|\src(I^\xi)|$; its only nonzero entry is $p_{b,a}$ in position $(b,a)$.} Then for any $M \in \mod A$ we have
\begin{equation}
\label{eq:formula-int-rk-inv-general-mor}
{\rank_I^\xi M} = \rank R_{I}(M)(\bfg) - \rank R_{I}(M)(\bfg_1) - \rank R_{I}(M)(\bfg_2).
\end{equation}
\end{thm}

{The sequences $\src_1(I^\xi)$ and $\snk_1(I^\xi)$ may be empty.}
To deal with these cases, we make the following remark.

\begin{rmk}
\label{rmk:usage-empty-mat}
Let $B$ be a linear category, and $x, x', y, {y'} \in (\Ds B)_0$.
Consider a morphism $\bfg = \bmat{\bfg_{11} & \bfg_{12}\\\bfg_{21} & \bfg_{22}}
\colon x \ds x' \to y \ds y'$ in $\Ds B$.

(1) If $x' = ()$, then $\bmat{\bfg_{12}\\\bfg_{22}} = \sfJ_{(|y|+|y'|), 0}$,
and we have $\bfg = \bmat{\bfg_{11}\\\bfg_{21}} \colon x \to y \ds y'$.

(2) If $y = ()$, then $\bmat{\bfg_{11} & \bfg_{12}} = \sfJ_{0, (|x|+|x'|)}$, and we have
$\bfg = \bmat{\bfg_{21} & \bfg_{22}} \colon {x\ds x' \to y'}$.

{(3) For matrices of linear maps, these observations agree with the zero-dimensional matrix convention used in \cref{thm:general w/o conditions}; see also \cref{exm:Ds-k}.}
\end{rmk}

{To compute an $I$-rank after restriction to another poset, we require that a matrix giving the formula can be lifted to that poset. The following definitions make this requirement precise.}

\begin{dfn}
\label{dfn:covers}
Let ${\om} \colon \bfZ \to \bfP$ be an order-preserving map, and $\al\colon x \to y$ a morphism in $\Ds \k[\bfP]$. We say that ${\om}$ \emph{covers} $\al$ if there exists a morphism $\al' \colon x' \to y'$ in $\Ds\k[\bfZ]$ such that ${\om}(\al') = \al$ (see \cref{exm:Ds-ze} for ${\om}(\al')$).
\end{dfn}

\begin{dfn}
\label{dfn:ess-cov2-int-rk-inv}
Let $\xi = \left(\xi_I \colon I^{\xi} \to \bfP\right)_{I\in \bbI}$ be a compression system, and $I$ an interval of $\bfP$.

(1) A morphism $\bfg\coloneqq \bmat{\bfg_1 & \mathbf{0}\\ \bfg_3 & \bfg_2}$ in $\Ds\k[I^\xi]$ is called an $I$-{\emph{rank matrix}} under $\xi$
if for any $M \in \mod A$,  we have a formula
\begin{equation}
\label{eq:comp-multiplicity-formula-ess}
{\rank_I^\xi M} = \rank \bmat{R_{I}(M)(\bfg_1) & \mathbf{0}\\
R_{I}(M)(\bfg_3) & R_{I}(M)(\bfg_2)\\
}
- \rank \bmat{R_{I}(M)(\bfg_1) & \mathbf{0}\\
\bfzero & R_{I}(M)(\bfg_2)\\
}.
\end{equation}

(2) Let ${\om} \colon \bfZ \to \bfP$ be an order-preserving map.
We say that ${\om}$ \emph{essentially covers} $I$ \emph{relative to $\xi$}
(or that ${\om}$ is an \emph{essential cover} of $I$ \emph{relative to $\xi$})
if there exists an order-preserving map ${\om_{I}}\colon \bfZ\to I^\xi$ that makes
the diagram
\[
\begin{tikzcd}[ampersand replacement=\&]
	\bfZ \&\& \bfP \\
	\& I^\xi
	\arrow["{\om}", from=1-1, to=1-3]
	\arrow["{{\om_I}}"', dashed, from=1-1, to=2-2]
	\arrow["{\xi_I}"', from=2-2, to=1-3]
\end{tikzcd}
\]
commutative, and covers an {$I$-rank matrix} $\bfg$ under $\xi$.
\end{dfn}

We remark here that~\cref{thm:restate-int-rk-inv-formula} guarantees the existence
of an {$I$-rank matrix} $\bfg$ under $\xi$.
We also caution the reader that this $\bfg$ is not unique in general, for example, due to the redundancy explained in~\cref{rmk:redund-matrix} and \cref{lem:degeneratedcase}.

{To recover the $I$-rank from an essential cover, we will use additivity of the matrix ranks with respect to direct sums of modules.}

\begin{lem}
\label{lem:direct-sum-rank}
Let $B$ be a linear category, $W$ a $B$-module, and $m, n$ nonnegative integers.
For each matrix $\bfg = [g_{ji}]_{(j,i)\in [n]\times [m]}$ 
{whose entries}
{$g_{ji}\colon x_i\to y_j$ are morphisms} in $B$,
we set
\[
W(\bfg)\coloneqq [W(g_{ji})]_{j,i}\colon
\Ds_{i \in [m]} W(x_i) \to \Ds_{j \in [n]} W(y_j)
\]
to be the linear map expressed by this matrix.
Assume that we have a direct sum decomposition $W \iso W_1 \ds W_2$
of $B$-modules.
Then we have an equivalence $W(\bfg) \iso W_1(\bfg) \ds W_2(\bfg)$
of linear maps.
In particular, the equality 
\[
\rank W(\bfg) = \rank W_1(\bfg) + \rank W_2(\bfg)\]
holds.
\end{lem}

\begin{proof}
Let $f \colon W \to W_1 \ds W_2$ be an isomorphism of $B$-modules.
Then for any $i \in [m], j \in [n]$, we have a commutative diagram
\[
\begin{tikzcd}[column sep=70pt]
W(x_i) & \Nname{y}W(y_j)\\
\Nname{x12}W_1(x_i) \ds W_2(x_i) & W_1(y_j) \ds W_2(y_j)
\Ar{1-1}{x12}{"f_{x_i}" '}
\Ar{y}{2-2}{"f_{y_j}"}
\Ar{1-1}{y}{"W(g_{ji})"}
\Ar{x12}{2-2}{"W_1(g_{ji}) \ds W_2(g_{ji})" '}
\end{tikzcd}
\]
with both $f_{x_i}$ and $f_{y_j}$ isomorphisms.
This yields the commutative diagram
\[
\begin{tikzcd}[column sep=90pt]
\Nname{x}\DDs_{i\in [m]} W(x_i) & \Nname{y}\DDs_{j\in [n]}W(y_j)\\
\Nname{x12}\DDs_{i\in [m]}(W_1(x_i) \ds W_2(x_i)) & \Nname{y12}\DDs_{j\in [n]}(W_1(y_j) \ds W_2(y_j))\\
\Nname{+x12}(\DDs_{i\in [m]}W_1(x_i)) \ds (\DDs_{i\in [m]}W_2(x_i)) &\Nname{+y12}(\DDs_{j\in [n]}W_1(y_j)) \ds (\DDs_{j\in [n]}W_2(y_j))
\Ar{x}{x12}{"\DDs_{i\in [m]}f_{x_i}" '}
\Ar{x12}{+x12}{"\si_x" '}
\Ar{y12}{+y12}{"\si_y"}
\Ar{y}{2-2}{"\DDs_{j\in [n]}f_{y_j}"}
\Ar{x}{y}{"{[W(g_{ji})]_{j,i}}","= W(\bfg)" '}
\Ar{x12}{y12}{"{[W_1(g_{ji}) \ds W_2(g_{ji})]_{j,i}}" '}
\Ar{+x12}{+y12}{"{[W_1(g_{ji})]_{j,i} \ds [W_2(g_{ji})]_{j,i}}", "= W_1(\bfg)\ds W_2(\bfg)" '}
\end{tikzcd},
\]
where $\si_x, \si_y$ are given by the permutation matrices corresponding to the permutation
$\si_k$ (for $k = m, n$, respectively) of the set $[2k]$ defined by
\[
\si_k(i)\coloneqq \begin{cases}
    \ell, & \textnormal{if } i = 2\ell - 1 \textnormal{ for some } \ell \in [k],\\
    k + \ell, & \textnormal{if } i = 2\ell \textnormal{ for some } \ell \in [k]
\end{cases}
\textnormal{ for all } i \in [2k].
\]
{The nonzero entries of $\si_x$ and $\si_y$ are identity maps.}
Then since all vertical maps above are isomorphisms, the assertion holds.
\end{proof}

{The restriction of $V_I$ along an essential cover need not be indecomposable. We therefore use the following extension of summand multiplicity.}

\begin{ntn}\label{def:ds-counting}
{Let $B$ be a finite $\k$-linear category and let $M,L\in\mod B$ with $L\ne0$.
We denote by $\bar d_M(L)$ the largest integer $n\ge0$ such that $M\iso L^{n}\ds N$ for some $N\in\mod B$.
In particular, if $L$ is indecomposable, then $\bar d_M(L)=d_M(L)$.
If $L\iso\Ds_{i\in[m]}L_i^{n_i}$, where the $L_i$ are pairwise non-isomorphic indecomposable modules and $n_i\ge1$, then the Krull--Schmidt theorem gives
$\bar d_M(L)=\min_{i\in[m]}\lfloor d_M(L_i)/n_i\rfloor$.}
\end{ntn}

\subsection{{Recovering and comparing interval ranks}}

From now on, the restriction functor induced by ${\om}$ will be denoted by $R_{{\om}}$. {The next theorem recovers the $I$-rank from the maximal number of copies of $R_{\om}(V_I)$ that split off from $R_{\om}(M)$.}

\begin{thm}
\label{thm:ess-cov-int-rk-inv}
Let $\xi = \left(\xi_I \colon I^{\xi} \to \bfP\right)_{I\in \bbI}$ be a compression system. Fix an interval $I$ of $\bfP$ and let ${\om}\colon \bfZ\to \bfP$ be an order-preserving map that essentially covers $I$ relative to $\xi$. Then for every $M\in \mod A$ we have
\begin{align}
\label{eq:ess-cov-int-rk-inv}
{\rank_I^\xi M} = \bar{d}_{R_{{\om}}(M)}(R_{{\om}}(V_I)).
\end{align}
\end{thm}

\begin{proof}
{Set $r\coloneqq \rank_I^\xi M$ and $s\coloneqq \bar d_{R_\om(M)}(R_\om(V_I))$.
By the definition of $I$-rank, we have $R_I(M)\iso R_I(V_I)^{r}\ds N$ for some $N$.
Choose an order-preserving map $\om_I\colon\bfZ\to I^\xi$ as in \cref{dfn:ess-cov2-int-rk-inv} (2), so that $\om=\xi_I\circ\om_I$ and $\om_I$ covers an $I$-rank matrix under $\xi$.
Applying $R_{\om_I}$ to the decomposition gives $R_\om(M)\iso R_\om(V_I)^{r}\ds R_{\om_I}(N)$, and hence $s\ge r$.}
	
On the other hand, since $s = \bar{d}_{R_{{\om}}(M)}(R_{{\om}}(V_I))$, we can write $R_{{\om}}(M) \iso \left[R_{{\om}}(V_I)\right]^s \ds L$ for some module $L$ in $\mod \k[\bfZ]$. {Choose an $I$-rank matrix $\bfg = \bmat{\bfg_1 & \mathbf{0}\\ \bfg_3 & \bfg_2}$ under $\xi$ that is covered by $\om_I$, whose existence is guaranteed by \cref{dfn:ess-cov2-int-rk-inv} (2).} Because ${\om_{I}}$ covers $\bfg$, there exists a morphism $\bfg' = \bmat{\bfg'_1 & \mathbf{0}\\ \bfg'_3 & \bfg'_2}$ in $\Ds\k[\bfZ]$ such that ${\om_{I}}(\bfg') \coloneqq \k[{\om_{I}}](\bfg') = \bfg$. Then by applying \cref{lem:direct-sum-rank} to the isomorphism $R_{{\om}}(M) \iso \left[R_{{\om}}(V_I)\right]^s \ds L$, we have the following equalities:
\begin{equation}
\label{eq:eq:1stRI(M)_case-int-rk}
\begin{aligned}
\rank\bmat{R_{{\om}}(M)(\bfg'_{1}) & \mathbf{0}\\
R_{{\om}}(M)(\bfg'_3) & R_{{\om}}(M)(\bfg'_2)}
=&\, s \rank \bmat{R_{{\om}}(V_I)(\bfg'_{1}) & \mathbf{0}\\
R_{{\om}}(V_I)(\bfg'_3) & R_{{\om}}(V_I)(\bfg'_2)}\\
&+\rank \bmat{L(\bfg'_{1}) & \mathbf{0}\\
L(\bfg'_3) & L(\bfg'_2)
},
\end{aligned}
\end{equation}
and
\begin{equation}
\label{eq:2ndRI(M)_case-int-rk}
\begin{aligned}
\rank\bmat{R_{{\om}}(M)(\bfg'_{1}) & \mathbf{0}\\
\mathbf{0} & R_{{\om}}(M)(\bfg'_2)\\
}
=&\, s \rank\bmat{R_{{\om}}(V_I)(\bfg'_{1}) & \mathbf{0}\\
\mathbf{0} & R_{{\om}}(V_I)(\bfg'_2)
}\\
&+ \rank \bmat{L(\bfg'_{1}) & \mathbf{0}\\
\mathbf{0} & L(\bfg'_2)
}.
\end{aligned}
\end{equation}
Note that $R_{{\om}}(M) = R_{{\om_{I}}}(R_{I}(M)) = R_{I}(M)\circ \k[{\om_{I}}]$ for all $M\in \mod A$. Then \eqref{eq:eq:1stRI(M)_case-int-rk} and \eqref{eq:2ndRI(M)_case-int-rk} become
\begin{equation}
\label{eq:eq:1stRI(M)_case-int-rk-re}
\begin{aligned}
\rank\bmat{R_{I}(M)(\bfg_{1}) & \mathbf{0}\\
R_{I}(M)(\bfg_3) & R_{I}(M)(\bfg_2)}
=&\, s \rank \bmat{R_{I}(V_I)(\bfg_{1}) & \mathbf{0}\\
R_{I}(V_I)(\bfg_3) & R_{I}(V_I)(\bfg_2)}\\
&+\rank \bmat{L(\bfg'_{1}) & \mathbf{0}\\
L(\bfg'_3) & L(\bfg'_2)
},
\end{aligned}
\end{equation}
and
\begin{equation}
\label{eq:2ndRI(M)_case-int-rk-re}
\begin{aligned}
\rank\bmat{R_{I}(M)(\bfg_{1}) & \mathbf{0}\\
\mathbf{0} & R_{I}(M)(\bfg_2)\\
}
=&\, s \rank\bmat{R_{I}(V_I)(\bfg_{1}) & \mathbf{0}\\
\mathbf{0} & R_{I}(V_I)(\bfg_2)
}\\
&+ \rank \bmat{L(\bfg'_{1}) & \mathbf{0}\\
\mathbf{0} & L(\bfg'_2)
}.
\end{aligned}
\end{equation}
By applying formula {\eqref{eq:comp-multiplicity-formula-ess}}
to $M = V_I$, we have the equality
\begin{equation}
\label{eq:written-formula-by-mor-Z-for-VI-int-rk}
   d_{R_{I}(V_I)}(R_{I}(V_I)) =
\rank \bmat{R_{I}(V_I)(\bfg_1) & \mathbf{0}\\
R_{I}(V_I)(\bfg_3) & R_{I}(V_I)(\bfg_2)\\
}
- \rank \bmat{R_{I}(V_I)(\bfg_1) & \mathbf{0}\\
\bfzero & R_{I}(V_I)(\bfg_2)\\
}.
\end{equation}
{Subtracting \eqref{eq:2ndRI(M)_case-int-rk-re} from \eqref{eq:eq:1stRI(M)_case-int-rk-re} and using \eqref{eq:comp-multiplicity-formula-ess} and \eqref{eq:written-formula-by-mor-Z-for-VI-int-rk}, we obtain}
\begin{align*}
    r & = s\cdot d_{R_{I}(V_I)}(R_{I}(V_I)) + \rank \bmat{L(\bfg'_{1}) & \mathbf{0}\\
L(\bfg'_3) & L(\bfg'_2)
} - \rank \bmat{L(\bfg'_{1}) & \mathbf{0}\\
\mathbf{0} & L(\bfg'_2)
}\nonumber \\
& = s + \rank \bmat{L(\bfg'_{1}) & \mathbf{0}\\
L(\bfg'_3) & L(\bfg'_2)
} - \rank \bmat{L(\bfg'_{1}) & \mathbf{0}\\
\mathbf{0} & L(\bfg'_2)
}\geq s.
\end{align*}
Hence we have $r = s$, and the proof is completed. 
\end{proof}

{As a consequence of \cref{thm:ess-cov-int-rk-inv}, we obtain the following sufficient condition for two compression systems to induce the same invariants.}

\begin{cor}
\label{cor:sufficient condition when two compression systems induce the same interval rank invariants}
    Let $\xi = \left(\xi_I \colon I^{\xi} \to \bfP\right)_{I\in \bbI}$ and {$\ka = \left(\ka_I \colon I^{\ka} \to \bfP\right)_{I\in \bbI}$} be two compression systems for $A$ $(\coloneqq \k[\bfP])$. If for every interval $I$ of $\bfP$, {$\ka_I$} essentially covers $I$ relative to $\xi$ or $\xi_I$ essentially covers $I$ relative to {$\ka$}, then for each $M\in \mod A$,
\[
{\rank^{\xi}_{\bbI} M = \rank^{\ka}_{\bbI} M}
\]
holds.
In particular, if for every interval $I$ of $\bfP$, $\xi_I$ essentially covers $I$ relative to $\tot$, {then}
\[
\rank^{\xi}_{\bbI} M = \rank^{\tot}_{\bbI} M
\]
holds.
\end{cor}

\begin{proof}
{If $\ka_I$ essentially covers $I$ relative to $\xi$, then by \cref{thm:ess-cov-int-rk-inv} we have
$\rank_I^\xi M=\bar d_{R_I^\ka(M)}(R_I^\ka(V_I))=\rank_I^\ka M$,
since $R_I^\ka(V_I)$ is indecomposable by \cref{lem:com-mult-full-int}.
The other case follows by interchanging $\xi$ and $\ka$.}
\end{proof}

\subsection{{Choosing essential covers}}

{The next two examples illustrate how to find essential covers from $I$-rank matrices. The second also shows why removing redundant columns can be useful.}

\begin{exm}
\label{exm:int-rk_D_4_cases}
	Consider a poset $\bfP$ (of Dynkin type ${\bbD_4}$) having the following Hasse quiver:
\[
\begin{tikzcd}[ampersand replacement=\&]
	\& 4 \\
	1 \& 2 \& 3
	\arrow[from=1-2, to=2-2]
	\arrow[from=2-1, to=2-2]
	\arrow[from=2-2, to=2-3]
\end{tikzcd}.
\]
Let $M_1$ and $M_2$ be two persistence modules over $\bfP$ given by
\[M_1\coloneqq
\begin{tikzcd}[ampersand replacement=\&]
	\& \k \\
	\k \& {\k^2} \& \k
	\arrow["\sbmat{1 \\ 1}", from=1-2, to=2-2]
	\arrow["\sbmat{1 \\ 1}"', from=2-1, to=2-2]
	\arrow["\sbmat{1,\, 0}"', from=2-2, to=2-3]
\end{tikzcd}
\quad \text{and}\quad
M_2\coloneqq
\begin{tikzcd}[ampersand replacement=\&]
	\& \k \\
	\k \& {\k^2} \& \k
	\arrow["\sbmat{1 \\ 0}", from=1-2, to=2-2]
	\arrow["\sbmat{0 \\ 1}"', from=2-1, to=2-2]
	\arrow["\sbmat{1,\, 1}"', from=2-2, to=2-3]
\end{tikzcd}.
\]
We compute the $\bfP$-rank of both modules under $\xi = \tot$. In this case, $I = \bfP = I^\tot$.

By~\cref{thm:restate-int-rk-inv-formula}, there exists a morphism $\bfg = \bmat{\bfg_1 & \mathbf{0}\\ \bfg_3 & \bfg_2}$ in $\Ds\k[\bfP]$ such that~\eqref{eq:comp-multiplicity-formula-ess} holds. Here $\bfg$ has the form:
\[
\bfg\coloneqq\left[
\begin{array}{c|c}
\bfg_1 & \bfzero\\
\hline
\bfg_3 & \bfg_2
\end{array}
\right] = 
\left[
\begin{array}{cc}
p_{2, 1} & -p_{2, 4}\\
\hline
p_{3, 1} & \bfzero
\end{array}
\right].
\]
Hence, it is now clear that if we take the (not full) subposet $\bfZ$ of $\bfP$ given by
\[\bfZ\coloneqq
\begin{tikzcd}[ampersand replacement=\&]
	\& 4 \\
	1 \& 2 \& 3
	\arrow[from=1-2, to=2-2]
	\arrow[from=2-1, to=2-2]
	\arrow[curve={height=20pt}, from=2-1, to=2-3]
\end{tikzcd},
\]
then the inclusion map $\io\colon \bfZ\hookrightarrow \bfP$ essentially covers $\bfP$ relative to $\tot$. By~\cref{thm:ess-cov-int-rk-inv} it suffices to compute $\bar{d}_{R_{\io}(M_{j})}(R_{\io}(V_I)) = d_{R_{\io}(M_{j})}(V_{\bfZ})$ for $j\in \{1, 2\}$. Now, because
\[
R_{\io}(M_1)=
\begin{tikzcd}[ampersand replacement=\&]
	\& \k \\
	\k \& \k^2 \& \k
	\arrow["\sbmat{1 \\ 1}", from=1-2, to=2-2]
	\arrow["\sbmat{1 \\ 1}", from=2-1, to=2-2]
	\arrow["1"', curve={height=20pt}, from=2-1, to=2-3]
\end{tikzcd}
\iso 
\begin{tikzcd}[ampersand replacement=\&]
	\& \k \\
	\k \& \k \& \k
	\arrow["1", from=1-2, to=2-2]
	\arrow["1", from=2-1, to=2-2]
	\arrow["1"', curve={height=20pt}, from=2-1, to=2-3]
\end{tikzcd}
\ds
\begin{tikzcd}[ampersand replacement=\&]
	\& 0 \\
	0 \& \k \& 0
	\arrow[from=1-2, to=2-2]
	\arrow[from=2-1, to=2-2]
	\arrow[curve={height=20pt}, from=2-1, to=2-3]
\end{tikzcd}
\]
and
\[
R_{\io}(M_2)=
\begin{tikzcd}[ampersand replacement=\&]
	\& \k \\
	\k \& \k^2 \& \k
	\arrow["\sbmat{1 \\ 0}", from=1-2, to=2-2]
	\arrow["\sbmat{0 \\ 1}", from=2-1, to=2-2]
	\arrow["1"', curve={height=20pt}, from=2-1, to=2-3]
\end{tikzcd}
\iso 
\begin{tikzcd}[ampersand replacement=\&]
	\& 0 \\
	\k \& \k \& \k
	\arrow[from=1-2, to=2-2]
	\arrow["1", from=2-1, to=2-2]
	\arrow["1"', curve={height=20pt}, from=2-1, to=2-3]
\end{tikzcd}
\ds
\begin{tikzcd}[ampersand replacement=\&]
	\& \k \\
	0 \& \k \& 0
	\arrow["1", from=1-2, to=2-2]
	\arrow[from=2-1, to=2-2]
	\arrow[curve={height=20pt}, from=2-1, to=2-3]
\end{tikzcd},
\]
we conclude that $\rank^{\tot}_{\bfP}M_1 = 1$, but $\rank^{\tot}_{\bfP}M_2 = 0$.
\end{exm}

\begin{exm}
\label{exm:int-rk_D_4_cases_2nd}
	Consider another poset $\bfP$ (of Dynkin type ${\bbD_4}$) having the following Hasse quiver:
\[
\begin{tikzcd}[ampersand replacement=\&]
	\& 4 \\
	1 \& 2 \& 3
	\arrow[from=2-2, to=1-2]
	\arrow[from=2-2, to=2-1]
	\arrow[from=2-2, to=2-3]
\end{tikzcd}.
\]
Let $M$ be a persistence module over $\bfP$ given by
\[M\coloneqq
\begin{tikzcd}[ampersand replacement=\&]
	\& \k \\
	\k \& {\k^3} \& {\k^2}
	\arrow["\sbmat{1,\, 0,\, 0}", from=2-2, to=1-2]
	\arrow["\sbmat{1,\, 0,\, 0}", from=2-2, to=2-1]
	\arrow["\sbmat{1 & 0 & 0\\0 & 1 & 1}"', from=2-2, to=2-3]
\end{tikzcd}.
\]
We compute the $\bfP$-rank of $M$ under $\xi = \tot$. Again in this case, $I = \bfP = I^\tot$.

By~\cref{thm:restate-int-rk-inv-formula}, there exists an {$I$-rank matrix} $\bfg$ under $\tot$ in $\Ds\k[\bfP]$.
Here we first take $\bfg$ to be:
\begin{equation}
\label{eq:original-morphism}
\bfg\coloneqq\left[
\begin{array}{c|c}
\bfg_3 & \bfg_2
\end{array}
\right] = 
\left[
\begin{array}{c|ccc}
\bfzero & p_{1, 2} & p_{1, 2} & \bfzero\\
p_{3, 2} & -p_{3, 2} & \bfzero & p_{3, 2}\\
\bfzero & \bfzero & -p_{4, 2} & -p_{4, 2}
\end{array}
\right].
\end{equation}
Notice that the last column of $\bfg_2$ is {the second column minus the first, so} we may take another morphism $\tilde{\bfg}$ in $\Ds\k[\bfP]$ given by
\[
\tilde{\bfg}\coloneqq \left[
\begin{array}{c|c}
\bfg_3 & \tilde{\bfg}_2
\end{array}
\right] = 
\left[
\begin{array}{c|cc}
\bfzero & p_{1, 2} & p_{1, 2}\\
p_{3, 2} & -p_{3, 2} & \bfzero\\
\bfzero & \bfzero & -p_{4, 2}
\end{array}
\right].
\]
{For every $N\in\mod A$, the same column relation holds after applying $N$, so $\rank N(\bfg)-\rank N(\bfg_2)=\rank N(\tilde{\bfg})-\rank N(\tilde{\bfg}_2)$.}
This shows that the new morphism $\tilde{\bfg}$ 
is also an {$I$-rank matrix} under $\tot$.

Now, let us take the following zigzag poset
\[\bfZ\coloneqq
\begin{tikzcd}[ampersand replacement=\&]
	2 \&\& {2'} \&\& {2''} \\
	\& 3 \&\& 1 \&\& 4
	\arrow[from=1-1, to=2-2]
	\arrow[from=1-3, to=2-2]
	\arrow[from=1-3, to=2-4]
	\arrow[from=1-5, to=2-4]
	\arrow[from=1-5, to=2-6]
\end{tikzcd}
\]
and define the order-preserving map ${\om}\colon \bfZ\to \bfP$ by
\[
{\om}(x)\coloneqq
\begin{cases}
    2, & \textnormal{if } x\in \{2, 2', 2''\},\\
    x, & \textnormal{if } x\in \{1, 3, 4\}.
\end{cases}
\]
Then ${\om}$ essentially covers $\bfP$ relative to $\tot$. Indeed, we have the following equality:
\[
\k[{\om}]\left( \left[
\begin{array}{c|cc}
\bfzero & p_{1, 2'} & p_{1, 2''}\\
p_{3, 2} & -p_{3, 2'} & \bfzero\\
\bfzero & \bfzero & -p_{4, 2''}
\end{array}
\right] \right) = \left[
\begin{array}{c|cc}
\bfzero & p_{1, 2} & p_{1, 2}\\
p_{3, 2} & -p_{3, 2} & \bfzero\\
\bfzero & \bfzero & -p_{4, 2}
\end{array}
\right]{.}
\]
Hence by~\cref{thm:ess-cov-int-rk-inv} it suffices to compute $\bar{d}_{R_{{\om}}(M)}(R_{{\om}}(V_I)) = {d_{R_{\om}(M)}(V_{\bfZ})}$. Now, because
\begin{align*}
	R_{{\om}}(M) &=
\begin{tikzcd}[ampersand replacement=\&]
	{\k^3} \&\& {\k^3} \&\& {\k^3} \\
	\& {\k^2} \&\& \k \&\& \k
	\arrow["\sbmat{1 & 0 & 0\\0 & 1 & 1}"', from=1-1, to=2-2]
	\arrow["\sbmat{1 & 0 & 0\\0 & 1 & 1}"', from=1-3, to=2-2]
	\arrow["\sbmat{1,\, 0,\, 0}"', from=1-3, to=2-4]
	\arrow["\sbmat{1,\, 0,\, 0}", from=1-5, to=2-4]
	\arrow["\sbmat{1,\, 0,\, 0}", from=1-5, to=2-6]
\end{tikzcd}\\
	& \cong \sbmat{1 & \space & 1 & \space & 1 & \space \\ \space & 1 & \space & 1 & \space & 1} \ds \sbmat{1 & \space & 1 & \space & 0 & \space \\ \space & 1 & \space & 0 & \space & 0} \ds \sbmat{1 & \space & 0 & \space & 0 & \space \\ \space & 0 & \space & 0 & \space & 0} \ds \sbmat{0 & \space & 1 & \space & 0 & \space \\ \space & 0 & \space & 0 & \space & 0}\ds \sbmat{0 & \space & 0 & \space & 1 & \space \\ \space & 0 & \space & 0 & \space & 0}^{2},
\end{align*}
we conclude that $\rank^{\tot}_{\bfP}M = 1$.
\end{exm}

We highlight that in the example above, 
finding a new {$I$-rank matrix} $\tilde{\bfg}$ under $\tot$
is crucial for finding the zigzag poset $\bfZ$. Indeed, we first notice that ${\om}$ does not cover the original choice of $\bfg$ given in~\eqref{eq:original-morphism}. Next, it is straightforward to verify that the following order-preserving map ${\om'}\colon \bfZ'\to \bfP$ covers both $\bfg$ and $\tilde{\bfg}$:
\[
\bfZ'\coloneqq \begin{tikzcd}[ampersand replacement=\&,sep=small]
	\&\& {2'} \&\& {2''} \\
	2 \& 3 \&\& 1 \&\& 4 \\
	\\
	\&\&\& {2'''}
	\arrow[from=1-3, to=2-2]
	\arrow[from=1-3, to=2-4]
	\arrow[from=1-5, to=2-4]
	\arrow[from=1-5, to=2-6]
	\arrow[from=2-1, to=2-2]
	\arrow[from=4-4, to=2-2]
	\arrow[from=4-4, to=2-6]
\end{tikzcd},
\ \text{and }
{\om'}(x)\coloneqq
\begin{cases}
    2, & \textnormal{if } x\in \{2, 2', 2'', 2'''\},\\
    x, & \textnormal{if } x\in \{1, 3, 4\}.
\end{cases}
\]
However, $\bfZ'$ is {not a zigzag poset}.

\subsection{{An application to total compression on 2D-grids}}

{\cref{cor:2Dcase-tot-formula} identifies the morphisms needed to compute total interval ranks on a 2D-grid. Using these morphisms gives the following zigzag compression system. We use the notation for adjacent joins and meets in that corollary.}

\begin{exm}
\label{exm:zz}
{Let $\bfP$ be a 2D-grid. For each $I\in\bbI$, number $\src(I)=\{a_1,\ldots,a_n\}$ and $\snk(I)=\{b_1,\ldots,b_m\}$ as in \cref{cor:2Dcase-tot-formula}.
Define $I^{\zz}$ to have the formal elements $a_i$ $(i\in[n])$, $a_{i,i+1}$ $(i\in[n-1])$, $b_j$ $(j\in[m])$, and $b_{j,j+1}$ $(j\in[m-1])$, all taken to be distinct even when the corresponding points of $I$ coincide.
The order relation {is generated by}}
\[
{\begin{gathered}
a_i\leq a_{i,i+1},\quad a_{i+1}\leq a_{i,i+1}\quad(i\in[n-1]),\\
b_{j,j+1}\leq b_j,\quad b_{j,j+1}\leq b_{j+1}\quad(j\in[m-1]),
\quad a_1\leq b_1.
\end{gathered}}
\]

{The Hasse quiver of $I^{\zz}$ is of Dynkin type $\bbA_{2(n+m-1)}$ and has the following form:}
\[\begin{tikzcd}[ampersand replacement=\&, column sep=17pt]
	\&\&\& {b_1} \\
	{a_1} \& {a_{12}} \&\& {b_{12}} \& {b_2} \\
	\& {a_2} \& \rotldots \&\& \rotldots \& \space \\
	\&\& \space \& \rotldots \&\& \rotldots \& {b_{m-1}} \\
	\&\&\& {a_{n-1}} \& {a_{n-1,n}} \&\& {b_{m-1,m}} \& {b_m} \\
	\&\&\&\& {a_n}
	\arrow[curve={height=-25pt}, from=2-1, to=1-4]
	\arrow[from=2-1, to=2-2]
	\arrow[from=2-4, to=1-4]
	\arrow[from=2-4, to=2-5]
	\arrow[from=3-2, to=2-2]
	\arrow[from=3-2, to=3-3]
	\arrow[from=3-5, to=2-5]
	\arrow[from=4-6, to=4-7]
	\arrow[from=5-4, to=4-4]
	\arrow[from=5-4, to=5-5]
	\arrow[from=5-7, to=4-7]
	\arrow[from=5-7, to=5-8]
	\arrow[from=6-5, to=5-5]
\end{tikzcd}.
\]

{Let $\om_I\colon I^{\zz}\to I$ send each formal element to the corresponding point of $I$, and set $\zz_I\coloneqq\tot_I\circ\om_I\colon I^{\zz}\to\bfP$.
By \cref{cor:2Dcase-tot-formula}, the map $\om_I$ is order-preserving.
Since $I^{\zz}$ is connected and $\src(I)\cup\snk(I)\subseteq\Im\zz_I\subseteq I$, the family $\zz\coloneqq(\zz_I)_{I\in\bbI}$ is a compression system.
The map $\om_I$ need not be injective: for the interval in \cref{exm:src1-illustration}, the formal elements $a_{12}$ and $b_{12}$ both map to $(2,2)$.
When $n=m=1$, we take two distinct formal elements $a_1<b_1$, even when $I$ consists of a single point.}
\end{exm}

{Applying \cref{thm:ess-cov-int-rk-inv} to the zigzag covers above shows that total interval ranks on a 2D-grid can be computed from zigzag restrictions.}

\begin{cor}
\label{cor:tot-zz-case}
	{Let $\bfP$ be a 2D-grid, let} $\zz = \left(\zz_{I}\colon I^{\zz}{\to} \bfP\right)_{I\in \bbI}$ be the compression system defined above, {and let $\tot$ be the total compression system}. Then interval rank invariants under $\zz$ and $\tot$ coincide, i.e.,
	\[
	\rank^{\tot}_{\bbI} = \rank^{\zz}_{\bbI}.
	\]
\end{cor}

\begin{proof}
	We show that for every $I\in \bbI$, $\zz_{I}$ essentially covers $I$ relative to $\tot$. By~\cref{cor:2Dcase-tot-formula} and \cref{thm:restate-int-rk-inv-formula}, there exists a morphism $\bfg = \bmat{\bfg_1 & \mathbf{0}\\ \bfg_3 & \bfg_2}$ in $\Ds\k[I]$ such that~\eqref{eq:comp-multiplicity-formula-ess} holds. Here $\bfg_1$ has the form:
\[
\scalebox{0.95}{$
\begin{blockarray}{ccccccc}
& \scalebox{0.7}{$1$} & \scalebox{0.7}{$2$} & \scalebox{0.7}{$3$} & \scalebox{0.7}{$\cdots$} & \scalebox{0.7}{$n-1$} & \scalebox{0.7}{$n$} \\
\begin{block}{c[cccccc]}
  \scalebox{0.7}{$12$} & p_{a_{12},a_{1}} & -p_{a_{12},a_{2}} & \mathbf{0} & \cdots & \mathbf{0} & \mathbf{0} \\
  \scalebox{0.7}{$23$} & \mathbf{0} & p_{a_{23},a_{2}} & -p_{a_{23},a_{3}} & \cdots & \mathbf{0} & \mathbf{0} \\
  \scalebox{0.7}{$\vdots$} & \vdots & \vdots & \vdots &  & \vdots & \vdots \\
  \scalebox{0.7}{$n-2,n-1$} & \mathbf{0} & \mathbf{0} & \mathbf{0} & \cdots & -p_{a_{n-2,n-1},a_{n-1}} & \mathbf{0} \\
  \scalebox{0.7}{$n-1,n$} & \mathbf{0} & \mathbf{0} & \mathbf{0} & \cdots & p_{a_{n-1,n},a_{n-1}} & -p_{a_{n-1,n},a_{n}} \\
\end{block}
\end{blockarray}
$},
\]
$\bfg_3$ has the form:
\[
\scalebox{0.95}{$
\begin{blockarray}{ccccccc}
& \scalebox{0.7}{$1$} & \scalebox{0.7}{$2$} & \scalebox{0.7}{$3$} & \scalebox{0.7}{$\cdots$} & \scalebox{0.7}{$n-1$} & \scalebox{0.7}{$n$} \\
\begin{block}{c[cccccc]}
  \scalebox{0.7}{$1$} & p_{b_{1},a_{1}} & \mathbf{0} & \mathbf{0} & \cdots & \mathbf{0} & \mathbf{0} \\
  \scalebox{0.7}{$2$} & \mathbf{0} & \mathbf{0} & \mathbf{0} & \cdots & \mathbf{0} & \mathbf{0} \\
  \scalebox{0.7}{$\vdots$} & \vdots & \vdots & \vdots &  & \vdots & \vdots \\
  \scalebox{0.7}{$m-1$} & \mathbf{0} & \mathbf{0} & \mathbf{0} & \cdots & \mathbf{0} & \mathbf{0} \\
  \scalebox{0.7}{$m$} & \mathbf{0} & \mathbf{0} & \mathbf{0} & \cdots & \mathbf{0} & \mathbf{0} \\
\end{block}
\end{blockarray}
$},
\]
and $\bfg_2$ has the form:
\[
\scalebox{0.95}{$
\begin{blockarray}{ccccc}
& \scalebox{0.7}{$12$} & \scalebox{0.7}{$23$} & \scalebox{0.7}{$\cdots$} & \scalebox{0.7}{$m-1, m$} \\
\begin{block}{c[cccc]}
  \scalebox{0.7}{$1$} & p_{b_{1},b_{12}} & \mathbf{0} & \cdots & \mathbf{0} \\
  \scalebox{0.7}{$2$} & -p_{b_{2},b_{12}} & p_{b_{2},b_{23}} & \cdots & \mathbf{0} \\
  \scalebox{0.7}{$3$} & \mathbf{0} & -p_{b_{3},b_{23}} & \cdots & \mathbf{0} \\
  \scalebox{0.7}{$\vdots$} & \vdots & \vdots & & \vdots \\
  \scalebox{0.7}{$m-1$} & \mathbf{0} & \mathbf{0} & \cdots & p_{b_{m-1},b_{m-1,m}} \\
  \scalebox{0.7}{$m$} & \mathbf{0} & \mathbf{0} & \cdots & -p_{b_{m},b_{m-1,m}}\\
\end{block}
\end{blockarray}
$}.
\]
From the definition of $\zz_{I}$, we naturally have the following commutative diagram:
\[
\begin{tikzcd}[ampersand replacement=\&]
	I^{\zz} \&\& \bfP \\
	\& I^\tot = I
	\arrow["\zz_{I}", from=1-1, to=1-3]
	\arrow["{{\om_I}}"', from=1-1, to=2-2]
	\arrow["{\tot_I}"', from=2-2, to=1-3]
\end{tikzcd}
\]
{where $\om_I$ is the map defined in \cref{exm:zz}.
Let $\bfg'=\bmat{\bfg'_1 & \mathbf{0}\\ \bfg'_3 & \bfg'_2}$ be the morphism in $\Ds\k[I^{\zz}]$ obtained from $\bfg$ by replacing each row and column label with the corresponding formal element of $I^{\zz}$ and each nonzero entry with the canonical morphism between these formal elements, with the same coefficient.
These morphisms exist by the defining order relations of $I^{\zz}$.
Then $\om_I(\bfg')=\bfg$, so $\om_I$ covers $\bfg$.} Thus for every $M\in \mod A$,
\[
\rank_I^\tot M = \bar{d}_{R_{\zz_{I}}(M)}(R_{\zz_{I}}(V_I))
\]
by~\cref{thm:ess-cov-int-rk-inv}. Notice that $R_{\zz_{I}}(V_I) = {V_{I^{\zz}}}$ is an interval module in $\mod \k[I^{\zz}]$, hence an indecomposable module by~\cref{lem:com-mult-full-int}. It follows that
\[
\bar{d}_{R_{\zz_{I}}(M)}(R_{\zz_{I}}(V_I)) = d_{R_{\zz_{I}}(M)}(R_{\zz_{I}}(V_I)) = \rank^{\zz}_{I} M.
\]
Therefore, the assertion follows.
\end{proof}

\begin{rmk}
{By \cref{rmk:concide-gri-int-rk}, the total interval rank invariant is the generalized rank invariant.
Thus \cref{cor:tot-zz-case} recovers the zigzag formula in \cite[Theorem~3.12]{deyComputingGeneralizedRank2024} for 2D-grids.
As in the boundary-cap construction in \cite[Definition~3.7 and equations~(8)--(10)]{deyComputingGeneralizedRank2024}, the same point of $I$ may occur more than once along this zigzag.
Here zigzag persistence computes the invariant through a particular choice of essential cover.}
\end{rmk}

\section{Examples}\label{sec6}
{We first examine the incompleteness of interval rank invariants and then use interval replacements to distinguish filtrations (\cref{exm:app-of-int-rep}). Here, incompleteness means that non-isomorphic modules $M,N\in\mod A$ may satisfy $\rank^\xi_\bbI M=\rank^\xi_\bbI N$ under a compression system $\xi$, or equivalently $\de^\xi_M(I)=\de^\xi_N(I)$ for all $I\in\bbI$. We begin with the following known example of Kim and M\'emoli.}

\begin{exm}[Example 3.21 in \cite{kim2021generalized}]
\label{exm:KM}
{Let $\bfP$ be the poset of Dynkin type $\bbD_4$ with Hasse quiver}
\[
\begin{tikzcd}
	& a & \\
	& b \\
	c && d
	\arrow[from=1-2, to=2-2]
	\arrow[from=3-1, to=2-2]
	\arrow[from=3-3, to=2-2]
\end{tikzcd}
\]
and set $A\coloneqq \k[\bfP]$.
Define $F, G, H \in \mod A$ by
\[
\begin{tikzcd}[ampersand replacement=\&]
	\& \k \& \\
	\& {\k^2} \\
	\k \&\& \k
	\arrow[from=1-2, to=2-2, "\sbmat{1\\0}"]
	\arrow[from=3-1, to=2-2, "\sbmat{0\\1}"]
	\arrow[from=3-3, to=2-2, "\sbmat{1\\1}"']
\end{tikzcd},
\qquad
\begin{tikzcd}[ampersand replacement=\&]
	\& \k \& \\
	\& {\k^3} \\
	\k \&\& \k
	\arrow[from=1-2, to=2-2, "\sbmat{1\\0\\0}"]
	\arrow[from=3-1, to=2-2, "\sbmat{0\\1\\0}"]
	\arrow[from=3-3, to=2-2, "\sbmat{1\\1\\0}"']
\end{tikzcd},
\quad\text{and}\quad
\begin{tikzcd}[ampersand replacement=\&]
	\& \k \& \\
	\& {\k^3} \\
	\k \&\& \k
	\arrow[from=1-2, to=2-2, "\sbmat{1\\0\\0}"]
	\arrow[from=3-1, to=2-2, "\sbmat{0\\1\\0}"]
	\arrow[from=3-3, to=2-2, "\sbmat{0\\0\\1}"']
\end{tikzcd},
\]
respectively.
{Then $G \not\iso H$, but $\rank^\tot_\bbI G = \rank^\tot_\bbI H$, and both modules are decomposable:}
\[
G \iso F \ds V_{\{b\}},\quad\text{and}\quad
H \iso V_{\{a, b\}} \ds V_{\{c, b\}} \ds V_{\{d, b\}}.
\]
\end{exm}
{The next two examples show that the incompleteness in \cref{exm:KM} also occurs among indecomposable modules: they give one-parameter families of pairwise non-isomorphic indecomposable modules with the same interval rank invariant under $\tot$. \cref{exm:M-lambda} gives such a family over the 2D-grid $G_{5,2}$ with nontrivial commutativity relations, and the interval rank invariant under $\ss$ is also constant on this family.}

\begin{exm}
(1) Define a poset $\bfP_1$ and persistence modules $M(\theta)$ over $\bfP_1$ by
\[
\bfP_1\coloneqq \begin{tikzcd}
1 & 2 \\
3 & 4
	\arrow[from=1-1, to=1-2]
	\arrow[from=2-2, to=1-2]
	\arrow[from=2-2, to=2-1]
	\arrow[from=1-1, to=2-1]
\end{tikzcd}
, \quad
M(\theta)\coloneqq \begin{tikzcd}
\bbR & \bbR \\
\bbR & \bbR
	\arrow["1", from=1-1, to=1-2]
	\arrow["\theta"', from=2-2, to=1-2]
	\arrow["1", from=2-2, to=2-1]
	\arrow["1"', from=1-1, to=2-1]
\end{tikzcd}
\]
for $\theta \in \bbR \setminus \{0,1\}$. We take $\xi\coloneqq \tot$ (see \cref{exm:total}). Let $\theta_1, \theta_2 \in \bbR \setminus \{0,1\}$ such that $\theta_1 \neq \theta_2$. Then $M(\theta_1)$ and $M(\theta_2)$ are indecomposable and clearly not isomorphic to each other but they have the same interval rank invariant under $\tot$. One can compute the interval replacement of $M(\theta)$ for $\theta=\theta_1, \theta_2$ by using \cref{rmk:altdefdelta}:
\begin{table}[ht]
    \caption{Computation of $\delta_{M(\theta)}^\xi(I)$ for $\theta \in \bbR \setminus \{0,1\}$.}
    \label{tab:table1}
    \begin{tabular}{c|c|c}
      \textbf{Interval} & \textbf{$I$-rank} & \textbf{Signed interval multiplicity}\\
        $I$ & $\rank^\xi_I M(\theta)$ & $\delta_{M(\theta)}^\xi(I)$ \\
      \hline
      $\{1,2,3,4\}$ & 0 & 0\\
      $\{1,2,3 \}$ & 1 & 1\\
      $\{1,2,4 \}$ & 1 & 1\\
      $\{1,3,4 \}$ & 1 & 1\\
      $\{2,3,4 \}$ & 1 & 1\\
      $\{1,2 \}$ & 1 & -1\\
      $\{1,3 \}$ & 1 & -1\\
      $\{2,4 \}$ & 1 & -1\\
      $\{3,4 \}$ & 1 & -1\\
      $\{1 \}$ & 1 & 0\\
      $\{2 \}$ & 1 & 0\\
      $\{3 \}$ & 1 & 0\\
      $\{4 \}$ & 1 & 0\\
    \end{tabular}
\end{table}

(2) Define a poset $\bfP_2$ and persistence modules $N(\theta)$ over $\bfP_2$ by
\[
\bfP_2\coloneqq
\begin{tikzcd}[column sep=18pt]
& 7 &&& 8 \\
3 &&& 4\\
& 5 &&& 6 \\
1 &&& 2
	\arrow[from=4-1, to=4-4]
	\arrow[from=3-2, to=3-5]
	\arrow[from=4-1, to=3-2]
	\arrow[from=4-4, to=3-5]
	\arrow[from=4-1, to=2-1]
	\arrow[from=3-2, to=1-2]
	\arrow[from=3-5, to=1-5]
	\arrow[from=4-4, to=2-4]
	\arrow[from=2-1, to=1-2]
	\arrow[from=1-2, to=1-5]
	\arrow[from=2-1, to=2-4]
	\arrow[from=2-4, to=1-5]
\end{tikzcd}
, \quad
N(\theta)\coloneqq
\begin{tikzcd}[column sep=18pt]
& \bbR &&& 0 \\
 \bbR &&& \bbR && {} \\
& \bbR &&& \bbR \\
 0 &&& \bbR
	\arrow[from=4-1, to=4-4]
	\arrow["1"{pos=0.4}, from=3-2, to=3-5]
	\arrow[from=4-1, to=3-2]
	\arrow["\theta"', from=4-4, to=3-5]
	\arrow[from=4-1, to=2-1]
	\arrow["1"{pos=0.3}, from=3-2, to=1-2]
	\arrow[from=3-5, to=1-5]
	\arrow["1"'{pos=0.3}, from=4-4, to=2-4]
	\arrow["1", from=2-1, to=1-2]
	\arrow[from=1-2, to=1-5]
	\arrow["1", from=2-1, to=2-4]
	\arrow[from=2-4, to=1-5]
\end{tikzcd}
\]
for $\theta \in \bbR \setminus \{0,1\}$.
Then for any distinct $\th_1, \th_2 \in \bbR \setminus \{0,1\}$,
$N(\th_1)$ and $N(\th_2)$ are indecomposable and not isomorphic,
but have the same interval rank invariant under $\tot$.

(3) Define a poset $\bfP_3$ and persistence modules $L(\theta)$ over $\bfP_3$ by
\[
\bfP_3\coloneqq
\begin{tikzcd}
1 & 2 & 3\\
4 & 5 & 6\\
7 & 8 & 9
	\arrow[from=2-2, to=1-2]
	\arrow[from=2-2, to=2-1]
	\arrow[from=2-2, to=3-2]
	\arrow[from=2-2, to=2-3]
	\arrow[from=1-2, to=1-1]
	\arrow[from=2-1, to=1-1]
	\arrow[from=3-2, to=3-1]
	\arrow[from=2-1, to=3-1]
	\arrow[from=1-2, to=1-3]
	\arrow[from=2-3, to=1-3]
	\arrow[from=3-2, to=3-3]
	\arrow[from=2-3, to=3-3]
\end{tikzcd}
, \quad
L(\theta)\coloneqq
\begin{tikzcd}
\bbR & \bbR & \bbR &&& {} \\
\bbR & 0 & \bbR \\
\bbR & \bbR & \bbR
	\arrow[from=2-2, to=1-2]
	\arrow[from=2-2, to=2-1]
	\arrow[from=2-2, to=3-2]
	\arrow[from=2-2, to=2-3]
	\arrow["1"', from=1-2, to=1-1]
	\arrow["1", from=2-1, to=1-1]
	\arrow["1", from=3-2, to=3-1]
	\arrow["1"', from=2-1, to=3-1]
	\arrow["1", from=1-2, to=1-3]
	\arrow["1"', from=2-3, to=1-3]
	\arrow["\theta"', from=3-2, to=3-3]
	\arrow["1", from=2-3, to=3-3]
\end{tikzcd}
\]
for $\theta \in \bbR \setminus \{0,1\}$.
Then for any distinct $\th_1, \th_2 \in \bbR \setminus \{0,1\}$,
$L(\th_1)$ and $L(\th_2)$ are indecomposable and not isomorphic,
but have the same interval rank invariant under $\tot$.
\end{exm}

Note that {in all the examples above}, the commutativity relations are trivial.

{We next give an example over a 2D-grid, where the commutativity relations are nontrivial, using non-isomorphic indecomposable modules with the same interval rank invariant. We first consider commutative ladders $\bfP=G_{n,2}$.}
{For} $n \le 4$, there exists no such example.
Indeed, {the algebra $\k[\bfP]$ is representation-finite \cite{escolarPersistenceModulesCommutative2016}, and its Auslander--Reiten quiver is directed.
Hence any two non-isomorphic indecomposable modules $M,N\in\mod\k[\bfP]$ satisfy $\udim M\ne\udim N$.}
Thus $\dim M(x) \ne \dim N(x)$ for some $x \in \bfP$, and hence
$\rank^\xi_I M \ne \rank^\xi_I N$ for $I = \{x\}$.
{We now give such an example for $n=5$, which is minimal among commutative ladders.}

\begin{exm}
\label{exm:M-lambda}
Let $\la \in \k$ and $M_\la$ be the following
representation of $\bfP\coloneqq G_{5,2}$:
\[\begin{tikzcd}
\k & {\k^2} & {\k^2} & \k & 0 \\
0 & \k & {\k^2} & {\k^2} & \k
\arrow["{\binom{1}{0}}", from=1-1, to=1-2]
\arrow["\id", from=1-2, to=1-3]
\arrow["{(\lambda,-1)}", from=1-3, to=1-4]
\arrow[from=1-4, to=1-5]
\arrow[from=2-1, to=1-1]
\arrow[from=2-1, to=2-2]
\arrow["{\binom{0}{1}}"', from=2-2, to=1-2]
\arrow["{\binom{0}{1}}"', from=2-2, to=2-3]
\arrow["\id"', from=2-3, to=2-4]
\arrow["{(1,-1)}"', from=2-4, to=2-5]
\arrow["\id"', from=2-3, to=1-3]
\arrow["{(\lambda,-1)}"', from=2-4, to=1-4]
\arrow[from=2-5, to=1-5]
\end{tikzcd}.
\]
Then it is easy to see that the endomorphism algebra of $M_\la$ is
isomorphic to $\k$, and hence $M_\la$ is indecomposable, and
if $\la \ne \mu$ in $\k$, then ${\Hom_{\k[\bfP]}}(M_\la, M_\mu) = 0$.
Thus $M_\la \iso M_\mu$ if and only if $\la = \mu$.
Let $\la \ne \mu$ in $\k \setminus \{0, 1\}$.
{In view of \cref{thm:general w/o conditions} and \cref{rmk:exist-cond-case}, we verify {using MPIntInv} that for $\xi \in \{\tot, \ss\}$, $\rank^\xi_{\bbI} M_\la = \rank^\xi_{\bbI} M_\mu$.}

The dimension vector of $M_\la$ is taken from \cite[A2.\ The frames of the tame concealed algebras]{ringel2006tame}
for $\tilde{\bbE}_7$, and the representation $M_\la$
is constructed by modifying a homogeneous representation of
$\tilde{\bbD}_4$ in
\cite[Chapter 6 Tables]{dlab1976indecomposable}.
\end{exm}

{The next corollary gives a vanishing criterion for the map between the endpoints of a segment.}

\begin{cor}
\label{cor:non-interval indecomposable has zero rank invariant}
Let $M \in \mod A$ and $I = [x,y] \in \Seg(\bfP)$.
If {$R^{\tot}_{I}(M) \in \mod \k[I]$} is indecomposable
and not isomorphic to $V_I$, then $M(p_{y,x}) = 0$.
\end{cor}

\begin{proof}
Here we take the total compression system $\tot$ in \cref{exm:total}. Then by assumption, we have ${\rank_{I}^{\tot} M} = 0$. Hence by {\cref{prp:seg-rank}}, we obtain $\rank M(p_{y,x}) = \rank_{I}^\tot M = 0$, which shows the assertion.
\end{proof}

\removelastskip
\begin{rmk}
\label{rmk:vanishing-history}
\cref{lem:gen-rk-inv} has the same statement as
\cite[Lemma 3.1]{chambersPersistentHomologyDirected2018}, and the above corollary follows immediately from this lemma.
However, we found a gap in the proof of \cite[Lemma 3.1]{chambersPersistentHomologyDirected2018} and therefore could not simply invoke it without first establishing its validity. {The gap and its resolution are explained in \cref{rmk:gap-chambers}.}

This also motivated us to examine the assertion of the above corollary in nontrivial examples.
In the cases known to us at that time (namely the cases of the commutative ladders ${\rm CL}(3)$ and ${\rm CL}(4)$), either $M(x)=0$ or $M(y)=0$, so that
$M(p_{y,x})=0$ holds trivially.
We therefore looked for an example in which both $M(x)\neq 0$ and $M(y)\neq 0$.
\cref{exm:2D-grid} provides such an example, and nevertheless satisfies
$M(p_{y,x})=0$.
This example provided evidence that the statement of \cref{lem:gen-rk-inv} is indeed valid and motivated us to investigate its proof more thoroughly.
We were then able to fill the gap in the earlier proof and give a complete proof of the lemma.
\end{rmk}
\removelastskip

\begin{exm}
\label{exm:2D-grid}
{Let $\bfP = G_{5,2}$ in \eqref{eq:G52} and $I = [x,y]=\bfP$
with $x = (1,1),\, y = (5,2)$.}
{We construct an indecomposable $\k[I]$-module $M$ with dimension vector} $\sbmat{2\ 3\ 3\ 2\ 1\\1\ 2\ 3\ 3\ 2}$.
Then since $M \not\iso V_I$, we {must have} $M(p_{y,x}) = 0$
by the corollary above.

{It} is not hard to check that the following module $M$
{is indecomposable with this dimension vector:}
\[
\begin{tikzcd}[ampersand replacement=\&]
\Nname{1'}\k^2 \& \Nname{2'}\k^3 \& \Nname{3'}\k^3 \& \Nname{4'}\k^2 \& \Nname{5'}\k\\
\Nname{1}\k \& \Nname{2}\k^2 \& \Nname{3}\k^3 \& \Nname{4}\k^3 \& \Nname{5}\k^2
\Ar{1}{2}{"\sbmat{0\\1}" '}
\Ar{2}{3}{"\sbmat{1&0\\0&1\\0&-1}" '}
\Ar{3}{4}{"\id" '}
\Ar{4}{5}{"\sbmat{1&0&0\\0&0&1}" '}
\Ar{1'}{2'}{"\sbmat{1&1\\-1&0\\0&-1}"}
\Ar{2'}{3'}{"\id"}
\Ar{3'}{4'}{"\sbmat{1&0&0\\0&1&0}"}
\Ar{4'}{5'}{"\sbmat{1&0}"}
\Ar{1}{1'}{"\sbmat{-1\\1}"}
\Ar{2}{2'}{"\sbmat{1&0\\0&1\\0&-1}" '}
\Ar{3}{3'}{"\id" '}
\Ar{4}{4'}{"\sbmat{1&0&0\\0&1&0}" '}
\Ar{5}{5'}{"\sbmat{1&0}" '}
\end{tikzcd}.
\]
{This module satisfies the condition} {$M(p_{y,x})=0$} stated above.
We note that this $M$ is obtained as the Auslander--Reiten translation $\ta M_\la$ of the indecomposable module $M_\la$ with $\la = 1$ in \cref{exm:M-lambda}.
\end{exm}

{Finally, we use interval replacements to distinguish two filtrations.}

\begin{exm}
\label{exm:app-of-int-rep}
	Let $\calF_1$ and $\calF_2$ be two filtrations indexed by the 2D-grid (\cref{fig:two-filtrations}). We consider {the first homology} $H_{1}(\blank;\bbZ/2\bbZ)$ {and set} $M_{j}\coloneqq H_{1}(\blank;\bbZ/2\bbZ)\circ \calF_{j}$ ($j \in \{1, 2\}$). {Using MPIntInv, we obtain} the interval replacements of $M_{j}$ under $\tot$ and $\ss$ as follows (interval modules are written as their dimension vectors).
	\begin{align}
		\de^{\tot}(M_{1}) & = \br{\de^{\tot}(M_{1})_{+}} - \br{\de^{\tot}(M_{1})_{-}}\nonumber\\
		& = \br{\sbmat{1 & 1 & 1 & 0 & 0\\
1 & 1 & 1 & 1 & 0} \ds \sbmat{1 & 1 & 1 & 1 & 1\\
0 & 0 & 1 & 1 & 1} \ds \sbmat{0 & 1 & 1 & 1 & 1\\
0 & 1 & 1 & 1 & 1} \ds \sbmat{0 & 1 & 1 & 1 & 0\\
0 & 0 & 1 & 1 & 1}} - \br{\sbmat{0 & 1 & 1 & 1 & 1\\
0 & 0 & 1 & 1 & 1}}, \label{eq:repl_M_1_tot}\\
\de^{\ss}(M_{1}) & = \br{\de^{\ss}(M_{1})_{+}} - \br{\de^{\ss}(M_{1})_{-}}\nonumber\\
		& = \br{\sbmat{1 & 1 & 1 & 0 & 0\\
1 & 1 & 1 & 1 & 0} \ds \sbmat{1 & 1 & 1 & 1 & 0\\
0 & 0 & 1 & 1 & 1} \ds \sbmat{1 & 1 & 1 & 1 & 1\\
0 & 1 & 1 & 1 & 1} \ds \sbmat{0 & 1 & 1 & 1 & 0\\
0 & 1 & 1 & 1 & 1}} - \br{\sbmat{1 & 1 & 1 & 1 & 0\\
0 & 1 & 1 & 1 & 1}}, \label{eq:repl_M_1_ss}
	\end{align}
and for any $\xi\in \{\tot, \ss\}$,
	\begin{align}
	\label{eq:repl_M_2}
		\de^{\xi}(M_{2}) & = \br{\de^{\xi}(M_{2})_{+}} - \br{\de^{\xi}(M_{2})_{-}}\nonumber\\
		& = \br{\sbmat{1 & 1 & 1 & 0 & 0\\
1 & 1 & 1 & 1 & 0} \ds \sbmat{1 & 1 & 1 & 1 & 1\\
0 & 0 & 1 & 1 & 1} \ds \sbmat{0 & 1 & 1 & 1 & 0\\
0 & 1 & 1 & 1 & 1}}.
	\end{align}
{Equations~\eqref{eq:repl_M_1_tot}, \eqref{eq:repl_M_1_ss}, and \eqref{eq:repl_M_2} distinguish the filtrations $\calF_1$ and $\calF_2$ by their distinct interval replacements under either $\tot$ or $\ss$. Moreover, $M_1$ is not interval decomposable because $\de^{\tot}(M_1)_-\ne0$. For $M_2$, both replacements have zero negative part.}

\end{exm}

\begin{figure}[ht]
  \centering
  \subfloat[%
    Filtration $\calF_1$%
    \label{fig:sub:Filtration_ess_cov_int_mult_a}%
  ]{%
    \includegraphics[width=\textwidth]{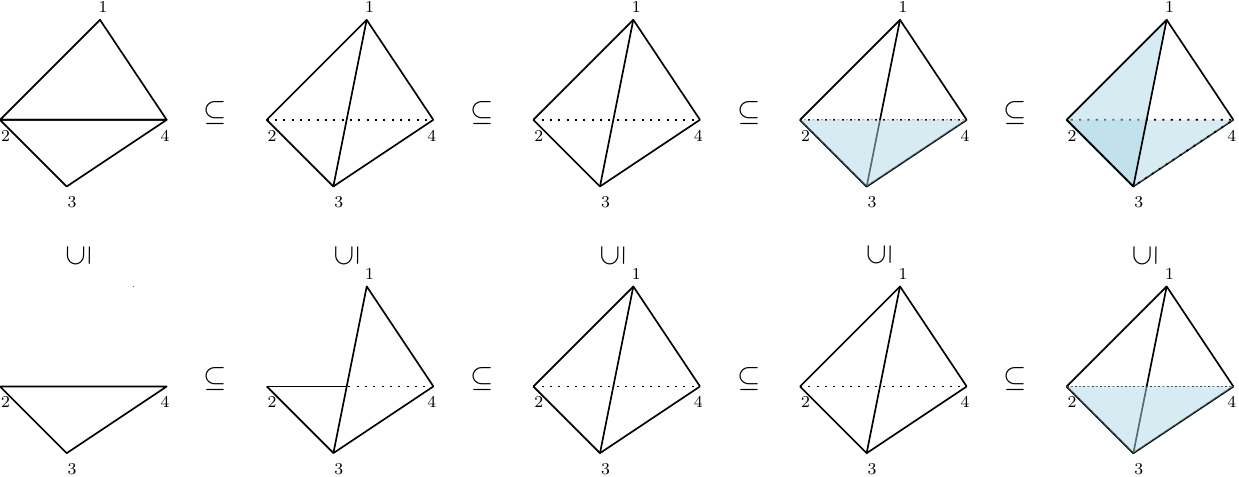}%
  }

  \vspace{1ex}

  \subfloat[%
    Filtration $\calF_2$%
    \label{fig:sub:Filtration_ess_cov_int_mult_b}%
  ]{%
    \includegraphics[width=\textwidth]{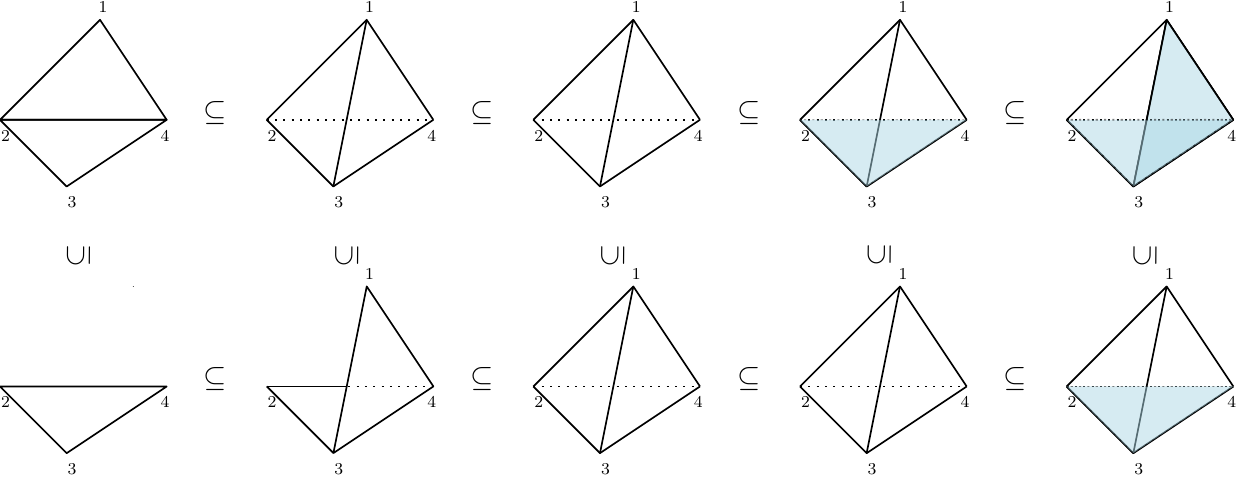}%
  }

  \caption{Two filtrations indexed by $G_{5,2}$ in \cref{exm:app-of-int-rep}}
  \label{fig:two-filtrations}
\end{figure}

\begin{appendices}

\section{Original definition of a compression system}
\label[appendix]{Original definition of a compression system}
{This appendix records the original formulation of compression systems in terms of quivers. We recall the necessary quiver notation, give the definition, and explain its relation to \cref{dfn:comp-sys-simplified-ver} in \cref{rmk:simplified-ver}. The quiver formulation also permits extensions to other bound quivers; for background, we refer the reader to \cite{assemElementsRepresentationTheory2006}.}

\begin{dfn}
(1) A \emph{quiver} is a quadruple $Q=(Q_0,Q_1,s,t)$ of
sets $Q_0, Q_1$ and maps $s,t : Q_1 \to Q_0$. If we draw each $x \in Q_0$ as a vertex, and each $\al \in Q_1$ with $s(\al) = x$ and $t(\al) = y$ as an arrow $x \ya{\ \al\ } y$, then $Q$ can be expressed as a directed graph.
For this, elements of $Q_1$ are said to be the arrows of $Q$,
and $s(\al), t(\al)$ are said to be the \emph{source} and
\emph{target} of the arrow $\al$. {A \emph{source} (resp.\ \emph{sink}) of $Q$ is a vertex with no incoming (resp.\ outgoing) arrows. We denote the corresponding sets of vertices by $\src(Q)$ and $\snk(Q)$.}

(2) Let $Q, Q'$ be quivers.
A \emph{quiver morphism} $F$ from $Q$ to $Q'$ is a pair $(F_0, F_1)$ of maps
$F_i \colon Q_i \to Q'_i$\ ($i = 0,1$) such that for any arrow $a \colon x \to y$ in $Q$, $F_1(a)$ is an arrow $F_0(x) \to F_0(y)$.
By abuse of notation, we write $F(a) = F_1(a), F(x)\coloneqq F_0(x)$
for all $a \in Q_1$, $x \in Q_0$.

(3) A \emph{path} from $x$ to $y$ of \emph{length} $n \ge 0$
is a symbol $p = (y | \alpha_n,\ldots,\alpha_1 | x)$
consisting of $\al_1, \dots, \al_n \in Q_1$ such that
$y = t(\alpha_n)$, $x = s(\alpha_1)$
and $s(\alpha_{i+1}) = t(\alpha_i)$ for $i \in [n-1]$.
If $n = 0$, we require $x = y$ and set $e_x\coloneqq (x||x)$.
The symbol $p$ is sometimes expressed by $p \colon x \leadsto y$.
{To extend the definitions} of source and target to paths, we set $\ovl{s}(p)\coloneqq x,\, \ovl{t}(p)\coloneqq y$ and call them the \emph{source} and \emph{target} of $p$.
Two paths $p$ and $q$ are said to be \emph{parallel} if
$\ovl{s}(p) = \ovl{s}(q)$ and $\ovl{t}(p) = \ovl{t}(q)$.

(4) A path $p$ {of length at least 1} is called an \emph{oriented cycle} if $\ovl{s}(p) = \ovl{t}(p)$.
$Q$ is said to be \emph{acyclic} if $Q$ has no oriented cycles.
Furthermore, we denote by $Q_n$ the set of all paths of length $n$, and the set of all paths of $Q$ is denoted by {$Q_{\geq 0}$; thus} we identify the paths of length 0 and 1 with vertices and arrows, respectively.

(5) The {\em path category} $\k[Q]$ of a finite quiver $Q$ is defined as follows:
The set of objects of $\k[Q]$ is given by $Q_0$.
For any $x, y \in Q_0$, we {define} $\k[Q](x,y)$ to be the $\k$-vector space
with basis $\{p \in Q_{\ge 0} \mid p \colon x \leadsto y\}$,
the identity morphism $\id_x$ at $x \in Q_0$ is given by $\id_x\coloneqq e_x$, and
the composition is given by concatenation of paths:
for any paths $p = (y | \alpha_m,\ldots,\alpha_1 | x)$ and
$q = (z | \be_n,\ldots,\be_1 | y)$ with $m,n \ge 0$,
$q\circ p\coloneqq (z | \be_n,\ldots,\be_1, \al_m,\ldots, \al_1 | x)$.

(6) A \emph{walk} between $x$ and $y$ in $Q$ is a sequence $(p_i)_{i=1}^{2n}$ of paths
in $Q$ of length $\ge 0$ with $n \ge 1$ having the following form:
\[
\begin{tikzcd}[column sep =20pt]
& x_1 && x_2 && \cdots && x_n\\
x && y_1 && y_2 & \cdots & y_{n-1}&&  y
\Ar{1-2}{2-1}{"p_1" ', rightsquigarrow}
\Ar{1-2}{2-3}{"p_2", rightsquigarrow}
\Ar{1-4}{2-3}{"p_3" ', rightsquigarrow}
\Ar{1-4}{2-5}{"p_4", rightsquigarrow}
\Ar{1-6}{2-5}{"p_5" ', rightsquigarrow}
\Ar{1-6}{2-7}{"p_{2n-2}", rightsquigarrow}
\Ar{1-8}{2-7}{"p_{2n-1}", rightsquigarrow}
\Ar{1-8}{2-9}{"p_{2n}", rightsquigarrow}
\end{tikzcd}.
\]

(7) A {\em commutativity relation} in $Q$ is a difference
$p-q$, where $p$ and $q$ are distinct parallel paths $x \leadsto y$.  Note that $p-q \in \k[Q](x, y)$, and that
if $p$ is an oriented cycle with $x = \ovl{s}(p)$,
then the difference $p - e_x$ is also a commutativity relation in $Q$.
\end{dfn}

\begin{dfn}
A quiver $Q$ is said to be \emph{connected} if 
$Q_0$ is nonempty, and 
for any vertices $x, y$ in $Q$, there exists a walk between $x$ and $y$ in $Q$.
\end{dfn}

\begin{dfn}%
Let $Q$ be a quiver, and $Q'$ a full subquiver of $Q$.

(1) $Q'$ is said to be \emph{convex} in $Q$
if for any $x,y \in Q'_0$, and any path $p \colon x \leadsto y$ in $Q$, all vertices of $p$ are in $Q'_0$ (and thus $p$ is a path in $Q'$).

(2) $Q'$ is called an \emph{interval} of $Q$ if 
$Q'$ is convex and connected. 
The set of all interval subquivers of $Q$ is denoted by $\bbI(Q)$ ($\bbI$ for short).

(3) Suppose that $Q$ is finite and acyclic.
A {\em segment} of $Q$ is an interval $Q'$ such that $Q'$ has a unique
source $x$ and a unique sink $y$, and is denoted by $[x,y]$.
The set of all segments of $Q$ is denoted by $\Seg(Q)$.
\end{dfn}

We regard a category $\calC$ {as a quiver}
$U(\calC)\coloneqq (\calC_0, \calC_1, \dom, \cod)$
(called the {\em underlying quiver} of $\calC$)
with a structure given by
the family $(\id_x)_{x \in \calC_0}$ of identities and the composition of $\calC$,
where $\calC_0$ (resp.\ $\calC_1$) is the class of objects (resp.\ morphisms)
and $\dom\ (\text{resp.\ }\cod) \colon \calC_1 \to \calC_0$ is a map
sending $f \colon X \to Y$ in $\calC_1$ to the domain $X$
(resp.\ codomain $Y$) of $f$.

Then a functor $F \colon \calC \to \calC'$ between categories
is given by a quiver morphism $F \colon U(\calC) \to U(\calC')$ satisfying
the axioms of a functor, which is called the {\em underlying quiver morphism} of $F$
and is denoted by $U(F)$ (actually we have $F = U(F)$).

For a quiver $Q$, we denote by $\com_Q$ the ideal of the category $\k[Q]$ generated by all commutativity relations in $Q$.

We are now in a position to state the definition of a compression system
in quiver language.

\begin{dfn}
\label{dfn:assignment}
Let $Q$ be an acyclic finite quiver without multiple arrows, and set
$A\coloneqq \k[Q]/\com_Q$.
A \emph{compression system} for $A$ is a family $\xi\coloneqq (\xi_I)_{I \in \bbI}$ of quiver morphisms $\xi_I \colon Q_I^{\xi} \to U(A)$ from
a connected finite quiver $Q_I^{\xi}$
satisfying the following two conditions for each $I \in \bbI(Q)$:

\begin{enumerate}
\item
$\xi_I$ factors through the inclusion morphism $U(A_I) \hookrightarrow U(A)$ of quivers, where $A_I\coloneqq  \k[I]/\com_I$; and
\item
{the image} $\xi_I((Q_I^{\xi})_0)$ of vertices contains $\src(I) \cup \snk(I)$.
\end{enumerate}
{Let $I\in\bbI$, and let}
{$\k[\xi_I]\colon\k[Q_I^\xi]\to A$ be the unique linear extension of $\xi_I$.}
Then we set $B_I\coloneqq \k[Q_I^{\xi}]/\Ker \k[\xi_I]$.
Then $\k[\xi_I]$ induces a functor $\widetilde{\xi_I} \colon B_I \to A$.
The restriction functor
$R^\xi_I \colon \mod A \to \mod B_I$ is defined by
sending $M$ to $M \circ \widetilde{\xi_I}$ for all $M \in \mod A$.
The functor $R^\xi_I$ is simply denoted by $R_I$
{if no confusion can arise}.
\end{dfn}

Note that the definition above extends to $A=\k[Q]/R$ for any ideal $R$ of $\k[Q]$, with $A_I$ taken to be the full subcategory of $A$ on the vertex set $I_0$.
{The results in the main text use the poset version in \cref{dfn:comp-sys-simplified-ver}, where the maps are induced by order-preserving maps.
For such systems, the additional property of covering segments is discussed in \cref{rmk:cond3=ess-cov}.}

{We now relate the quiver and poset formulations.}
Recall that $H(\bfP)$ is the Hasse quiver of $\bfP$ defined in \cref{dfn:ss-intervals}.

\begin{rmk}
\label{rmk:ACQ=FP}
For any finite acyclic quiver $Q$, let
$O(Q)\coloneqq (Q_0, \preceq)$ be the poset on $Q_0$,
{where, for $x,y\in Q_0$, we have $x\preceq y$ if and only if there exists a path $x\leadsto y$.}
There is a canonical isomorphism
$\k[Q]/\com_Q \iso \k[O(Q)]$.

Let $\ACQ$ be the category whose objects are finite acyclic quivers in which no arrow is parallel to any other path, and whose morphisms are quiver morphisms,
and $\FP$ the category of all finite posets and order-preserving maps between them.
Then
$$
O(H(\bfP)) = \bfP, \text{ and } H(O(Q)) \iso Q
$$
for all $\bfP \in \FP_0$ and $Q \in \ACQ_0$.
Note that if $Q = H(\bfP)$ for some $\bfP \in \FP_0$,
then $O(Q)=\bfP$ and hence
$\k[\bfP]=\k[O(Q)] \iso \k[Q]/\com_Q$.
We freely pass between finite posets and their Hasse quivers
using these correspondences.
These correspondences do not identify the two categories,
since an order-preserving map need not induce a quiver morphism
between the corresponding Hasse quivers.
More precisely, we have a canonical bijection
$$
\FP(O(Q), O(Q')) \iso \Qvr(Q, U(O(Q'))),
$$
for all $Q, Q' \in \ACQ_0$,
where $\Qvr$ denotes the category of finite quivers and quiver morphisms between them. 
\end{rmk}

\begin{rmk}
Let $I$ be a full subposet of $\bfP$, set $Q\coloneqq H(\bfP)$, and let
$Q'$ be the full subquiver of $Q$ with $Q'_0 = I$.
Then clearly $I$ is convex in $\bfP$ if and only if $Q'$ is convex in $Q$.
For connectedness, consider the following conditions:
\begin{enumerate}
\item
$I$ is connected as a poset.
\item
The Hasse quiver $H(I)$ of $I$ is connected as a quiver.
\item
$Q'$ is connected as a quiver.
\end{enumerate}
Then (1) and (2) are equivalent
because for any $x, y \in I$, we have $x \le y$ if and only if
there exists a path $x \leadsto y$ in $H(I)$ {by the definition of the Hasse quiver}.

Moreover, (3) implies (1) because for any $x, y \in I$,
if there exists a path $x \leadsto y$ in $Q'$, then $x \le y$ in $I$.

However, (1) does not imply (3).
For example, {consider the case where}
$\bfP = \{1 < 2 < 3\}$ and $ I\coloneqq \{1, 3\}$.
In this case, $Q = (1 \to 2 \to 3)$, $H(I) = (1\to 3)$ and $Q'= (1\ \ 3)$.
Thus $I$ is connected, but $Q'$ is not connected.

Nevertheless, if $I$ is convex in $\bfP$
(or equivalently $Q'$ is convex in $Q$),
then (1) implies (3), and
all the conditions above are equivalent.
Indeed,
{if $x\le y$ in $I$, there exists a path $p\colon x\leadsto y$ in $Q$. Since $x,y\in Q^{\prime}_0$ and $Q^{\prime}$ is convex, $p$ is a path in $Q^{\prime}$.}

Therefore, $I$ is an interval of $\bfP$ if and only if
$Q'$ is an interval of $Q$.
\end{rmk}

\begin{rmk}
\label{rmk:simplified-ver}
Let $\bfP \in \FP$ and set $Q\coloneqq H(\bfP)$.
Thus we have $Q_0 = \bfP$, and {we identify $\k[\bfP]$ with $\k[Q]/\com_Q$}.
Then

(1) The coset of each path $p \colon x \leadsto y$ in $Q$
is identified with the morphism $p_{y,x}$ in $\bfP$.

(2) Since $A = \k[\bfP]$ is isomorphic to $\k[Q]/\com_Q$, the category $\mod A$ of persistence modules is isomorphic to the category $\rep_\k(Q,\linebreak[3]\com_Q)$ of $\k$-representations of the bound quiver $(Q, \com_Q)$. We usually identify these two categories.

(3) In \cref{dfn:assignment}, if we restrict ourselves to the case where 
$\xi_I$ factors through $U(O(I)) \to U(A)$
(note that $O(I)$ is a subcategory of $\k[O(I)] \iso A_I$), and
$Q^\xi_I$ is an acyclic connected finite quiver,
then \cref{dfn:comp-sys-simplified-ver} is obtained.
Note here that $I^\xi\coloneqq O(Q_I^{\xi})$ (\cref{rmk:ACQ=FP})
is a finite connected poset,
and that $\Ker \k[\xi_I] = \com_{Q_I^{\xi}}$, and hence
in this case $B_I = \k[Q_I^{\xi}]/\com_{Q_I^{\xi}} = \k[I^\xi]$.
\end{rmk}

\begin{dfn}
Let $Q, Q'$ be quivers.
Then the product quiver $T\coloneqq Q \times Q'$ is defined as follows.

$T_0\coloneqq Q_0 \times Q'_0$,
$T_1\coloneqq (Q_1 \times Q'_0) \sqcup (Q_0 \times Q'_1)$.
For any $a \colon x \to y$ in $Q_1$ and $x' \in Q'_0$, we have
$(a,x') \colon (x,x') \to (y,x')$, and for any $x \in Q_0$ and
$a' \colon x' \to y'$ in $Q'_1$, we have
$(x, a') \colon (x,x') \to (x,y')$.
\end{dfn}

\begin{dfn}
\label{dfn:An-quiver}
{A quiver of \emph{Dynkin type} $\bbA_n$ ($n\ge1$) has underlying graph a path on $n$ vertices, with an arbitrary orientation. It corresponds to a zigzag poset.}
\end{dfn}

\begin{dfn}
{For $d\ge2$, one can consider products of arbitrarily oriented quivers of Dynkin types $\bbA_{n_1},\dots,\bbA_{n_d}$ ($n_1,\dots,n_d\ge1$). With full commutativity relations, the resulting bound quiver corresponds to the product poset of $d$ zigzag posets. Throughout this paper, a \emph{$d$D-grid} means the product of $d$ finite totally ordered sets, as in \cref{sec2}.}
\end{dfn}

\section{Formal additive hulls}
\label[appendix]{sec:form-add-hull}
{We give the construction of formal additive hulls used in \cref{sec:Essential cover relative to compression systems}. The extension of functors to these hulls is given there in \cref{prp:univ-formal-add-hull}.}

{Formal additive hulls describe morphisms between finite direct sums of projective modules as matrices and allow empty matrices, so that special and general cases can be treated uniformly.}

\begin{dfn}
\label{dfn:formal-add-hull}
For each linear category $B$, a linear category $\Ds B$, called the
\emph{formal additive hull} of $B$, is defined as follows:

{\bf Objects.} The set of objects is given by
$$
(\Ds B)_0\coloneqq \{(x_i)_{i \in [l]} = (x_1,\dots, x_l) \mid x_1, \dots, x_l \in B_0,
\, l \ge 0\}.
$$
Note that if $l = 0$ above, then $[l] = \varnothing$, and $(x_i)_{i \in [l]}$
is an empty sequence $()$.
For each $x = (x_i)_{i \in [l]} \in (\Ds B)_0$, we set $|x|\coloneqq l$, and call it
the \emph{size} of $x$.

\medskip
{\bf Morphisms.} For any $x, y \in (\Ds B)_0$ with $x = (x_i)_{i \in [l]},\, y = (y_j)_{j \in [m]}${,}
the set of morphisms from $x$ to $y$ is defined by setting
$$
(\Ds B)(x,y)\coloneqq \big\{\bmat{\al_{ji}}_{(j,i) \in [m]\times [l]} \mid \al_{ji} \in B(x_i, y_j) \text{ for all }(j,i) \in [m]\times [l]\big\},
$$
where $\bmat{\al_{ji}}_{(j,i) \in [m]\times [l]}$ is a matrix of size $(m,l)$,
which is defined to be the triple $(m, l, (\al_{ji})_{(j,i) \in [m]\times [l]})$ of
integers $l, m \ge 0$ and a family of morphisms $\al_{ji} \in B(x_i, y_j)$.
Note that if $l = 0$, then $x = ()$, and we have
\begin{equation}
\label{eq:emptymat-to}
(\Ds B)((), y) = \{\sfJ_{m,0}\},
\end{equation}
where we set $\sfJ_{m,0}\coloneqq (m, 0, ())$;
if $m = 0$, then $y = ()$, and we have
\begin{equation}
\label{eq:to-emptymat}
(\Ds B)(x, ()) = \{\sfJ_{0,l}\},
\end{equation}
where we set $\sfJ_{0,l}\coloneqq (0, l, ())$.
In particular, we have $(\Ds B)((), ()) = \{\sfJ_{0,0}\}$,
where $\sfJ_{0,0} = (0,0,())$.
The matrices $\sfJ_{m,0}, \sfJ_{0,l}, \sfJ_{0,0}$ are called
the \emph{empty matrices} of size
$(m, 0),\, (0, l),\, (0,0)$, respectively.
{We give $(\Ds B)(x,y)$ its usual vector space structure}
by the usual addition {and} scalar multiplication of matrices.
In particular, if $l=0$ or $m=0$,
then $(\Ds B)(x, y)$ becomes a trivial vector space.

\medskip
{\bf Composition.} For any $x, y, z \in (\Ds B)_0$
with $x = (x_i)_{i \in [l]},\, y = (y_j)_{j \in [m]},\, z = (z_k)_{k \in [n]}$,
the composition
$$
(\Ds B)(y,z) \times (\Ds B)(x,y) \to (\Ds B)(x,z),\ 
(\be, \al) \mapsto \be \cdot \al
$$
is defined by the usual matrix multiplication
$$
\bmat{\be_{kj}}_{(k,j)\in [n]\times [m]} \cdot \bmat{\al_{ji}}_{(j,i)\in [m]\times [l]}
\coloneqq \bmat{\sum_{j\in [m]} \be_{kj}\al_{ji}}_{(k,i)\in [n]\times [l]}
$$
for all $\al = \bmat{\al_{ji}}_{(j,i)\in [m]\times [l]}$
and $\be = \bmat{\be_{kj}}_{(k,j)\in [n]\times [m]}$.
In particular, if $l=0$, then $\be \cdot \sfJ_{m,0} = \sfJ_{n,0}$;
if $m=0$, then $\sfJ_{n,0}\cdot \sfJ_{0,l} = {(n,l,}(0)_{(k,i)\in [n]\times [l]}) {\eqqcolon 0_{n,l}}$;
and if $n=0$, then $\sfJ_{0,m}\cdot \al = \sfJ_{0,l}$.
Thus
if morphisms $\be,\, \al$ {have sizes $(k,p)$ and $(q,l)$, respectively, with $k,l,p,q\ge0$,}
and the composite $\be\cdot \al$ is defined, then $p = q$, and the size of $\be \cdot \al$ is $(k,l)$ as in the case of usual matrix multiplication.

{As is easily seen}, $\Ds B$ is a linear category.
Note that equalities \eqref{eq:emptymat-to} and
\eqref{eq:to-emptymat} show that $()$ is a zero object in $\Ds B$, and hence
we may write $() = 0$ in $\Ds B$.
Moreover, we have
$$
\begin{aligned}
&(x_i)_{i \in [m]} \iso (x_1) \ds \cdots \ds (x_m),\\
&(x_i)_{i \in [m]} \ds (y_j)_{j \in [n]} \iso {(x_1,\dots,x_m,y_1,\dots,y_n)}, \text{ and}\\
&(x_1) \ds \cdots \ds (x_m) \iso (x_1 \ds \cdots \ds x_m)\text{ if $x_1 \ds \cdots \ds x_m$ exists in $B$}
\end{aligned}
$$
for all $x_1,\dots, x_m, y_1,\dots, y_n \in B_0$.
Thus $\Ds B$ turns out to be an additive category.

We regard $B$ as a full subcategory of $\Ds B$ by the embedding
$\si_B \colon (f \colon x \to y) \mapsto (\bmat{f}\colon (x) \to (y))$ for all morphisms $f$ in $B$.
{Throughout the paper}, we will frequently consider the case
where $B = \k[S]$ for a finite poset $S$.
\end{dfn}

\begin{exm}
\label{exm:Ds-k}
Regard the field $\k$ as a linear category with only one object $\ast$
having the set $\k$ of morphisms, where the composition is given by the multiplication of $\k$.
Then we can define an isomorphism $\ph$ of linear categories from $\Ds \k$ to the
full subcategory $\free \k$ of $\mod \k$ consisting of $\k^n$ with $n \ge 0$ as follows:

For each $n \ge 0$, we set $\ast^n\coloneqq (\overbrace{\ast, \dots, \ast}^n)$,
and define $\ph(\ast^n)\coloneqq \k^n$.
For each $l, m \ge 0$, and a morphism $\al = \bmat{\al_{j,i}}_{(j,i) \in [m]\times [l]} \colon \ast^l \to \ast^m$, we set
{$\ph(\al)$ to be the linear map $\k^l\to\k^m$ given by left multiplication by $\al$, which is a matrix over $\k$.}
In this case, the empty matrices $\sfJ_{m,0}$ and $\sfJ_{0,l}$ are sent by $\ph$ to
the zero maps $0 \to {\k^m}$ and ${\k^l} \to 0$, respectively.

Then the composite $\ph' \colon \Ds \k \ya{\ph} \free \k \hookrightarrow \mod \k$
turns out to be an equivalence.
\end{exm}

\begin{rmk}
\label{rmk:cond3=ess-cov}
{We say that a compression system $\xi$ \emph{covers segments} if, for every $I=[x,y]\in\Seg(\bfP)$,
there are $u,v\in I^\xi$ with $u\le v$, $\xi_I(u)=x$, and $\xi_I(v)=y$.
This means that $\xi_I$ covers $p_{y,x}$ in the sense of \cref{dfn:covers}.

For $I=[x,y]$, \cref{prp:seg-rank} gives $\rank_I^\tot M=\rank M_{y,x}$.
Thus $p_{y,x}$, regarded as the matrix
\[
\left[\begin{array}{c|c}
\sfJ_{0,1} & \sfJ_{0,0}\\
\hline
p_{y,x} & \sfJ_{1,0}
\end{array}\right]\colon (x)\ds()\to()\ds(y)
\]
in $\Ds\k[I]$, is an $I$-rank matrix under $\tot$.
If $\xi$ covers segments, the map $\xi'\colon I^\xi\to I$ in \cref{ntn:tot} covers this matrix, since $\k[\xi'](p_{v,u})=p_{y,x}$.
Together with $\xi_I=\tot_I\circ\xi'$, this shows that $\xi_I$ essentially covers $I$ relative to $\tot$.}
\end{rmk}

\end{appendices}

\backmatter

\bmhead*{Acknowledgements}
The authors would like to acknowledge Ond\v{r}ej Draganov for his assistance in refining the code for {computing the interval rank invariant and the interval replacement} under the source-sink compression system, as well as for providing valuable suggestions to optimize and accelerate {the code on GitHub}.
{The authors also thank the anonymous referee for a careful reading of the manuscript and helpful comments and suggestions.}

\section*{Declarations}

\begin{itemize}
\item Funding: Hideto Asashiba was partially supported by JSPS Grant-in-Aid for Scientific Research (C) 18K03207, JSPS Grant-in-Aid for Transformative Research Areas (A) (22A201), and by Osaka Central Advanced Mathematical Institute
(MEXT Promotion of Distinctive Joint Research Center Program JPMXP0723833165). Enhao Liu was supported by JST SPRING, Grant Number JPMJSP2110.
\item Competing interests: the authors declare that they have no financial conflicts of interest with regard to the content of this article.
\end{itemize}

\removelastskip

\bibliography{sn-bibliography}%

\end{document}